\documentclass[a4paper,oneside,10pt]{article}
\usepackage{amsmath,amssymb}
\usepackage[hidelinks]{hyperref}
\usepackage{color}
\newtheorem{theorem}{Theorem}[section]
\newtheorem{corollary}[theorem]{Corollary}
\newtheorem{definition}[theorem]{Definition}
\newtheorem{lemma}[theorem]{Lemma}
\newtheorem{assumption}[theorem]{Assumption}
\newtheorem{proposition}[theorem]{Proposition}
\newtheorem{remark}[theorem]{Remark}
\newenvironment{proof}[1][Proof]{\noindent \textbf{#1.} }{\  $\Box$}
\numberwithin{equation}{section}

\makeatletter
\newcommand{\equationlabelalias}[2]{%
  \@ifundefined{r@#2}{}{%
    \protected@write\@auxout{}{%
      \string\newlabel{#1}{\csname r@#2\endcsname}}}%
}
\makeatother
\begin{document}

\title{Vector- and operator-valued backward stochastic equations with finite-variation drivers and a maximum principle for singular stochastic control in infinite dimensions}
\author{Ying Hu\thanks{IRMAR, Universit\'e
        Rennes 1, Rennes, France.\ ying.hu@univ-rennes1.fr. This author's
        research is partially supported by Lebesgue Center of Mathematics
        ``Investissements d'avenir'' Program (ANR-11-LABX-0020-01), by ANR CAESARS
        (ANR-15-CE05-0024) and by ANR MFG (ANR-16-CE40-0015-01).}
    \and Guomin Liu\thanks{School of Mathematical Sciences, Nankai University, Tianjin,
China.\ gmliu@nankai.edu.cn. Research supported by National
Natural Science Foundation of China (12571479, 12201315).}
\and Shanjian Tang\thanks{School of Mathematical Sciences, Fudan University, Shanghai, China. sjtang@fudan.edu.cn. Research supported by National Key R\&D
	Program of China (2018YFA0703903) and National Natural Science Foundation
	of China (11631004).}
}
\date{}
\maketitle

\begin{abstract}
We study a mixed regular--singular control problem for stochastic
evolution equations in a Hilbert space with possibly unbounded random linear
operators, a nonconvex regular-control domain, and a state-dependent
singular coefficient. The singular control is an adapted nondecreasing
c\`adl\`ag process with a possibly unbounded terminal value. We establish
well-posedness and weighted moment estimates for the forward and backward
equations and characterize the second-order adjoint through a conditionally
expected operator-valued backward stochastic integral equation. An
It\^o-type formula for the quadratic form associated with this adjoint,
combined with spike and convex variations, yields a Hamiltonian condition
with a second-order correction for the regular control, together with
nonnegativity and a contact condition for the optional singular Hamiltonian.
We illustrate the maximum principle with a stochastic heat equation
subject to a three-valued regular control and nonnegative singular inputs.



\end{abstract}

\vspace{-10pt}

\section{Introduction}

We consider a mixed regular--singular stochastic optimal control problem
governed by stochastic evolution equations (SEEs) in a Hilbert space.
The state equation takes the form
\begin{align}
	x(t)
	&=
	x_0
	+
	\int_0^t
	\big[
	A(s)x(s)+a(s,x(s),u(s))
	\big]\,ds
	\notag\\
	&\quad
	+
	\int_0^t
	\big[
	B(s)x(s)+b(s,x(s),u(s))
	\big]\,dW(s)
	+
	\int_{(0,t]}
	\gamma(s,x(s-))\,d\xi(s),
	\quad t\in[0,T],
	\label{eq:introduction-state}
\end{align}
where \(A(\cdot)\) and \(B(\cdot)\) are possibly unbounded random linear
operators, \(u(\cdot)\) is a regular control with values in a separable
metric space \(U\), and \(\xi(\cdot)\) is an adapted nondecreasing
c\`adl\`ag process. The control domain \(U\) need not be convex.
The performance criterion is
\begin{equation}
	J(u,\xi)
	=
	\mathbb E\left[
	h(x(T))
	+
	\int_0^T k(t,x(t),u(t))\,dt
	+
	\int_{(0,T]}
	g(t,x(t-))\,d\xi(t)
	\right].
	\label{eq:introduction-cost}
\end{equation}

Singular controls act through random measures and arise, for example,
in monotone-follower problems \cite{BSW80,KS84}, transaction-cost and
limit-order-book models \cite{DN90,HorstNaujokat2014}, and singular or reflected SPDE models 
\cite{OksendalSulemZhang2014,AgramHilbertOksendal2020}.
In finite dimensions, such problems have been studied using dynamic
programming and maximum principles; see \cite{BSW80,CMR85,KS84,DN90,HS94,HS95,HS95-2,
CadenillasHaussmann1994,BahlaliMezerdi2005}. 

For regular stochastic control problems, Peng
\cite{Peng1990} established a general maximum principle allowing a
nonconvex control domain and a control-dependent diffusion coefficient. A spike variation produces a first state variation of order
\(O(\rho^{1/2})\), so the cost expansion must retain the second state
variation. The resulting maximum principle involves a second-order
adjoint process and a quadratic Hamiltonian correction.

For mixed regular--singular control with an additive term
\(G(t)\,d\xi(t)\), Bahlali and Mezerdi \cite{BahlaliMezerdi2005}
combined regular spike variations with singular convex perturbations,
allowing a nonconvex regular-control domain. The relaxed formulations of
Bahlali, Djehiche, and Mezerdi \cite{Bahlali2007} and Andersson
\cite{Andersson2009} allow measure-valued regular controls, with
uncontrolled and control-dependent diffusion, respectively. Further
developments include \cite{JiXue09,SongWu2022,Zhang2021}.

State-dependent singular coefficients have been treated by different
methods. Dufour and Miller \cite{DufourMiller2006} used a time transformation
to derive a maximum principle for a nonlinear singular system with a
state-independent singular direction. Their subsequent work
\cite{DufourMiller2007} allowed the coefficient multiplying the singular
control to depend on the state. Neither formulation contains an independent
regular control in the drift or diffusion coefficient. Kushner
\cite{Kushner2013} developed a numerical approximation method for stochastic
singular control problems with state-dependent singular controls, again
without a regular control. Wang, Zhao, Wang, Song, and Ma
\cite{WangZhaoWangSongMa2018} studied a one-dimensional partial-information
regular--singular jump-diffusion with a state-dependent singular coefficient.
They obtained a necessary maximum principle through first-order
directional variations, assuming an open and convex regular-control
domain.

Early maximum principles for distributed parameter systems and semilinear
stochastic evolution equations were established by Bensoussan
\cite{Bensoussan1983} and Hu and Peng \cite{HuPeng1990}. Subsequent work
established general maximum principles for regular controls with
nonconvex control domains in infinite dimensions; see, for example,
\cite{DM13,FHT13,LZ14,SW21,LT21}. A central issue in this setting is how to characterize the
operator-valued second-order adjoint process. Existing approaches use
stochastic bilinear forms, transposition solutions, or operator-valued
backward equations. Liu and Tang
\cite{LT21} introduced a conditionally expected operator-valued backward
stochastic integral equation (BSIE) that characterizes the second-order
adjoint process dynamically without treating the nonseparable operator space
\(\mathfrak L(H)\) as a Hilbert space. These works concern regular controls
and do not include a singular finite-variation component.

Infinite-dimensional singular control has also been studied under different
structural assumptions. {\O}ksendal, Sulem, and Zhang
\cite{OksendalSulemZhang2014} considered SPDEs in which the singular measure
is multiplied by the state-dependent coefficient
\(\lambda(t,x,Y(t,x))\). They derived necessary and sufficient maximum principles, the latter
under a concavity condition, and related the adjoint equation to a
reflected backward SPDE. Agram, Hilbert, and {\O}ksendal
\cite{AgramHilbertOksendal2020} derived necessary and sufficient maximum principles for SPDEs with
space-mean dynamics. Their adjoint equation is a reflected backward SPDE
with space-mean dependence. Federico, Ferrari, Riedel, and
R\"ockner \cite{FedericoEtAl2021} formulated a class of linear singular control
problems on partially ordered infinite-dimensional spaces and derived
necessary and sufficient first-order optimality conditions by semigroup and
convex-analytic methods. 
None of these formulations includes an independent
regular control.

The state equation \eqref{eq:introduction-state} combines a nonconvex
regular-control domain and control-dependent diffusion with a nonlinear
state-dependent singular coefficient in an infinite-dimensional setting.
The random linear operators \(A(\cdot)\) and \(B(\cdot)\)
may also be unbounded as operators on \(H\).
The first- and second-order expansions of the singular coefficient
give rise to finite-variation terms in the variational equations; see \eqref{eq:first-variation-regular} and
\eqref{eq:second-variation-regular}. Accordingly, the first-order adjoint
equation is a backward stochastic evolution equation with a finite-variation
driver, whereas the second-order adjoint equation is an operator-valued BSIE
with a finite-variation term.

The first analytic issue is the well-posedness of forward and backward
stochastic evolution equations driven by \(dt\), \(dW\), and a nondecreasing
c\`adl\`ag process. Hu and Peng \cite{HuPeng1991} studied adapted solutions of semilinear
backward stochastic evolution equations. In finite dimensions,
Pardoux and Zhang \cite{PardouxZhang1998} treated generalized BSDEs driven
by a continuous nondecreasing process, and Li and Tang
\cite{LiTang2015} studied a recursive optimal control problem governed by
such BSDEs. In the present setting, the driver \(\xi\) may have jumps.
We use exponential weights with separate parameters for time and the
Stieltjes driver to estimate the \(dt\)- and \(d\xi\)-integrals and
the jump terms.

The second issue is that \(\mathfrak L(H)\) is generally nonseparable and
not a Hilbert space. We extend the conditionally expected BSIE
formulation of \cite{LT21} to include a c\`adl\`ag finite-variation term.
A contraction argument in a weighted space, together with aggregation of
scalar conditional expectations, gives a weakly progressively measurable
operator-valued solution. We also establish an It\^o-type formula for
the quadratic form associated with this solution along a spike-generated
state variation.

A general nondecreasing c\`adl\`ag driver need not satisfy
\(\xi(t_0+\rho)-\xi(t_0)=O(\rho)\).
The moving-spike estimates therefore rely on right continuity and
uniform bounds for the homogeneous flow.
We combine regular spikes with convex singular perturbations.
Vector-valued duality eliminates the linear state-variation terms from
the cost expansion, and the operator-valued It\^o formula yields the
quadratic diffusion correction in the regular Hamiltonian. Optional
projection and stopping-time tests then yield nonnegativity and the contact
condition from the singular variational inequality.

We also apply the maximum principle to a stochastic heat equation
with a three-valued regular control and nonnegative singular inputs.
We derive a threshold rule for the optimal regular control in terms
of the adjoint process and show that any optimal singular input
is concentrated at the terminal time. The size of this terminal input
is given by an explicit threshold formula.

Section~2 introduces the notation and operator-valued conditional
expectation. Section~3 establishes well-posedness and estimates for
SEEs and BSEEs. Section~4 treats operator-valued BSIEs and their
It\^o-type formula. Section~5 formulates the mixed regular--singular
stochastic control problem and derives the maximum principle. The Appendix
contains the bounded-driver constructions and the deferred technical proofs.

\section{Preliminaries}\label{preliminary}
Let \((\Omega,\mathcal F,\mathbb P)\) be a probability space equipped with a
filtration \(\mathbb F=\{\mathcal F_t\}_{0\le t\le T}\) satisfying the usual
conditions. We denote by \(\operatorname{Prog}\) the progressive
\(\sigma\)-algebra on \(\Omega\times[0,T]\). A subset of
\(\Omega\times[0,T]\) is called evanescent if it is contained in
\(N\times[0,T]\) for some \(N\in\mathcal F\) with \(\mathbb P(N)=0\).
For Banach spaces \(X\) and \(Y\),
\(\mathfrak L(X;Y)\) denotes the space of bounded linear operators from \(X\)
to \(Y\), and we write \(\mathfrak L(X)=\mathfrak L(X;X)\).

Let \(E\) be a separable Hilbert space,
\(\mathcal G\subset\mathcal F\) a sub-\(\sigma\)-algebra, and \(\alpha\ge1\).
The space \(L^\alpha(\mathcal G;E)\) consists of
\(\mathcal G\)-measurable random variables with norm
\[
	\|\eta\|_{L^\alpha(\mathcal G;E)}
:=\bigl(\mathbb E\|\eta\|_E^\alpha\bigr)^{1/\alpha}.\]
The spaces 
\(L^\alpha_{\mathbb F}(0,T;E)\),
\(L^{1,\alpha}_{\mathbb F}(0,T;E)\), and
\(L^{2,\alpha}_{\mathbb F}(0,T;E)\) consist of progressively measurable
processes with the respective norms
\begin{align*}
&\|\phi\|_{L^\alpha_{\mathbb F}(0,T;E)}
:=\left(\mathbb E\int_0^T\|\phi(t)\|_E^\alpha\,dt\right)^{1/\alpha},\quad
\|\phi\|_{L^{1,\alpha}_{\mathbb F}(0,T;E)}
:=\left[\mathbb E\left(\int_0^T\|\phi(t)\|_E\,dt\right)^\alpha\right]^{1/\alpha},\\
&\|\phi\|_{L^{2,\alpha}_{\mathbb F}(0,T;E)}
:=\left[\mathbb E\left(\int_0^T\|\phi(t)\|_E^2\,dt\right)^{\alpha/2}\right]^{1/\alpha}.
\end{align*}
The space
\(\mathbb D^\alpha_{\mathbb F}(0,T;E)\) consists of adapted c\`adl\`ag processes
with norm
\[
	\|\phi\|_{\mathbb D^\alpha_{\mathbb F}(0,T;E)}
	:=\left(\mathbb E\sup_{0\le t\le T}\|\phi(t)\|_E^\alpha\right)^{1/\alpha}.
\]
Let \(F\) be a real-valued adapted c\`adl\`ag finite-variation process
on \([0,T]\), and let \(|F|_v\) denote its total variation process.
For \(q\in\{1,2\}\) and \(\alpha\ge1\), the space
\(L^{q,\alpha}_{\mathbb F,F}(0,T;E)\) consists of progressively measurable
processes with finite norm
\[
 \|\phi\|_{L^{q,\alpha}_{\mathbb F,F}(0,T;E)}
 :=\left[\mathbb E\left(
   \int_{(0,T]}\|\phi(t)\|_E^q\,d|F|_v(t)
 \right)^{\alpha/q}\right]^{1/\alpha}.
\]
We write \(L^2_{\mathbb F,F}=L^{2,2}_{\mathbb F,F}\).
When \(F\) is nondecreasing, \(d|F|_v=dF\). When
\(E=\mathbb R\), we omit \(E\) from the notation.
On a subinterval \([r,T]\), the Lebesgue integrals are taken over
\([r,T]\), the Stieltjes integrals over \((r,T]\), and the maximum
in the \(\mathbb D^\alpha_{\mathbb F}(r,T;E)\)-norm over \([r,T]\).

Let \(H\) be a separable real Hilbert space. An \(\mathfrak L(H)\)-valued
random variable \(Z\) is weakly \(\mathcal G\)-measurable, where
\(\mathcal G\subset\mathcal F\), if
\(\omega\mapsto\langle Z(\omega)u,v\rangle_H\) is
\(\mathcal G\)-measurable for every \(u,v\in H\). Weak progressive
measurability of an operator-valued process \(Q\) means that
\((\omega,t)\mapsto\langle Q(\omega,t)u,v\rangle_H\) is progressively
measurable for every \(u,v\in H\). Equivalently, weak (progressive) measurability means (progressive)
measurability with respect to the \(\sigma\)-algebra
\(\mathcal B_w(\mathfrak L(H))\) generated by
\(Q\mapsto\langle Qu,v\rangle_H\), \(u,v\in H\). For \(\alpha\ge1\), we write $L_w^\alpha(\mathcal G;\mathfrak L(H))$ for the space of $\mathcal G$-measurable mappings with norm
\[
\|Z\|_{L_w^\alpha(\mathcal G;\mathfrak L(H))}
:=\bigl(\mathbb E\|Z\|_{\mathfrak L(H)}^\alpha\bigr)^{1/\alpha}.
\]
For \(\alpha\ge1\), we denote by
\(L^\alpha_{\mathbb F,w}(0,T;\mathfrak L(H))\) the space of weakly
progressively measurable processes \(Q:\Omega\times[0,T]\to\mathfrak L(H)\)
such that
\[
\mathbb E\int_0^T\|Q(t)\|_{\mathfrak L(H)}^\alpha\,dt<\infty.
\]
If \(F\) is a real-valued adapted c\`adl\`ag finite-variation process, then
\(L^\alpha_{\mathbb F,F,w}(0,T;\mathfrak L(H))\) consists of weakly
progressively measurable processes satisfying
\[
\mathbb E\int_{(0,T]}\|Q(t)\|_{\mathfrak L(H)}^\alpha\,d|F|_v(t)<\infty.
\]
For \(q\in\{1,2\}\) and \(\alpha\ge1\), the spaces
\(L^{q,\alpha}_{\mathbb F,w}(0,T;\mathfrak L(H))\) and
\(L^{q,\alpha}_{\mathbb F,F,w}(0,T;\mathfrak L(H))\) consist of weakly
progressively measurable processes with the respective finite norms
\begin{align*}
\|Q\|_{L^{q,\alpha}_{\mathbb F,w}(0,T;\mathfrak L(H))}
&:=\left[\mathbb E\left(
\int_0^T\|Q(t)\|_{\mathfrak L(H)}^q\,dt
\right)^{\alpha/q}\right]^{1/\alpha},\\
\|Q\|_{L^{q,\alpha}_{\mathbb F,F,w}(0,T;\mathfrak L(H))}
&:=\left[\mathbb E\left(
\int_{(0,T]}\|Q(t)\|_{\mathfrak L(H)}^q\,d|F|_v(t)
\right)^{\alpha/q}\right]^{1/\alpha}.
\end{align*}

We recall the operator-valued conditional expectation introduced in
\cite[Section~2.1]{LT21}. Let
\(Y:H\times H\to L^1(\mathcal F)\) be a bounded bilinear mapping. A
weakly \(\mathcal G\)-measurable \(\mathfrak L(H)\)-valued random variable
\(Z\) is called a conditional expectation of \(Y\) given
\(\mathcal G\), denoted by \(\mathbb E[Y\mid\mathcal G]\), if, for every
\(u,v\in H\),
\begin{equation}
    \langle Zu,v\rangle_H
    =
    \mathbb E\big[Y(u,v)\mid\mathcal G\big],
    \quad \mathbb P\text{-a.s.}
    \label{eq:operator-valued-conditional-expectation}
\end{equation}
By \cite[Theorem~2.1]{LT21}, such a conditional
expectation exists and is integrable if and only if there is a nonnegative
\(g\in L^1(\mathcal G)\) such that for every
\(u,v\in H\),
\begin{equation}
    \left|\mathbb E\big[Y(u,v)\mid\mathcal G\big]\right|
    \le g\,\|u\|_H\|v\|_H,
    \quad \mathbb P\text{-a.s.}
    \label{eq:operator-valued-conditional-expectation-domination}
\end{equation}
In this case, the conditional expectation is unique up to \(\mathbb P\text{-a.s.}\) equality and
\begin{equation}
    \|\mathbb E[Y\mid\mathcal G]\|_{\mathfrak L(H)}\le g,
    \quad\mathbb P\text{-a.s.}
    \label{eq:operator-valued-conditional-expectation-bound}
\end{equation}
If, in addition, \(g\in L^2(\mathcal G)\), then
\(\mathbb E[Y\mid\mathcal G]\in L_w^2(\mathcal G;\mathfrak L(H))\).

\section{SEEs and BSEEs with finite-variation drivers}
Throughout this section, \((V,H,V^*)\) is a Gelfand triple: \(V\) is a
separable real Hilbert space densely and continuously embedded in the
separable real Hilbert space \(H\), which is identified with its dual. We
write \(\langle\cdot,\cdot\rangle_H\) for the inner product in \(H\) and
\(\langle\cdot,\cdot\rangle_*\) for the duality pairing between \(V^*\) and
\(V\). Let \(W\) be a one-dimensional Brownian motion and let \(\mathbb F\)
be the usual augmentation of its natural filtration.

Let \(\xi\) be a real-valued adapted nondecreasing c\`adl\`ag process such that
\(\xi(0)=0\) and \(\xi(T)<\infty\) \(\mathbb P\text{-a.s.}\) 
We write \(\Delta\xi(t):=\xi(t)-\xi(t-)\),
\(t\in(0,T]\), and denote by \(\xi^c\) its continuous part. Its purely
discontinuous part is
\(\xi^d(t):=\sum_{0<s\le t}\Delta\xi(s), \  t\in[0,T],\)
so that \(\xi=\xi^c+\xi^d\).

We define the Dol\'eans measure
\begin{equation}
\mu^\xi(B)
:=
\mathbb E\int_{(0,T]} I_B(\omega,t)\,d\xi(t),
\quad
B\in\mathcal F\otimes\mathcal B([0,T]).
\label{eq:Doleans-measure-xi}
\end{equation}

For \(\ell\ge0\) and \(\lambda>0\), set
\begin{equation}
    \mathfrak E_{\lambda,\ell}(\xi)
    :=
    \sup_{\tau}
    \left\|
        \mathbb E\left[
            \exp\!\left\{
                \lambda\ell\bigl(\xi(T)-\xi(\tau)\bigr)
            \right\}
            \,\middle|\,
            \mathcal F_\tau
        \right]
    \right\|_{L^\infty(\Omega)},
    \label{eq:unbounded-xi-conditional-exponential-moment}
\end{equation}
where the supremum is taken over all \([0,T]\)-valued stopping times.
We use the convention \(\mathfrak E_{\lambda,0}(\xi)=1\).

Fix \(\kappa>2\). For \(0\le r\le t\le T\), define
\begin{equation}
\begin{aligned}
    \rho_{\beta,\ell}^{r}(t)
    &:={}
    \exp\!\left\{
        -\beta(t-r)-\kappa\ell\bigl(\xi(t)-\xi(r)\bigr)
    \right\},
    \quad
    \theta_{\beta,\ell}^{r}(t)
    &:={}
    \bigl(\rho_{\beta,\ell}^{r}(t)\bigr)^{-1}.
\end{aligned}
\label{eq:unbounded-xi-simple-weights}
\end{equation}
We use \(\beta\) to absorb the terms integrated against \(dt\), and
\(\kappa\ell\) to absorb the multiplicative part of the Stieltjes coefficient.

The following inequalities will be used to absorb the homogeneous jump
terms in the forward and backward energy estimates.

For every \(z\ge0\) and every \(\kappa\ge2\),
\begin{align}
    e^{\kappa z}-1
    &\ge
    \kappa z+\kappa z^2,
    \label{eq:unbounded-xi-exponential-quadratic-lower-bound}
    \\
    e^{\kappa z/2}
    &\ge 1+z,
    \label{eq:unbounded-xi-exponential-linear-lower-bound}\\
    e^{\kappa z}-(1+z)^2
    &\ge
    (\kappa-2)z+(\kappa-1)z^2.
    \label{eq:unbounded-xi-exponential-flow-lower-bound}
\end{align}
These inequalities follow from the elementary bounds
\(e^u\ge1+u+u^2/2\) and \(e^u\ge1+u\), for \(u\ge0\),
using \(\kappa\ge2\); the third inequality is the first with
\(2z+z^2\) subtracted from both sides.
\equationlabelalias{lem:unbounded-xi-exponential-jump-inequalities}{eq:unbounded-xi-exponential-quadratic-lower-bound}

\subsection{Weighted It\^o formulas}

We first present the weighted forward and backward It\^o formulas used in the
subsequent energy estimates.

Let \(F\) be a real-valued adapted c\`adl\`ag finite-variation process on
\([0,T]\). Suppose that
\[
    h_0\in L^2(\mathcal F_0;H),\quad
    v\in L^2_{\mathbb F}(0,T;V),\quad
    v^*\in L^2_{\mathbb F}(0,T;V^*),\quad
    \eta\in L^{1,2}_{\mathbb F}(0,T;H),\quad
    \zeta\in L^{1,2}_{\mathbb F,F}(0,T;H),
\]
and let \(M\) be an \(H\)-valued continuous local martingale with \(M(0)=0\).
We denote by \(\langle M\rangle\) its scalar quadratic variation: the
unique continuous adapted nondecreasing process, starting at zero, such
that \(\|M\|_H^2-\langle M\rangle\) is a local martingale.

We use the following Hilbert-space It\^o formula from \cite[Theorem~1]{GK82}.
\begin{lemma}
\label{Ito-lemma}
Suppose that, for every \(\varphi\in V\),
\begin{align}
    \langle v(t),\varphi\rangle_H
    &=
    \langle h_0,\varphi\rangle_H
    +
    \int_0^t \langle v^*(s),\varphi\rangle_*\,ds
    +
    \int_0^t \langle \eta(s),\varphi\rangle_H\,ds
    +
    \int_{(0,t]}
        \langle \zeta(s),\varphi\rangle_H\,dF(s)
    +
    \langle M(t),\varphi\rangle_H
    \label{eq:FV-Ito-weak-form}
\end{align}
for \(dt\otimes d\mathbb P\)-almost every \((t,\omega)\). Then there
exists an adapted c\`adl\`ag \(H\)-valued process \(h\) such that
\(h=v\) \(dt\otimes d\mathbb P\)-almost everywhere, and outside a
\(\mathbb P\)-null set, \(h(0)=h_0\) and, for every \(t\in[0,T]\),
\begin{align}
    \|h(t)\|_H^2
    &=
    \|h_0\|_H^2
    +
    2\int_0^t \langle v^*(s),v(s)\rangle_*\,ds
    +
    2\int_0^t \langle \eta(s),h(s)\rangle_H\,ds
    +
    \langle M\rangle(t)
    +
    2\int_0^t \langle h(s),dM(s)\rangle_H
    \notag\\
    &\quad
    +
    \int_{(0,t]}
    \Big[
        2\langle h(s-),\zeta(s)\rangle_H
        +
        \|\zeta(s)\|_H^2\Delta F(s)
    \Big] \,dF(s).
    \label{eq:FV-Ito-left-limit}
\end{align}
\end{lemma}

For the forward estimates, we use a decreasing weight. Let
\(K\) be a real-valued adapted nondecreasing c\`adl\`ag process with
\(K(0)=0\), and put \(\rho(t):=e^{-K(t)}\), \(t\in[0,T]\).
\begin{corollary}
\label{cor:weighted-Ito-general}
Under the assumptions of Lemma~\ref{Ito-lemma}, the following identity
holds outside a \(\mathbb P\)-null set for every \(t\in[0,T]\):
\begin{align}
    \rho(t)\|h(t)\|_H^2
    &=
    \|h_0\|_H^2
    +
    2\int_0^t
        \rho(s)\langle v^*(s),v(s)\rangle_*\,ds
    +
    2\int_0^t
        \rho(s)\langle \eta(s),h(s)\rangle_H\,ds
    \notag\\
    &\quad
    +
    \int_0^t \rho(s)\,d\langle M\rangle(s)
    -
    \int_0^t
        \rho(s)\|h(s)\|_H^2\,dK^c(s)
    \notag\\
    &\quad
    +
    \int_{(0,t]}
    \rho(s)
    \Big[
        2\langle h(s-),\zeta(s)\rangle_H
        +
        \|\zeta(s)\|_H^2\Delta F(s)
    \Big] \,dF(s)
    \notag\\
    &\quad
    -
    \sum_{0<s\le t}
    \rho(s)\big(e^{\Delta K(s)}-1\big)\|h(s-)\|_H^2
    +
    2\int_0^t
        \rho(s)\langle h(s),dM(s)\rangle_H.
    \label{eq:weighted-FV-Ito}
\end{align}
\end{corollary}
\begin{proof}
	Set \(Y(t)=\|h(t)\|_H^2\). Lemma~\ref{Ito-lemma} gives
	\begin{equation}
		\Delta Y(s)
		=
		\Big[
		2\langle h(s-),\zeta(s)\rangle_H
		+\|\zeta(s)\|_H^2\Delta F(s)
		\Big]\Delta F(s).
		\label{eq:delat-Y-explicit}
	\end{equation}
	The finite-variation chain rule for \(\rho=e^{-K}\) yields
	\begin{equation}
		d\rho(s)
		=
		-\rho(s)\,dK^c(s)
		-\sum_{0<r\le T}
		\rho(r)\bigl(e^{\Delta K(r)}-1\bigr)\delta_r(ds),
		\label{eq:d-rho-explicit}
	\end{equation}
	and hence
	\begin{equation}
		\Delta\rho(s)
		=-\rho(s)\bigl(e^{\Delta K(s)}-1\bigr).
		\label{eq:delta-rho-explicit}
	\end{equation}
	Since \(\rho\) has finite variation, the product formula gives
	\begin{equation}
		d(\rho Y)(s)
		=
		\rho(s-)\,dY(s)+Y(s-)\,d\rho(s)
		+\sum_{0<r\le T}\Delta\rho(r)\Delta Y(r)\delta_r(ds).
		\label{eq:product-formula-1}
	\end{equation}
	
	Since c\`adl\`ag paths have at most countably many jumps,
the set of jump times has measure zero with respect to
\(ds\), \(d|F^c|_v\), \(d\langle M\rangle\), and \(dK^c\).
Thus left limits can be replaced by current values
	in the continuous terms, including the integral against the
	continuous martingale \(M\). At a jump time, the product formula,
	\eqref{eq:delat-Y-explicit}, and \eqref{eq:delta-rho-explicit} give
	\begin{align*}
		\Delta(\rho Y)(s)
		&=\rho(s)\Delta Y(s)+Y(s-)\Delta\rho(s)\\
		&=\rho(s)
		\Big[
		2\langle h(s-),\zeta(s)\rangle_H
		+\|\zeta(s)\|_H^2\Delta F(s)
		\Big]\Delta F(s)-\rho(s)\bigl(e^{\Delta K(s)}-1\bigr)
		\|h(s-)\|_H^2.
	\end{align*}
	We substitute \eqref{eq:FV-Ito-left-limit} and
	\eqref{eq:d-rho-explicit} into \eqref{eq:product-formula-1}
	and use the preceding identity for \(\Delta(\rho Y)(s)\).
	Integrating over \((0,t]\) then yields \eqref{eq:weighted-FV-Ito},
	since \(\rho(0)=1\) and \(h(0)=h_0\).
\end{proof}

 Let
\[
 h_T\in L^2(\mathcal F_T;H),\quad
 v\in L^2_{\mathbb F}(0,T;V),\quad
 v^*\in L^2_{\mathbb F}(0,T;V^*),\quad
 \zeta\in L^{1,2}_{\mathbb F,F}(0,T;H),
\]
and let \(M\) be an \(H\)-valued continuous local martingale.
Following the argument in the proof of
\cite[Theorem~1]{GK82}, we obtain the following backward It\^o formula.
\begin{lemma}
\label{lem:backward-FV-Ito-M}
Suppose that, for every \(\varphi\in V\),
\begin{align}
    \langle \varphi,v(t)\rangle_H
    &=
    \langle \varphi,h_T\rangle_H
    +
    \int_t^T
        \langle v^*(s),\varphi\rangle_*\,ds
    +
    \int_{(t,T]}
        \langle \zeta(s),\varphi\rangle_H\,dF(s)
    +
    \langle \varphi,M(T)-M(t)\rangle_H
    \label{eq:backward-FV-Ito-weak-form-M}
\end{align}
for \(dt\otimes d\mathbb P\)-almost every \((t,\omega)\). Then there
exists an adapted c\`adl\`ag \(H\)-valued process \(h\) such that
\(h=v\) \(dt\otimes d\mathbb P\)-almost everywhere, and outside a
\(\mathbb P\)-null set, \(h(T)=h_T\) and, for every \(t\in[0,T]\),
\begin{align}
    \|h(t)\|_H^2
   & +
    \int_t^T d\langle M\rangle(s)
    =
    \|h_T\|_H^2
    +
    2\int_t^T
        \langle v^*(s),v(s)\rangle_*\,ds
    \notag\\
    &\quad
    +
    \int_{(t,T]}
    \Big[
        2\langle h(s-),\zeta(s)\rangle_H
        -
        \|\zeta(s)\|_H^2\Delta F(s)
    \Big] \,dF(s)
    +
    2\int_t^T
        \langle h(s),dM(s)\rangle_H.
    \label{eq:backward-FV-Ito-M-equivalent}
\end{align}
\end{lemma}

For the backward energy estimates, we use an increasing weight.
Let \(K\) be a real-valued adapted nondecreasing c\`adl\`ag process with
\(K(0)=0\), and set \(\theta(t):=e^{K(t)}\), \(t\in[0,T]\).
\begin{corollary}
\label{cor:positive-weight-backward-FV-Ito-M}
Under the assumptions of Lemma~\ref{lem:backward-FV-Ito-M}, the following
identity holds outside a \(\mathbb P\)-null set for every \(t\in[0,T]\):
\begin{align}
    &\theta(t)\|h(t)\|_H^2
    +
    \int_t^T \theta(s)\,d\langle M\rangle(s)
    +
    \int_t^T \theta(s)\|h(s)\|_H^2\,dK^c(s)
    +
    \sum_{t<s\le T}
    \theta(s)\big(1-e^{-\Delta K(s)}\big)\|h(s-)\|_H^2
    \notag\\
    &=
    \theta(T)\|h_T\|_H^2
    +
    2\int_t^T
        \theta(s)\langle v^*(s),v(s)\rangle_*\,ds
    \notag\\
    &\quad
    +
    \int_{(t,T]}
    \theta(s)
    \Big[
        2\langle h(s-),\zeta(s)\rangle_H
        -
        \|\zeta(s)\|_H^2\Delta F(s)
    \Big] \,dF(s)
    +
    2\int_t^T
        \theta(s)\langle h(s),dM(s)\rangle_H.
    \label{eq:positive-weight-backward-FV-Ito-M}
\end{align}
\end{corollary}
\begin{proof}
	Set \(Y(t)=\|h(t)\|_H^2\). By
	Lemma~\ref{lem:backward-FV-Ito-M},
	\begin{equation}
		\begin{aligned}
			dY(s)
			&=d\langle M\rangle(s)
			-2\langle v^*(s),v(s)\rangle_*\,ds
			-2\langle h(s-),dM(s)\rangle_H\\
			&\quad-
			\Big[
			2\langle h(s-),\zeta(s)\rangle_H
			-\|\zeta(s)\|_H^2\Delta F(s)
			\Big]\,dF(s).
		\end{aligned}
		\label{eq:dY-explicit-backward}
	\end{equation}
	The finite-variation chain rule for \(\theta=e^K\) gives
	\begin{equation}
		d\theta(s)
		=
		\theta(s)\,dK^c(s)
		+\sum_{0<r\le T}
		\theta(r)\bigl(1-e^{-\Delta K(r)}\bigr)\delta_r(ds).
		\label{eq:d-theta-explicit}
	\end{equation}
	Applying the product formula to \(\theta Y\), we obtain
	\begin{equation}
		d(\theta Y)(s)
		=
		\theta(s-)\,dY(s)+Y(s-)\,d\theta(s)
		+\sum_{0<r\le T}\Delta\theta(r)\Delta Y(r)\delta_r(ds).
		\label{eq:product-2}
	\end{equation}
	
	As in the proof of Corollary~\ref{cor:weighted-Ito-general},
	left limits can be replaced by current values in the continuous
	terms of \eqref{eq:product-2}. At each jump time,
	\eqref{eq:dY-explicit-backward} and
	\eqref{eq:d-theta-explicit} yield
	\begin{align*}
		\Delta(\theta Y)(s)
		&=\theta(s)\Delta Y(s)+Y(s-)\Delta\theta(s)\\
		&=-\theta(s)
		\Big[
		2\langle h(s-),\zeta(s)\rangle_H
		-\|\zeta(s)\|_H^2\Delta F(s)
		\Big]\Delta F(s)+\theta(s)\bigl(1-e^{-\Delta K(s)}\bigr)
		\|h(s-)\|_H^2.
	\end{align*}
	We substitute \eqref{eq:dY-explicit-backward} and
	\eqref{eq:d-theta-explicit} into \eqref{eq:product-2}, use the
	jump contribution computed above, and integrate over \((t,T]\).
	Using \(h(T)=h_T\) and collecting the continuous and atomic
	parts of \(dF\) gives \eqref{eq:positive-weight-backward-FV-Ito-M}.
\end{proof}

\begin{remark}
\label{rem:multiple-FV-driver-reduction}
Let \(\nu_j\), \(1\le j\le m\), be adapted c\`adl\`ag
real-valued finite-variation processes with \(\nu_j(0)=0\), and set
$
K(t):=\sum_{j=1}^m|\nu_j|_v(t).
$
The Radon--Nikodym theorem for adapted finite-variation processes
provides progressively measurable densities \(q_j\) such that,
\(\mathbb P\)-a.s.,
\[
d\nu_j=q_j\,dK,\quad
d|\nu_j|_v=|q_j|\,dK,\quad
\sum_{j=1}^m|q_j|=1\quad dK\text{-a.e.}
\]
Setting the densities to zero on the progressively measurable set
where \(\sum_j|q_j|\ne1\) preserves these measure identities and
ensures \(\sum_j|q_j|\le1\) everywhere.

For progressively measurable \(H\)-valued processes \(\zeta_j\)
satisfying
\(\int_{(0,T]}\|\zeta_j(t)\|_H\,d|\nu_j|_v(t)<\infty\),
\(\mathbb P\)-a.s., define \(\zeta=\sum_jq_j\zeta_j\). Then
\[
\zeta\,dK=\sum_{j=1}^m\zeta_j\,d\nu_j,\quad
\zeta(s)\Delta K(s)
=\sum_{j=1}^m\zeta_j(s)\Delta\nu_j(s).
\]
Under their respective integrability hypotheses,
Lemmas~\ref{Ito-lemma} and~\ref{lem:backward-FV-Ito-M}
therefore apply with finite-variation driver \(K\).
In these formulas, each linear Stieltjes term is replaced by
its sum over \(j\), and
\(\|\zeta(s)\Delta F(s)\|_H^2\) is replaced by
$
\|\sum_{j=1}^m\zeta_j(s)\Delta\nu_j(s)\|_H^2.
$

For each \(j\), let
\(\gamma_j,\psi_j:\Omega\times[0,T]\times H\to H\)
denote the coefficients multiplying \(d\nu_j\) in the SEE and
BSEE, respectively. Assume that these maps are
\(\operatorname{Prog}\otimes\mathcal B(H)\)-measurable and
Lipschitz in the state variable, with Lipschitz constants
bounded by \(L_j\). Define
\[
\widetilde\gamma(t,x)
:=\sum_{j=1}^m q_j(t)\gamma_j(t,x),\quad
\widetilde\psi(t,y)
:=\sum_{j=1}^m q_j(t)\psi_j(t,y).
\]
Since \(\sum_j|q_j|\le1\), both combined coefficients have
Lipschitz constant at most \(\max_jL_j\). For every nonnegative
measurable weight \(w\),
\[
\int_{(0,T]}w(t)\|\widetilde\gamma(t,0)\|_H\,dK(t)
\le
\sum_{j=1}^m
\int_{(0,T]}w(t)\|\gamma_j(t,0)\|_H\,d|\nu_j|_v(t),
\]
and the same bound holds for \(\widetilde\psi\).
Thus the single-driver SEE and BSEE estimates and well-posedness
results apply to the combined coefficients with driver \(K\),
provided the remaining hypotheses of the respective result hold.
The weights and exponential-moment conditions are then formulated
in terms of \(K\).

One can retain a given adapted nondecreasing c\`adl\`ag
reference driver \(\xi\) if, \(\mathbb P\)-a.s.,
\[
d\nu_j\ll d\xi\quad\text{for every }j,\quad
\sum_{j=1}^m L_j\,d|\nu_j|_v\le\ell\,d\xi.
\]
Take progressively measurable densities \(q_j\) with respect
to \(d\xi\), so that
\[
d\nu_j=q_j\,d\xi,\quad
d|\nu_j|_v=|q_j|\,d\xi,\quad
\sum_{j=1}^m L_j|q_j|\le\ell
\quad d\xi\text{-a.e.},\quad \mathbb P\text{-a.s.}
\]
Setting all \(q_j\) to zero where the last inequality fails
ensures that it holds everywhere without changing the measure identities.
With these densities, the coefficients
\(\widetilde\gamma,\widetilde\psi\) are Lipschitz with constant at most
\(\ell\), and the preceding data bound holds with \(K\) replaced
by \(\xi\). Under the remaining hypotheses of the corresponding
result, the single-driver estimates apply with driver \(\xi\)
and weights \(\rho_{\beta,\ell}^{r}\) and
\(\theta_{\beta,\ell}^{0}\).
A coefficient independent of the state variable has \(L_j=0\)
and contributes only to the data term; for the fixed reference
driver \(\xi\), it does not change the weight.
\end{remark}

\subsection{Stochastic evolution equations with finite-variation drivers}\label{subsec:FV-SEE}

We consider the following nonlinear stochastic evolution equation:
\begin{equation}
\begin{aligned}
    x(t)
    &=
    x_0
    +
    \int_0^t
        \big[
            A(s)x(s)+a(s,x(s))
        \big]\,ds
    +
    \int_0^t
        \big[
            B(s)x(s)+b(s,x(s))
        \big]\,dW(s)
    \\
    &\quad
    +
    \int_{(0,t]}
        \gamma(s,x(s-))\,d\xi(s),
    \quad t\in[0,T],
\end{aligned}
\label{eq:nonlinear-FV-SEE}
\end{equation}
where
\(A:[0,T]\times\Omega\to\mathfrak L(V;V^*), \ B:[0,T]\times\Omega\to\mathfrak L(V;H)\)
are random linear operators. The mappings
\(a,b,\gamma:\Omega\times[0,T]\times H\to H\) are nonlinear coefficients.

\begin{definition}
\label{def:nonlinear-FV-SEE-solution}
An adapted c\`adl\`ag \(H\)-valued process \(x\), with
\(x(\cdot,\omega)\in L^2(0,T;V)\) for \(\mathbb P\)-almost every \(\omega\), is called a
solution of
\eqref{eq:nonlinear-FV-SEE} if, for every \(t\in[0,T]\) and
\(\varphi\in V\), the following identity holds outside one
\(\mathbb P\)-null set:
\begin{align}
    \langle x(t),\varphi\rangle_H
    &=
    \langle x_0,\varphi\rangle_H
    +
    \int_0^t
        \langle A(s)x(s),\varphi\rangle_*\,ds
    +
    \int_0^t
        \langle a(s,x(s)),\varphi\rangle_H\,ds
    \notag\\
    &\quad
    +
    \int_0^t
        \langle B(s)x(s)+b(s,x(s)),\varphi\rangle_H\,dW(s)
    +
    \int_{(0,t]}
        \langle\gamma(s,x(s-)),\varphi\rangle_H\,d\xi(s).
    \label{eq:nonlinear-FV-SEE-weak-form}
\end{align}

\end{definition}

We use the following assumptions.
\begin{description}

\item[\((A1)\)]
For every \(x,y\in V\) and \(z\in H\), the scalar processes
\((t,\omega)\mapsto\langle A(t)x,y\rangle_*\) and
\((t,\omega)\mapsto\langle B(t)x,z\rangle_H\) are progressively
measurable. There exist constants
\(\delta>0\), \(C_A>0\), and \(C_B>0\) such that the following
inequalities hold for every \((t,\omega)\in[0,T]\times\Omega\) and
\(x\in V\):
\begin{equation}
    2\langle A(t)x,x\rangle_*
    +
    \|B(t)x\|_H^2
    \le
    -\delta\|x\|_V^2
    +
    C_A\|x\|_H^2,
    \label{eq:SEE-coercivity}
\end{equation}
\begin{equation}
    \|A(t)x\|_{V^*}
    \le
    C_A\|x\|_V,
    \label{eq:SEE-operator-boundedness}
\end{equation}
and
\begin{equation}
    \big|
        \langle B(t)x,x\rangle_H
    \big|
    \le
    C_B\|x\|_H^2.
    \label{eq:SEE-B-quasi-skew}
\end{equation}

\item[\((A2)\)]
The mappings \(a,b:\Omega\times[0,T]\times H\to H\) are
\(\operatorname{Prog}\otimes\mathcal B(H)\)-measurable. There exists
\(L>0\) such that, for every \((t,\omega)\in[0,T]\times\Omega\),
\[
\|a(t,x)-a(t,y)\|_H+\|b(t,x)-b(t,y)\|_H
\le L\|x-y\|_H,
\quad x,y\in H.
\]
\item[\((A3)\)]
The mapping \(\gamma:\Omega\times[0,T]\times H\to H\) is
\(\operatorname{Prog}\otimes\mathcal B(H)\)-measurable. There exists
\(\ell\ge0\) such that, for every \((t,\omega)\in[0,T]\times\Omega\),
\[
\|\gamma(t,x)-\gamma(t,y)\|_H\le \ell\|x-y\|_H,
\quad x,y\in H.
\]
\end{description}

The coercivity and boundedness conditions in \((A1)\), together with the
continuous embedding \(V\hookrightarrow H\), imply the uniform bound
\(\|B(t)x\|_H\le C\|x\|_V, \ x\in V.\)

For use with the state and variational equations, we formulate
the estimates on a deterministic interval \([r,T]\). Consider
\begin{equation}
\left\{
\begin{aligned}
    dX(t)
    &={}
    [A(t)X(t)+a(t,X(t))]dt
    +[B(t)X(t)+b(t,X(t))]dW(t)
    +
    \gamma(t,X(t-))d\xi(t),
    \\
    X(r)&=X_r.
\end{aligned}
\right.
\label{eq:unbounded-xi-forward-SEE}
\end{equation}

The next proposition gives weighted and unweighted estimates for the maximum
of the state process and its \(V\)-energy.
\begin{proposition}
	\label{prop:unbounded-xi-SEE-estimate}
	\label{prop:nonlinear-FV-SEE-estimate}
	Let \(\alpha\ge1\) and assume that
	\((A1)\)--\((A3)\) hold. Let $X_r\in L^{2\alpha}(\mathcal F_r;H)$, and suppose that $X$ solves \eqref{eq:unbounded-xi-forward-SEE}. There exist constants
	\(\beta_\alpha>0\) and \(C_\alpha>0\), depending only on
	\(\alpha,\delta,C_A,C_B,L\) and the norm of the embedding
	\(V\hookrightarrow H\), such that, for every
	\(\beta\ge\beta_\alpha\), if
	\[
\begin{aligned}
 \bigl(\rho_{\beta,\ell}^{r}\bigr)^{1/2}a(\cdot,0)
		\in L^{1,2\alpha}_{\mathbb F}(r,T;H),\quad
 \bigl(\rho_{\beta,\ell}^{r}\bigr)^{1/2}b(\cdot,0)
		\in L^{2,2\alpha}_{\mathbb F}(r,T;H),\quad
 \bigl(\rho_{\beta,\ell}^{r}\bigr)^{1/2}\gamma(\cdot,0)
		\in L^{1,2\alpha}_{\mathbb F,\xi}(r,T;H),
\end{aligned}
\]
	then
	\begin{align}
		&\mathbb E\Bigg[
		\sup_{r\le t\le T}
		\bigl(\rho_{\beta,\ell}^{r}(t)\bigr)^\alpha
		\|X(t)\|_H^{2\alpha}
		+
		\left(
		\int_r^T
		\rho_{\beta,\ell}^{r}(t)\|X(t)\|_V^2\,dt
		\right)^\alpha
        \notag\\
        &\quad+
        \left(\beta\int_r^T\rho_{\beta,\ell}^{r}(t)
                                  \|X(t)\|_H^2\,dt\right)^\alpha
        +\left((\kappa-2)\ell\int_{(r,T]}
              \rho_{\beta,\ell}^{r}(t)\|X(t-)\|_H^2\,d\xi(t)\right)^\alpha
		\Bigg]
		\notag\\
		&\quad\le C_\alpha\,\mathbb E\Bigg[
		\|X_r\|_H^{2\alpha}
		+
		\left(
		\int_r^T
		\bigl(\rho_{\beta,\ell}^{r}(t)\bigr)^{1/2}
		\|a(t,0)\|_H\,dt
		\right)^{2\alpha}
		\notag\\
		&\hspace{31mm}
		+
		\left(
		\int_r^T
		\rho_{\beta,\ell}^{r}(t)
		\|b(t,0)\|_H^2\,dt
		\right)^\alpha
		+
		\left(
		\int_{(r,T]}
		\bigl(\rho_{\beta,\ell}^{r}(t)\bigr)^{1/2}
		\|\gamma(t,0)\|_H\,d\xi(t)
		\right)^{2\alpha}
		\Bigg],
		\label{eq:unbounded-xi-weighted-SEE-estimate}
	\end{align}
	Assume in addition that $
		\mathfrak E_{2\alpha\kappa,\ell}(\xi)<\infty
	$
	and
	\begin{equation*}
\begin{aligned}
 &X_r\in L^{4\alpha}(\mathcal F_r;H),\quad
  a(\cdot,0)\in L^{1,4\alpha}_{\mathbb F}(r,T;H),\quad
 b(\cdot,0)\in L^{2,4\alpha}_{\mathbb F}(r,T;H),\quad
  \gamma(\cdot,0)\in L^{1,4\alpha}_{\mathbb F,\xi}(r,T;H).
\end{aligned}
\end{equation*}
	Then, for every
	\(\beta\ge\max\{\beta_\alpha,\beta_{2\alpha}\}\),
	\begin{align}
		&\mathbb E\left[
		\sup_{r\le t\le T}\|X(t)\|_H^{2\alpha}
		+
		\left(
		\int_r^T\|X(t)\|_V^2\,dt
		\right)^\alpha
		\right]
		\notag\\
		&\quad\le
		C_{\alpha,\beta}
		\mathfrak E_{2\alpha\kappa,\ell}(\xi)^{1/2}
		\Bigg\{
		\mathbb E\Bigg[
		\|X_r\|_H^{4\alpha}
		+
		\left(
		\int_r^T\|a(t,0)\|_H\,dt
		\right)^{4\alpha}
		\notag\\
		&\hspace{37mm}
		+
		\left(
		\int_r^T\|b(t,0)\|_H^2\,dt
		\right)^{2\alpha}
		+
		\left(
		\int_{(r,T]}\|\gamma(t,0)\|_H\,d\xi(t)
		\right)^{4\alpha}
		\Bigg]
		\Bigg\}^{1/2}.
		\label{eq:unbounded-xi-unweighted-SEE-estimate}
	\end{align}
	Here \(C_{\alpha,\beta}\) depends only on
	\(\alpha,\beta,T,\delta,C_A,C_B,L\) and the norm of the embedding
	\(V\hookrightarrow H\).
	
	Both \eqref{eq:unbounded-xi-weighted-SEE-estimate} and
	\eqref{eq:unbounded-xi-unweighted-SEE-estimate} remain valid with
	\(\mathbb E[\cdot\mid\mathcal F_r]\) in place of each expectation,
	under the respective data conditions. The deterministic constants
	are unchanged.
\end{proposition}

\begin{proof}
Throughout this proof, write $\rho(t)=\rho_{\beta,\ell}^{r}(t)$ and set
\begin{align*}
 D&=\|X_r\|_H^2
   +\left(\int_r^T\rho(t)^{1/2}\|a(t,0)\|_H\,dt\right)^2
   +\int_r^T\rho(t)\|b(t,0)\|_H^2\,dt+\left(\int_{(r,T]}\rho(t)^{1/2}
                          \|\gamma(t,0)\|_H\,d\xi(t)\right)^2.
\end{align*}
The data assumptions are equivalent to \(D\in L^\alpha(\mathcal F_T)\).

Applying Corollary~\ref{cor:weighted-Ito-general} to \(X\), we obtain
	\begin{align}
		&\rho(t)\|X(t)\|_H^2
		=
		\|X_r\|_H^2
		\notag\\
		&\quad+
		\int_r^t\rho(s)
		\Big[
		2\langle A(s)X(s)+a(s,X(s)),X(s)\rangle_*
		+\|B(s)X(s)+b(s,X(s))\|_H^2
		-\beta\|X(s)\|_H^2
		\Big]\,ds
		\notag\\
		&\quad+
		\int_{(r,t]}\rho(s)
		\Big[
		2\langle X(s-),\gamma(s,X(s-))\rangle_H
		+\|\gamma(s,X(s-))\|_H^2\Delta\xi(s)
		\Big]\,d\xi(s)
		\notag\\
		&\quad-
		\kappa\ell\int_r^t
		\rho(s)\|X(s)\|_H^2\,d\xi^c(s)-
		\sum_{r<s\le t}
		\rho(s)
		\bigl(e^{\kappa\ell\Delta\xi(s)}-1\bigr)
		\|X(s-)\|_H^2
		+M(t),
		\label{eq:unbounded-xi-forward-weighted-identity}
	\end{align}
\equationlabelalias{eq:FV-SEE-high-order-Vstar-weighted-identity}{eq:unbounded-xi-forward-weighted-identity}
	where
	\begin{equation}
M(t)
	=
	2\int_r^t
	\rho(s-)
	\big\langle
	X(s-),B(s)X(s)+b(s,X(s))
	\big\rangle_H\,dW(s)
\label{eq:SEE-weighted-energy-martingale}
\end{equation}
	 is a continuous local martingale.
	
	We estimate the \(d\xi\)-integral, the \(d\xi^c\)-integral,
	and the jump sum in
	\eqref{eq:unbounded-xi-forward-weighted-identity} together.
	Since \(X(t)=X(t-)\) for \(d\xi^c\)-almost every \(t\),
	\(\mathbb P\)-a.s., assumption \((A3)\) yields
	\begin{align}
		&2\langle X(t),\gamma(t,X(t))\rangle_H
		-\kappa\ell\|X(t)\|_H^2
		\le
		-(\kappa-2)\ell\|X(t)\|_H^2
		+2\|X(t)\|_H\|\gamma(t,0)\|_H.
		\label{eq:unbounded-xi-continuous-FV-bound}
	\end{align}
	At a jump time \(t\), 
	(\ref{eq:unbounded-xi-exponential-quadratic-lower-bound}) with \(z=\ell\Delta\xi(t)\) gives
	\begin{align}
		&2\langle X(t-),\gamma(t,X(t-))\rangle_H\Delta\xi(t)
		+\|\gamma(t,X(t-))\|_H^2\bigl(\Delta\xi(t)\bigr)^2-
		\bigl(e^{\kappa\ell\Delta\xi(t)}-1\bigr)
		\|X(t-)\|_H^2
		\notag\\
		&\le
		-(\kappa-2)\ell\|X(t-)\|_H^2\Delta\xi(t)
		-(\kappa-2)\ell^2\|X(t-)\|_H^2
		\bigl(\Delta\xi(t)\bigr)^2
		\notag\\
		&\quad+
		2\|X(t-)\|_H\|\gamma(t,0)\|_H\Delta\xi(t)
		+2\|\gamma(t,0)\|_H^2\bigl(\Delta\xi(t)\bigr)^2.
		\label{eq:unbounded-xi-forward-jump-bound}
	\end{align}
	Since \(\rho\) is nonincreasing and \(X\) is c\`adl\`ag,
	taking the left limit at \(s\) gives, for \(r<s\le t\le T\),
	\[
	\begin{aligned}
		\bigl(\rho(s)\bigr)^{1/2}\|X(s-)\|_H
		&\le
		\bigl(\rho(s-)\bigr)^{1/2}\|X(s-)\|_H \le
		\sup_{r\le u\le t}
		\bigl(\rho(u)\bigr)^{1/2}\|X(u)\|_H .
	\end{aligned}
	\]
	
	The additive jump terms satisfy
	\[
	\sum_{r<s\le t}
	\rho(s)\|\gamma(s,0)\|_H^2
	\bigl(\Delta\xi(s)\bigr)^2
	\le
	\left(
	\int_{(r,t]}
	\bigl(\rho(s)\bigr)^{1/2}
	\|\gamma(s,0)\|_H\,d\xi(s)
	\right)^2.
	\]
	We multiply \eqref{eq:unbounded-xi-continuous-FV-bound} by
	\(\rho(s)\) and integrate against \(d\xi^c(s)\) over \((r,t]\),
	then multiply \eqref{eq:unbounded-xi-forward-jump-bound} by
	\(\rho(s)\) and sum over \(r<s\le t\). Adding these inequalities
	and applying the preceding maximum and additive-jump bounds gives
	\begin{align}
		&\int_{(r,t]}\rho(s)
		\Big[2\langle X(s-),\gamma(s,X(s-))\rangle_H
		+\|\gamma(s,X(s-))\|_H^2\Delta\xi(s)\Big]\,d\xi(s)
		\notag\\
		&\quad-\kappa\ell\int_r^t\rho(s)\|X(s)\|_H^2\,d\xi^c(s)
		-\sum_{r<s\le t}\rho(s)
		\bigl(e^{\kappa\ell\Delta\xi(s)}-1\bigr)\|X(s-)\|_H^2
		\notag\\
		&\quad\le-(\kappa-2)\ell
		\int_{(r,t]}
		\rho(s)
		\|X(s-)\|_H^2\,d\xi(s)
		\notag\\
		&\quad+
		2\sup_{r\le s\le t}
		\bigl(\rho(s)\bigr)^{1/2}
		\|X(s)\|_H
		\int_{(r,t]}
		\bigl(\rho(s)\bigr)^{1/2}
		\|\gamma(s,0)\|_H\,d\xi(s)
		\notag\\
		&\quad+
		2\left(
		\int_{(r,t]}
		\bigl(\rho(s)\bigr)^{1/2}
		\|\gamma(s,0)\|_H\,d\xi(s)
		\right)^2.
		\label{eq:unbounded-xi-forward-dxi-bound}
	\end{align}

	For the Lebesgue integrand in
	\eqref{eq:unbounded-xi-forward-weighted-identity},
	\eqref{eq:SEE-coercivity} and \((A2)\) give 
	\begin{align}
		&2\langle A(t)X(t)+a(t,X(t)),X(t)\rangle_*
		+\|B(t)X(t)+b(t,X(t))\|_H^2
		\notag\\
		&\quad\le
		-\frac{\delta}{2}\|X(t)\|_V^2
		+C_0\|X(t)\|_H^2
		+C_0\|b(t,0)\|_H^2
		+2\|X(t)\|_H\|a(t,0)\|_H,
		\label{eq:nonlinear-SEE-dt-bound}
	\end{align}
	where \(C_0\) depends only on
	\(\delta,C_A,L\) and the norm of the embedding \(V\hookrightarrow H\).
	
Inserting
	\eqref{eq:nonlinear-SEE-dt-bound} and
	\eqref{eq:unbounded-xi-forward-dxi-bound} into
	\eqref{eq:unbounded-xi-forward-weighted-identity}, we obtain
	\begin{align}
		&\rho(t)\|X(t)\|_H^2
		+\frac{\delta}{2}\int_r^t
		\rho(s)\|X(s)\|_V^2\,ds
		+(\beta-C_0)\int_r^t
		\rho(s)\|X(s)\|_H^2\,ds
		\notag\\
		&\quad
		+(\kappa-2)\ell\int_{(r,t]}
		\rho(s)
		\|X(s-)\|_H^2\,d\xi(s)
		\notag\\
		&\le
		\|X_r\|_H^2
		+C_0\int_r^t
		\rho(s)\|b(s,0)\|_H^2\,ds
		\notag\\
		&\quad
		+2\left(
		\sup_{r\le u\le t}
		\rho(u)\|X(u)\|_H^2
		\right)^{1/2}
		\Bigg[
		\int_r^t
		\bigl(\rho(s)\bigr)^{1/2}
		\|a(s,0)\|_H\,ds
		+\int_{(r,t]}
		\bigl(\rho(s)\bigr)^{1/2}
		\|\gamma(s,0)\|_H\,d\xi(s)
		\Bigg]
		\notag\\
		&\quad
		+2\left(
		\int_{(r,t]}
		\bigl(\rho(s)\bigr)^{1/2}
		\|\gamma(s,0)\|_H\,d\xi(s)
		\right)^2
		+M(t).
		\label{eq:unbounded-xi-forward-pre-Young}
	\end{align}
Choose \(\beta\ge2C_0+1\) and a stopping time
\(\sigma\in[r,T]\). For \(r\le t\le\sigma\), Young's inequality gives
\begin{align*}
&2\left(\sup_{r\le u\le t}\rho(u)\|X(u)\|_H^2\right)^{1/2}
\left[\int_r^t\rho(s)^{1/2}\|a(s,0)\|_H\,ds
+\int_{(r,t]}\rho(s)^{1/2}\|\gamma(s,0)\|_H\,d\xi(s)\right]\\
&\quad\le\varepsilon\sup_{r\le u\le\sigma}\rho(u)\|X(u)\|_H^2
+C_\varepsilon D.
\end{align*}
Using this bound for the product in
\eqref{eq:unbounded-xi-forward-pre-Young}, we evaluate the resulting inequality
at \(t=\sigma\) to control the nonnegative integrals and take
the supremum over \(t\in[r,\sigma]\) to control the maximum of
\(\rho(t)\|X(t)\|_H^2\). Adding these two bounds and absorbing the multiple of the maximum
for sufficiently small \(\varepsilon\), we obtain
	\begin{align}
		&\sup_{r\le t\le\sigma}
		\rho(t)\|X(t)\|_H^2
		+\int_r^\sigma
		\rho(t)\|X(t)\|_V^2\,dt
		+\beta\int_r^\sigma
		\rho(t)\|X(t)\|_H^2\,dt
        \notag\\
        &\quad+(\kappa-2)\ell\int_{(r,\sigma]}
                    \rho(t)\|X(t-)\|_H^2\,d\xi(t)
		\notag\\
		&\quad\le C\left[D+\sup_{r\le t\le\sigma}|M(t)|\right].
        \label{eq:unbounded-xi-forward-pathwise-estimate}
    \end{align}

	It remains to control the martingale maximum. Set
	\[
	\sigma_n
	:=
	\inf\{t\in[r,T]:\langle M\rangle(t)\ge n\}\wedge T,
	\quad
	\inf\varnothing:=+\infty.
	\]
	By \eqref{eq:SEE-B-quasi-skew} and \((A2)\), the quadratic variation
	of \(M\) defined in \eqref{eq:SEE-weighted-energy-martingale} satisfies
	\begin{align}
		\langle M\rangle(\sigma_n)
		&\le
		C\left(
		\sup_{r\le t\le\sigma_n}
		\rho(t)\|X(t)\|_H^2
		\right)
		\left[
		\int_r^{\sigma_n}
		\rho(t)\|X(t)\|_H^2\,dt
		+
		\int_r^T
		\rho(t)\|b(t,0)\|_H^2\,dt
		\right].
		\label{eq:unbounded-xi-forward-bracket-bound}
	\end{align}
\equationlabelalias{eq:FV-SEE-high-order-Vstar-bracket}{eq:unbounded-xi-forward-bracket-bound}
	The constant in \eqref{eq:unbounded-xi-forward-bracket-bound} depends only on \(C_B,L\).
	
	Since \(\langle M\rangle\) is continuous,
	\(\langle M\rangle(\sigma_n)\le n\), and the BDG inequality gives
	\(\mathbb E\sup_{r\le t\le\sigma_n}|M(t)|^\alpha
	\le C_\alpha n^{\alpha/2}<\infty\). Thus
	\eqref{eq:unbounded-xi-forward-pathwise-estimate}, with
	\(\sigma=\sigma_n\), shows that each stopped term on its left
	is integrable to order \(\alpha\), since \(D\in L^\alpha\).
	Combining the BDG inequality with
	\eqref{eq:unbounded-xi-forward-bracket-bound} and Young's
	inequality, we therefore obtain, for every \(\varepsilon>0\),
	\begin{align}
		\mathbb E
		\sup_{r\le t\le\sigma_n}|M(t)|^\alpha
		&\quad\le
		\varepsilon
		\mathbb E\sup_{r\le t\le\sigma_n}
		\bigl(\rho(t)\bigr)^\alpha
		\|X(t)\|_H^{2\alpha}
		+
		C_{\alpha,\varepsilon}
		\mathbb E\left(
		\int_r^{\sigma_n}
		\rho(t)\|X(t)\|_H^2\,dt
		\right)^\alpha
		\notag\\
		&\quad+
		C_{\alpha,\varepsilon}
		\mathbb E\left(
		\int_r^T
		\rho(t)\|b(t,0)\|_H^2\,dt
		\right)^\alpha.
		\label{eq:unbounded-xi-forward-BDG-bound}
	\end{align}

	We set \(\sigma=\sigma_n\) in
	\eqref{eq:unbounded-xi-forward-pathwise-estimate}, raise both
	sides to the power \(\alpha\), and take expectations.
	We then use \eqref{eq:unbounded-xi-forward-BDG-bound} to estimate
	\(\mathbb E\sup_{r\le t\le\sigma_n}|M(t)|^\alpha\).
	After rescaling the parameter in Young's inequality, we obtain, for every
	\(\varepsilon>0\),
	\begin{align}
		&\mathbb E\sup_{r\le t\le\sigma_n}
		\bigl(\rho(t)\bigr)^\alpha
		\|X(t)\|_H^{2\alpha}
		+
		\mathbb E\left(
		\int_r^{\sigma_n}
		\rho(t)\|X(t)\|_V^2\,dt
		\right)^\alpha
		+\beta^\alpha\mathbb E\left(
		\int_r^{\sigma_n}
		\rho(t)\|X(t)\|_H^2\,dt
		\right)^\alpha
        \notag\\
        &\quad+\mathbb E\left((\kappa-2)\ell\int_{(r,\sigma_n]}
                  \rho(t)\|X(t-)\|_H^2\,d\xi(t)\right)^\alpha
		\notag\\
		&\le
		C_{\alpha,\varepsilon}\mathbb ED^\alpha
		+\varepsilon\,
		\mathbb E\sup_{r\le t\le\sigma_n}
		\bigl(\rho(t)\bigr)^\alpha
		\|X(t)\|_H^{2\alpha}
		+C_{\alpha,\varepsilon}\mathbb E\left(
		\int_r^{\sigma_n}
		\rho(t)\|X(t)\|_H^2\,dt
		\right)^\alpha .
		\label{eq:unbounded-xi-forward-before-absorption}
	\end{align}

	Choose \(\varepsilon=1/2\), and then choose
	\[
	\beta_\alpha
	\ge
	\max\left\{
	2C_0+1,\,
	\bigl(2C_{\alpha,1/2}\bigr)^{1/\alpha}
	\right\}.
	\]
	For every \(\beta\ge\beta_\alpha\), we have
	\(C_{\alpha,1/2}\le\frac{\beta^\alpha}{2}.\)
	Absorbing the last two terms on the right-hand side of
	\eqref{eq:unbounded-xi-forward-before-absorption}, we obtain
	\begin{align}
		&\frac12\mathbb E\sup_{r\le t\le\sigma_n}
		\bigl(\rho(t)\bigr)^\alpha
		\|X(t)\|_H^{2\alpha}
		+
		\mathbb E\left(
		\int_r^{\sigma_n}
		\rho(t)\|X(t)\|_V^2\,dt
		\right)^\alpha
		\notag\\
		&\quad
		+\frac{\beta^\alpha}{2}\mathbb E\left(
		\int_r^{\sigma_n}
		\rho(t)\|X(t)\|_H^2\,dt
		\right)^\alpha
        +\mathbb E\left((\kappa-2)\ell\int_{(r,\sigma_n]}
                  \rho(t)\|X(t-)\|_H^2\,d\xi(t)\right)^\alpha
		\notag\\
		&\le
		C_\alpha\mathbb ED^\alpha.
		\label{eq:unbounded-xi-forward-stopped-estimate}
	\end{align}
	The right-hand side of
	\eqref{eq:unbounded-xi-forward-stopped-estimate} is independent of
	\(n\).
Since \(M\) is a continuous local martingale on \([r,T]\),
    \(\langle M\rangle(T)<\infty\) almost surely. Consequently,
    \(\sigma_n=T\) eventually on almost every path. Fatou's
    lemma applied to the stopped maximum and each stopped energy in
    \eqref{eq:unbounded-xi-forward-stopped-estimate} gives
    \eqref{eq:unbounded-xi-weighted-SEE-estimate}.

    For the conditional estimate, fix \(E\in\mathcal F_r\) and
    multiply the pathwise inequality
    \eqref{eq:unbounded-xi-forward-pathwise-estimate} by \(I_E\).
    On \([r,T]\), \(I_EM\) is a continuous local martingale with
    bracket \(I_E\langle M\rangle\). Applying the BDG inequality
    to \(I_EM(\cdot\wedge\sigma_n)\) yields
    \eqref{eq:unbounded-xi-forward-BDG-bound} with
    \(\mathbb E[I_E\,\cdot]\) in place of each expectation.
    The same choices of \(\varepsilon\) and \(\beta_\alpha\) then
    give \eqref{eq:unbounded-xi-forward-stopped-estimate} with
    \(\mathbb E[I_E\,\cdot]\) in place of each expectation. Fatou's lemma and the
    characterization of conditional expectation by testing against events give
    \eqref{eq:unbounded-xi-weighted-SEE-estimate}, including all four
    nonnegative terms on its left, conditional on \(\mathcal F_r\).

	It remains to remove the weight. Let
	\(\beta\ge\max\{\beta_\alpha,\beta_{2\alpha}\}\). Since
	\(\bigl(\rho(t)\bigr)^{-1} \le e^{\beta T+\kappa\ell(\xi(T)-\xi(r))}, \ r\le t\le T,\)
the conditional Cauchy--Schwarz inequality, followed by the conditional form of
\eqref{eq:unbounded-xi-weighted-SEE-estimate} at order \(2\alpha\), gives
\begin{equation}
\begin{aligned}
&\mathbb E\left[
 \sup_{r\le t\le T}\|X(t)\|_H^{2\alpha}
 +\left(\int_r^T\|X(t)\|_V^2\,dt\right)^\alpha
 \,\middle|\,\mathcal F_r\right]\\
&\quad\le C_\alpha e^{\alpha\beta T}
 \left(\mathbb E\left[
  e^{2\alpha\kappa\ell(\xi(T)-\xi(r))}
  \,\middle|\,\mathcal F_r\right]\right)^{1/2}
 \left(\mathbb E[D^{2\alpha}\mid\mathcal F_r]\right)^{1/2}.
\end{aligned}
\label{eq:SEE-conditional-unweighting}
\end{equation}
In \eqref{eq:SEE-conditional-unweighting},
\(\mathbb E[e^{2\alpha\kappa\ell(\xi(T)-\xi(r))}
\mid\mathcal F_r]^{1/2}\le
\mathfrak E_{2\alpha\kappa,\ell}(\xi)^{1/2}\) by
\eqref{eq:unbounded-xi-conditional-exponential-moment}.
Expanding \(D\) and using \(\rho\le1\) gives the conditional
form of \eqref{eq:unbounded-xi-unweighted-SEE-estimate}.
Taking expectations in this conditional bound and using
\(\mathbb E[(\mathbb E[D^{2\alpha}\mid\mathcal F_r])^{1/2}]
\le(\mathbb ED^{2\alpha})^{1/2}\) gives the unconditional estimate.
\end{proof}

We next establish well-posedness in the weighted class.
\begin{theorem}
\label{thm:FV-SEE-weighted-wellposed}
Let \(\alpha\ge1\), assume that
\((A1)\)--\((A3)\) hold, and let $X_r\in L^{2\alpha}(\mathcal F_r;H)$. Fix \(\kappa>2\) and
\(\beta\ge\beta_\alpha\), and suppose that
\begin{equation*}
\begin{aligned}
 \bigl(\rho_{\beta,\ell}^{r}\bigr)^{1/2}a(\cdot,0)
\in L^{1,2\alpha}_{\mathbb F}(r,T;H),\quad
 \bigl(\rho_{\beta,\ell}^{r}\bigr)^{1/2}b(\cdot,0)
\in L^{2,2\alpha}_{\mathbb F}(r,T;H),\quad
 \bigl(\rho_{\beta,\ell}^{r}\bigr)^{1/2}\gamma(\cdot,0)
\in L^{1,2\alpha}_{\mathbb F,\xi}(r,T;H).
\end{aligned}
\end{equation*}
Then \eqref{eq:unbounded-xi-forward-SEE} has a unique
solution \(X\) such that
$
\bigl(\rho_{\beta,\ell}^{r}\bigr)^{1/2}X
\in \mathbb D^{2\alpha}_{\mathbb F}(r,T;H)
\cap L^{2,2\alpha}_{\mathbb F}(r,T;V).
\label{eq:FV-SEE-weighted-solution-class}
$
\end{theorem}
\begin{proof}
To prove uniqueness, let \(X\) and \(\widetilde X\) be two solutions
satisfying the weighted integrability condition in the statement. Then
\(Y:=X-\widetilde X\) has zero initial value and satisfies the
difference equation whose nonlinear coefficients are
\[
\begin{aligned}
a(t,z+\widetilde X(t))-a(t,\widetilde X(t)),\ \
b(t,z+\widetilde X(t))-b(t,\widetilde X(t)),\ \
\gamma(t,z+\widetilde X(t-))
 -\gamma(t,\widetilde X(t-)).
\end{aligned}
\]
They vanish at \(z=0\) and have the same Lipschitz bounds as
\(a,b,\gamma\). Hence
\eqref{eq:unbounded-xi-weighted-SEE-estimate}, applied to \(Y\), gives
\[
 \mathbb E\sup_{r\le t\le T}
 \bigl(\rho_{\beta,\ell}^{r}(t)\bigr)^\alpha
 \|Y(t)\|_H^{2\alpha}=0.
\]
Since \(\rho_{\beta,\ell}^{r}>0\), \(X\) and \(\widetilde X\)
are indistinguishable.

We next construct a solution. Since
\(\rho_{\beta,\ell}^{r}\) is nonincreasing and
\(\rho_{\beta,\ell}^{r}(T)>0\) \(\mathbb P\text{-a.s.}\), the weighted data
assumptions imply
\[
 \int_r^T\|a(t,0)\|_H\,dt
 +\int_r^T\|b(t,0)\|_H^2\,dt
 +\int_{(r,T]}\|\gamma(t,0)\|_H\,d\xi(t)
 <\infty
 \quad\mathbb P\text{-a.s.}
\]
Define the adapted nondecreasing c\`adl\`ag process
\[
 K_r(t)
 :=
 \xi(t)-\xi(r)
 +\int_r^t\|a(s,0)\|_H\,ds
 +\int_r^t\|b(s,0)\|_H^2\,ds
 +\int_{(r,t]}\|\gamma(s,0)\|_H\,d\xi(s).
\]
Then \(K_r(T)<\infty\) \(\mathbb P\text{-a.s.}\) For \(n\ge1\), set
\[
 \sigma_n
 :=
 \inf\{t\in[r,T]:K_r(t)>n\},
 \quad
 \inf\varnothing:=+\infty,
\]
and
\[
 \chi_n(t)
 :=
 \int_{(r,t]}I_{\{s<\sigma_n\}}\,d\xi(s).
\]
Since \(K_r(r)=0\) and \(K_r\) is right-continuous,
\(\sigma_n>r\) whenever \(\sigma_n<\infty\).
The process \(\chi_n\) is adapted, nondecreasing, and c\`adl\`ag,
\(\chi_n(r)=0\), \(\chi_n(T)\le n\), and
\[
\begin{aligned}
 \int_r^T I_{\{t<\sigma_n\}}\|a(t,0)\|_H\,dt\le n,\quad
  \int_r^T I_{\{t<\sigma_n\}}\|b(t,0)\|_H^2\,dt\le n,\quad\int_{(r,T]}\|\gamma(t,0)\|_H\,d\chi_n(t)\le n.
\end{aligned}
\]
Hence Lemma~\ref{lem:localized-FV-SEE-construction}, applied with
driver \(\chi_n\), initial value \(X_r\), and coefficients
$
 I_{\{t<\sigma_n\}}a(t,x),\allowbreak\
 I_{\{t<\sigma_n\}}b(t,x),\allowbreak\
 \gamma(t,x),
$
gives a unique
$
 X^n\in
 \mathbb D^2_{\mathbb F}(r,T;H)
 \cap L^2_{\mathbb F}(r,T;V).
$

To verify consistency of the localized solutions, fix
\(m\ge n\), set
$
 \tau:=\sigma_n,
 \
 Y:=X^m-X^n,
$
and define
\[
 \widehat Y(t):=I_{\{t<\tau\}}Y(t),
 \quad
 J^\tau(t):=I_{\{\tau\le t\}},
 \quad r\le t\le T.
\]
Since \(K_r\) is nondecreasing, \(\sigma_m\ge\sigma_n\).
Thus, on \([r,\tau)\), the two localized equations have the same
initial value, \(dt\)- and \(dW\)-coefficients, and
finite-variation driver. Subtracting the two equations gives, on
\([r,\tau)\),
\begin{align*}
dY(t)
&=
\Big[
 A(t)Y(t)
 +a(t,Y(t)+X^n(t))-a(t,X^n(t))
\Big]\,dt\\
&\quad+
\Big[
 B(t)Y(t)
 +b(t,Y(t)+X^n(t))-b(t,X^n(t))
\Big]\,dW(t)\\
&\quad+
\Big[
 \gamma(t,Y(t-)+X^n(t-))
 -\gamma(t,X^n(t-))
\Big]\,d\xi(t),
\quad
Y(r)=0.
\end{align*}

We apply integration by parts to
\(I_{\{t<\tau\}} Y(t)\) on \([r,T]\).
At \(\tau\), the jump of the indicator and its
cross term with \(\Delta Y(\tau)\) cancel the original jump of
\(Y\), so that
\[
 \Delta\widehat Y(\tau)= I_{\{\tau<\tau\}}Y(\tau)
 -
Y(\tau-)=-Y(\tau-)
 =-\widehat Y(\tau-).
\]
Thus,
\begin{align}
d\widehat Y(t)
&=
\Big[
 A(t)\widehat Y(t)
 +I_{\{t<\tau\}}
 \big(
   a(t,\widehat Y(t)+X^n(t))-a(t,X^n(t))
 \big)
\Big]\,dt
\notag\\
&\quad+
\Big[
 B(t)\widehat Y(t)
 +I_{\{t<\tau\}}
 \big(
   b(t,\widehat Y(t)+X^n(t))-b(t,X^n(t))
 \big)
\Big]\,dW(t)
\notag\\
&\quad+
\Big[
 \gamma(t,\widehat Y(t-)+X^n(t-))
 -\gamma(t,X^n(t-))
\Big]\,d\chi_n(t)
-
\widehat Y(t-)\,dJ^\tau(t),
\quad
\widehat Y(r)=0.
\label{eq:strictly-stopped-SEE-difference}
\end{align}
The process \(\widehat Y\) belongs to
\(\mathbb D^2_{\mathbb F}(r,T;H)\cap L^2_{\mathbb F}(r,T;V)\),
since the indicator is bounded and \(X^m,X^n\) belong to these spaces.
By Remark~\ref{rem:multiple-FV-driver-reduction}, the two
finite-variation terms in
\eqref{eq:strictly-stopped-SEE-difference} can be written with the
single nondecreasing driver
\[
 G_n(t):=\chi_n(t)+J^\tau(t),\quad G_n(r)=0,\quad G_n(T)\le n+1.
\]
The combined coefficient vanishes at zero and has a Lipschitz constant
of at most \(\max\{\ell,1\}\). 
Thus \eqref{eq:strictly-stopped-SEE-difference},
written with driver \(G_n\), satisfies the assumptions of
Proposition~\ref{prop:unbounded-xi-SEE-estimate}. We apply
\eqref{eq:unbounded-xi-weighted-SEE-estimate} with moment parameter
\(1\), Lipschitz bound \(\max\{\ell,1\}\), zero initial value,
zero values of the coefficients at the origin, and sufficiently large
\(\lambda\), to obtain
\[
 \mathbb E\sup_{r\le t\le T}
 e^{-\lambda(t-r)-\kappa\max\{\ell,1\}G_n(t)}
 \|\widehat Y(t)\|_H^2=0.
\]
The weight is strictly positive, so \(\widehat Y=0\) up to
indistinguishability. Consequently,
$
 X^m(t)=X^n(t),\ r\le t<\sigma_n,\ t\le T,
$
outside one null set for the fixed pair \((m,n)\).

We intersect the events of probability one on which
\(X^m(t)=X^n(t)\) for every \(r\le t<\sigma_n\), over all
integer pairs \(m\ge n\). Since \(K_r(T)<\infty\) \(\mathbb P\text{-a.s.}\), for every
\(\omega\) in the resulting set there exists \(n_0(\omega)\) such
that
$
 \sigma_n(\omega)=+\infty,
 \ n\ge n_0(\omega).
$
Hence the paths \(X^n(\cdot,\omega)\) eventually coincide on the whole
interval \([r,T]\). We define \(X(\cdot,\omega)\) to be this common
path and set \(X=0\) on the complement of this event of probability one.

Then \(X\) is adapted and c\`adl\`ag, and
$
 \int_r^T\|X(t)\|_V^2\,dt<\infty
 \ \mathbb P\text{-a.s.}
$
Indeed, on almost every path, \(X\) coincides on \([r,T]\) with one of the
processes \(X^n\).
The Lipschitz bounds, \(K_r(T)<\infty\), and the pathwise
\(V\)-integrability ensure that the drift and Stieltjes integrals in
\eqref{eq:unbounded-xi-forward-SEE} are finite. They also imply
\(\int_r^T\|B(t)X(t)+b(t,X(t))\|_H^2\,dt<\infty\) almost
surely, so the Brownian integral is a continuous local martingale.
For each \(n\), \(X=X^n\) on \([r,\sigma_n)\); the Lebesgue
and Stieltjes integrals agree there pathwise, and the Brownian
integrals agree by locality of stochastic integration.
Since these stochastic intervals cover \([r,T]\) almost surely,
the weak equations for \(X^n\) give
\eqref{eq:unbounded-xi-forward-SEE} in the sense of
Definition~\ref{def:nonlinear-FV-SEE-solution}.

Finally, applying \eqref{eq:unbounded-xi-weighted-SEE-estimate} with the original
parameters \(\alpha,\beta,\ell\) and the assumed weighted data shows
that \(X\) belongs to the weighted solution class stated in
Theorem~\ref{thm:FV-SEE-weighted-wellposed}.
\end{proof}

The following corollary gives well-posedness in the unweighted moment space
used in the control and adjoint analysis.

\begin{corollary}
\label{thm:FV-SEE-wellposed}
Let \(\alpha\ge1\) and assume
\((A1)\)--\((A3)\). Suppose that
$
    e^{\kappa\ell\xi(T)}\in L^{2\alpha}(\mathcal F_T),
$
and
\[
\begin{aligned}
 X_r\in L^{4\alpha}(\mathcal F_r;H),\quad
  a(\cdot,0)\in L^{1,4\alpha}_{\mathbb F}(r,T;H),\quad b(\cdot,0)\in L^{2,4\alpha}_{\mathbb F}(r,T;H),\quad
  \gamma(\cdot,0)\in L^{1,4\alpha}_{\mathbb F,\xi}(r,T;H).
\end{aligned}
\]
Then \eqref{eq:unbounded-xi-forward-SEE} has a unique 
solution in
$
    \mathbb D^{2\alpha}_{\mathbb F}(r,T;H)
    \cap L^{2,2\alpha}_{\mathbb F}(r,T;V).
$
Moreover, for
\(\beta\ge\max\{\beta_\alpha,\beta_{2\alpha}\}\),
\begin{align}
&\mathbb E\left[
    \sup_{r\le t\le T}\|X(t)\|_H^{2\alpha}
    +\left(\int_r^T\|X(t)\|_V^2dt\right)^\alpha
\right]
\notag\\
&\quad\le
C_{\alpha,\beta,\kappa}
\left(\mathbb E e^{2\alpha\kappa\ell\xi(T)}\right)^{1/2}
\Bigg\{
\mathbb E\Bigg[
    \|X_r\|_H^{4\alpha}
    +\left(\int_r^T\|a(t,0)\|_Hdt\right)^{4\alpha}
\notag\\
&\hspace{35mm}
    +\left(\int_r^T\|b(t,0)\|_H^2dt\right)^{2\alpha}
    +\left(\int_{(r,T]}\|\gamma(t,0)\|_Hd\xi(t)\right)^{4\alpha}
\Bigg]
\Bigg\}^{1/2}.
\label{eq:FV-SEE-unconditional-estimate}
\end{align}
The weighted estimate
\eqref{eq:unbounded-xi-weighted-SEE-estimate}, including its
conditional form given \(\mathcal F_r\), also holds.
If, in addition,
\(\mathfrak E_{2\alpha\kappa,\ell}(\xi)<\infty\), then the unweighted
conditional estimate in
Proposition~\ref{prop:unbounded-xi-SEE-estimate} holds as well,
with a deterministic constant independent of \(r\).
\end{corollary}

\begin{proof}
Since \(0<\rho_{\beta,\ell}^{r}\le1\), the data satisfy the assumptions of
Theorem~\ref{thm:FV-SEE-weighted-wellposed} with \(2\alpha\) in place of \(\alpha\).
Hence there exists a solution such that
$
\bigl(\rho_{\beta,\ell}^{r}\bigr)^{1/2}X
\in \mathbb D^{4\alpha}_{\mathbb F}(r,T;H)
\cap L^{2,4\alpha}_{\mathbb F}(r,T;V).
$

The pathwise bound
\(\rho_{\beta,\ell}^{r}(t)^{-1}
 \le e^{\beta T+\kappa\ell\xi(T)}\), the Cauchy--Schwarz inequality, and the
estimate \eqref{eq:unbounded-xi-weighted-SEE-estimate} at order
\(2\alpha\) give \eqref{eq:FV-SEE-unconditional-estimate}. Thus \(X\) belongs to the
asserted unweighted \(2\alpha\)-moment space. Since \(\rho\le1\), any
other solution in this space also belongs to the weighted class at order
\(\alpha\), and Theorem~\ref{thm:FV-SEE-weighted-wellposed} gives
indistinguishability.

Applying \eqref{eq:unbounded-xi-weighted-SEE-estimate} and its
conditional form to \(X\) gives the weighted estimates.
Under \(\mathfrak E_{2\alpha\kappa,\ell}(\xi)<\infty\), the
conditional form of \eqref{eq:unbounded-xi-unweighted-SEE-estimate}
gives the unweighted conditional estimate.
\end{proof}

The next lemma extends Proposition~\ref{prop:unbounded-xi-SEE-estimate}
to \(V^*\)-valued inhomogeneous drifts, as required by the
variational equations.

\par
\begin{lemma}
	\label{lem:FV-SEE-high-order-Vstar}
	Let \(\alpha\ge1\) and let $X_r\in L^{2\alpha}(\mathcal F_r;H)$. Assume
	\((A1)\), and let \(X\) solve \eqref{eq:unbounded-xi-forward-SEE}, where
	\(a:\Omega\times[r,T]\times H\to V^*, \ b,\gamma:\Omega\times[r,T]\times H\to H\)
	are jointly progressively measurable and satisfy, for every
	\((t,\omega)\in[r,T]\times\Omega\), \(x,y\in H\),
	\begin{align}
		\|a(t,x)-a(t,y)\|_{V^*}
		+\|b(t,x)-b(t,y)\|_H
		&\le L\|x-y\|_H,
		\label{eq:FV-SEE-high-order-fg-Lipschitz}\\
		\|\gamma(t,x)-\gamma(t,y)\|_H
		&\le\ell\|x-y\|_H.
		\label{eq:FV-SEE-high-order-eta-Lipschitz}
	\end{align}
	There exist constants \(\beta_\alpha>0\) and \(C_\alpha>0\), depending only on
	\(\alpha,\delta,C_A,C_B,L\) and the norm of the embedding
	\(V\hookrightarrow H\), such that, for every \(\beta\ge\beta_\alpha\), if 
		\[
\begin{aligned}
 \bigl(\rho_{\beta,\ell}^{r}\bigr)^{1/2}a(\cdot,0)
		\in L^{2,2\alpha}_{\mathbb F}(r,T;V^*),\quad\bigl(\rho_{\beta,\ell}^{r}\bigr)^{1/2}b(\cdot,0)
		\in L^{2,2\alpha}_{\mathbb F}(r,T;H),\quad\bigl(\rho_{\beta,\ell}^{r}\bigr)^{1/2}\gamma(\cdot,0)
		\in L^{1,2\alpha}_{\mathbb F,\xi}(r,T;H),
\end{aligned}
\]
	then
	\begin{align}
		&\mathbb E\Bigg[
		\sup_{r\le t\le T}
		\bigl(\rho_{\beta,\ell}^{r}(t)\bigr)^\alpha
		\|X(t)\|_H^{2\alpha}
		+
		\left(
		\int_r^T
		\rho_{\beta,\ell}^{r}(t)\|X(t)\|_V^2dt
		\right)^\alpha
		\Bigg]
		\notag\\
		&\quad\le
		C_\alpha\mathbb E\Bigg[
		\|X_r\|_H^{2\alpha}
		+
		\left(
		\int_r^T
		\rho_{\beta,\ell}^{r}(t)
		\|a(t,0)\|_{V^*}^2dt
		\right)^\alpha
		\notag\\
		&\hspace{29mm}
		+
		\left(
		\int_r^T
		\rho_{\beta,\ell}^{r}(t)
		\|b(t,0)\|_H^2dt
		\right)^\alpha
		+
		\left(
		\int_{(r,T]}
		\bigl(\rho_{\beta,\ell}^{r}(t)\bigr)^{1/2}
		\|\gamma(t,0)\|_H\,d\xi(t)
		\right)^{2\alpha}
		\Bigg].
		\label{eq:FV-SEE-high-order-Vstar-weighted-estimate}
	\end{align}
	If, in addition,
	\(\mathfrak E_{2\alpha\kappa,\ell}(\xi)<\infty\), 	$X_r\in L^{4\alpha}(\mathcal F_r;H)$,
	and
	\begin{equation}
\begin{aligned}
 a(\cdot,0)\in L^{2,4\alpha}_{\mathbb F}(r,T;V^*),\quad
  b(\cdot,0)\in L^{2,4\alpha}_{\mathbb F}(r,T;H),\quad
 \gamma(\cdot,0)\in
			L^{1,4\alpha}_{\mathbb F,\xi}(r,T;H),
\end{aligned}
\label{eq:FV-SEE-high-order-data}
\end{equation}
	then, for every
	\(\beta\ge\max\{\beta_\alpha,\beta_{2\alpha}\}\),
	\begin{align}
		&\mathbb E\left[
		\sup_{r\le t\le T}\|X(t)\|_H^{2\alpha}
		+
		\left(
		\int_r^T\|X(t)\|_V^2dt
		\right)^\alpha
		\right]
		\notag\\
		&\quad\le
		C_{\alpha,\beta}
		\mathfrak E_{2\alpha\kappa,\ell}(\xi)^{1/2}
		\Bigg\{
		\mathbb E\Bigg[
		\|X_r\|_H^{4\alpha}
		+
		\left(
		\int_r^T\|a(t,0)\|_{V^*}^2dt
		\right)^{2\alpha}
		\notag\\
		&\hspace{34mm}
		+
		\left(
		\int_r^T\|b(t,0)\|_H^2dt
		\right)^{2\alpha}
		+
		\left(
		\int_{(r,T]}
		\|\gamma(t,0)\|_H\,d\xi(t)
		\right)^{4\alpha}
		\Bigg]
		\Bigg\}^{1/2}.
		\label{eq:FV-SEE-high-order-Vstar-estimate}
	\end{align}
	Here \(C_{\alpha,\beta}\) depends only on
	\(\alpha,\beta,T,\delta,C_A,C_B,L\) and the norm of the embedding
	\(V\hookrightarrow H\).
	
	Both
	\eqref{eq:FV-SEE-high-order-Vstar-weighted-estimate} and
	\eqref{eq:FV-SEE-high-order-Vstar-estimate} hold with
	\(\mathbb E[\cdot\mid\mathcal F_r]\) in place of each expectation,
	under the respective data conditions and with unchanged deterministic
	constants.
\end{lemma}

\begin{proof}
We follow the proof of Proposition~\ref{prop:unbounded-xi-SEE-estimate},
estimating the drift in \(V^*,V\) duality.
Write \(\rho(t)=\rho_{\beta,\ell}^{r}(t)\), and set
\[
 D={}\|X_r\|_H^2 +\int_r^T\rho(t)\|a(t,0)\|_{V^*}^2\,dt +\int_r^T\rho(t)\|b(t,0)\|_H^2\,dt +\left(\int_{(r,T]}\rho(t)^{1/2} \|\gamma(t,0)\|_H\,d\xi(t)\right)^2.
\]
The assumptions imply \(D\in L^\alpha(\mathcal F_T)\). Since
\(\rho(T)>0\) \(\mathbb P\text{-a.s.}\), the weighted data conditions also imply pathwise
square integrability of \(a(\cdot,0)\) in \(V^*\).
The Lipschitz bound on \(a\) and the c\`adl\`ag paths of \(X\)
therefore imply
\(a(\cdot,X)\in L^2(r,T;V^*)\) \(\mathbb P\text{-a.s.}\) Thus the
localized It\^o formula gives
\eqref{eq:unbounded-xi-forward-weighted-identity}, with its drift
term interpreted in \(V^*,V\) duality.

For \(x\in V\) and \(\varepsilon>0\), the Lipschitz bound on the drift
and Young's inequality give
\begin{align}
 2|\langle a(t,x),x\rangle_*|
 &\le 2L\|x\|_H\|x\|_V
        +2\|a(t,0)\|_{V^*}\|x\|_V\notag\\
 &\le 2\varepsilon\|x\|_V^2
       +\frac{L^2}{\varepsilon}\|x\|_H^2
       +\frac1\varepsilon\|a(t,0)\|_{V^*}^2.
 \label{eq:FV-SEE-high-order-Vstar-drift-bound}
\end{align}
Also \(\|B(t)x\|_H\le C\|x\|_V\) and
\(\|b(t,x)\|_H\le L\|x\|_H+\|b(t,0)\|_H\), so
\begin{equation}
2|\langle B(t)x,b(t,x)\rangle_H|
 \le\varepsilon\|x\|_V^2
       +C_{\varepsilon,L}\|x\|_H^2
       +C_\varepsilon\|b(t,0)\|_H^2.
\label{eq:SEE-Vstar-diffusion-cross-bound}
\end{equation}
Choose \(3\varepsilon\le\delta/2\). In
\eqref{eq:FV-SEE-high-order-Vstar-drift-bound}, the energy term is
\(2\varepsilon\|x\|_V^2\), whereas in
\eqref{eq:SEE-Vstar-diffusion-cross-bound} it is
\(\varepsilon\|x\|_V^2\). Combining these bounds with
\eqref{eq:SEE-coercivity} and the growth bound on
\(\|b(t,x)\|_H^2\) gives
\begin{align}
 &2\langle A(t)x+a(t,x),x\rangle_*
       +\|B(t)x+b(t,x)\|_H^2\notag\\
 &\quad\le-\frac\delta2\|x\|_V^2+C_0\|x\|_H^2
       +C_0\|a(t,0)\|_{V^*}^2+C_0\|b(t,0)\|_H^2.
 \label{eq:FV-SEE-high-order-Vstar-dt-bound}
\end{align}
The constant \(C_0\) depends only on \(\delta,C_A,L\) and the
norm of the embedding \(V\hookrightarrow H\). Since the finite-variation bound
\eqref{eq:unbounded-xi-forward-dxi-bound} uses only the Lipschitz
bound on \(\gamma\), it remains valid. We substitute
\eqref{eq:FV-SEE-high-order-Vstar-dt-bound} and
\eqref{eq:unbounded-xi-forward-dxi-bound} into
\eqref{eq:unbounded-xi-forward-weighted-identity}, apply Young's inequality to
\(2(\sup_{r\le t\le\sigma}\rho(t)\|X(t)\|_H^2)^{1/2}
\int_{(r,\sigma]}\rho(t)^{1/2}\|\gamma(t,0)\|_H\,d\xi(t)\)
and absorb the resulting multiple of the maximum. For
\(\beta\ge2C_0+1\) and every stopping time \(\sigma\in[r,T]\),
we obtain
\begin{align}
 &\sup_{r\le t\le\sigma}\rho(t)\|X(t)\|_H^2
   +\int_r^\sigma\rho(t)\|X(t)\|_V^2\,dt
   +\beta\int_r^\sigma\rho(t)\|X(t)\|_H^2\,dt\notag\\
 &\quad+(\kappa-2)\ell\int_{(r,\sigma]}
                    \rho(t)\|X(t-)\|_H^2\,d\xi(t)
 \le C\left[D+\sup_{r\le t\le\sigma}|M(t)|\right].
 \label{eq:FV-SEE-high-order-Vstar-pathwise}
\end{align}
Here \(M\) is the stochastic integral defined in
\eqref{eq:SEE-weighted-energy-martingale}. Its integrand
contains \(B,b,X\), but not \(a\), so
\eqref{eq:unbounded-xi-forward-bracket-bound} and
\eqref{eq:unbounded-xi-forward-BDG-bound} remain unchanged.

Applying the localization and absorption argument leading to
\eqref{eq:unbounded-xi-forward-stopped-estimate} to
\eqref{eq:FV-SEE-high-order-Vstar-pathwise}, with the present
\(D\in L^\alpha\), gives
\eqref{eq:FV-SEE-high-order-Vstar-weighted-estimate}.
For \(E\in\mathcal F_r\), these martingale estimates apply to \(I_EM\),
whose bracket is \(I_E\langle M\rangle\). Testing against all such
\(E\) gives the conditional form of
\eqref{eq:FV-SEE-high-order-Vstar-weighted-estimate}.

Under the additional data assumptions,
\(D\in L^{2\alpha}\). The conditional estimate at order
\(2\alpha\), the bound
\(\rho(t)^{-1}\le e^{\beta T+\kappa\ell(\xi(T)-\xi(r))}\),
and the conditional Cauchy--Schwarz inequality give
\eqref{eq:FV-SEE-high-order-Vstar-estimate} and its conditional
form, as in \eqref{eq:SEE-conditional-unweighting}, with the
present definition of \(D\).
Taking expectations and applying Jensen's inequality to the
square root yields the unconditional bound.
\end{proof}

\begin{remark}
In \cite{GK82}, post-jump evaluation leads to the implicit relation
\(x(t)=x(t-)+\gamma(t,x(t))\Delta\xi(t)\) and an accompanying
small-jump condition. Here the pre-jump coefficient gives the explicit
update \(x(t)=x(t-)+\gamma(t,x(t-))\Delta\xi(t)\). The weighted
estimate therefore requires no smallness restriction on individual
jumps. 
\end{remark}
\subsection{Backward stochastic evolution equations with finite-variation drivers}
\label{subsec:FV-BSEE}
We next study backward stochastic evolution equations driven jointly by
Lebesgue measure and an adapted nondecreasing c\`adl\`ag process
\(\xi\), with \(\xi(0)=0\) and \(\xi(T)<\infty\) \(\mathbb P\text{-a.s.}\):
\begin{equation}
\begin{aligned}
    p(t)
    &=
    \varphi
    +
    \int_t^T
        \big[
            M(s)p(s)+N(s)q(s)+g(s,p(s),q(s))
        \big]\,ds
    -
    \int_t^T q(s)\,dW(s)
    \\
    &\quad
    +
    \int_{(t,T]}
        \psi(s,p(s))\,d\xi(s),
    \quad t\in[0,T].
\end{aligned}
\label{eq:FV-BSEE}
\end{equation}
The linear operators have the following domains and ranges:
\[
M:[0,T]\times\Omega\to\mathfrak L(V;V^*),\quad
N:[0,T]\times\Omega\to\mathfrak L(H;V^*),\quad
N^*:[0,T]\times\Omega\to\mathfrak L(V;H).
\]
Here \(N^*(t)\) is the adjoint of \(N(t)\) relative to the
Gelfand triple. The mappings
\(g:\Omega\times[0,T]\times H\times H\to H\) and
\(\psi:\Omega\times[0,T]\times H\to H\) are the generators of the
Lebesgue and finite-variation terms, respectively. The terminal datum
\(\varphi\) is \(\mathcal F_T\)-measurable and \(H\)-valued.

At a jump time \(t\), \eqref{eq:FV-BSEE} gives
\begin{equation}\label{eq:backward jump indentity}
    p(t-)=p(t)+\psi(t,p(t))\Delta\xi(t),
    \quad
    \Delta p(t)=-\psi(t,p(t))\Delta\xi(t).
\end{equation}
Thus, in contrast to the forward equation, the finite-variation coefficient
is evaluated at the post-jump value of the backward state.

\begin{definition}
\label{def:FV-BSEE-solution}
A pair \((p,q)\in
(\mathbb D^2_{\mathbb F}(0,T;H)\cap L^2_{\mathbb F}(0,T;V))
\times L^2_{\mathbb F}(0,T;H)\) is called a solution of
\eqref{eq:FV-BSEE} if, for each \(t\in[0,T]\) and \(\eta\in V\),
the following identity holds outside a \(\mathbb P\)-null set:
\begin{align}
    \langle \eta,p(t)\rangle_H
    &=
    \langle \eta,\varphi\rangle_H
    +
    \int_t^T
        \big\langle
            M(s)p(s)+N(s)q(s)+g(s,p(s),q(s)),
            \eta
        \big\rangle_*\,ds
    \notag\\
    &\quad
    +
    \int_{(t,T]}
        \langle \psi(s,p(s)),\eta\rangle_H\,d\xi(s)
    -
    \int_t^T
        \langle \eta,q(s)\rangle_H\,dW(s).
    \label{eq:FV-BSEE-weak-form}
\end{align} 
\end{definition}

We make the following assumptions on the operators and generators.
\begin{description}
\item[\((B1)\)] For every \(y,z\in V\) and \(h\in H\), the scalar
processes
\((t,\omega)\mapsto\langle M(t)y,z\rangle_*, \ (t,\omega)\mapsto\langle N^*(t)y,h\rangle_H\)
are progressively measurable. There exist constants \(\delta_B>0\),
\(\lambda\ge0\), and \(C_M>0\) such that, for
all \((t,\omega)\in[0,T]\times\Omega\) and \(y\in V\):
\begin{equation}
    2\langle M(t)y,y\rangle_*
    +
    \|N^*(t)y\|_H^2
    \le
    -\delta_B\|y\|_V^2
    +
    \lambda\|y\|_H^2,
    \label{eq:BSEE-coercivity}
\end{equation}
and
\begin{equation}
    \|M(t)y\|_{V^*}\le C_M\|y\|_V.
    \label{eq:BSEE-M-bound}
\end{equation}

\item[\((B2)\)] The mapping
\(g:\Omega\times[0,T]\times H\times H\to H\) is
\(\operatorname{Prog}\otimes\mathcal B(H)\otimes\mathcal B(H)\)-measurable.
There exists \(L>0\) such that, for every
\((t,\omega)\in[0,T]\times\Omega\),
\[
\|g(t,y,z)-g(t,y',z')\|_H
\le L(\|y-y'\|_H+\|z-z'\|_H),
\quad y,y',z,z'\in H.
\]
\item[\((B3)\)] The mapping
\(\psi:\Omega\times[0,T]\times H\to H\) is
\(\operatorname{Prog}\otimes\mathcal B(H)\)-measurable.
There exists \(\ell\ge0\) such that, for every
\((t,\omega)\in[0,T]\times\Omega\),
\[
\|\psi(t,y)-\psi(t,y')\|_H\le \ell\|y-y'\|_H,
\quad y,y'\in H.
\]
\end{description}

Assumption \((B1)\) also gives the operator bound needed for
\(N(t)q(t)\). By \eqref{eq:BSEE-coercivity},
\eqref{eq:BSEE-M-bound}, and the continuous embedding \(V\hookrightarrow H\),
we obtain, for every \((t,\omega)\),
\[
\|N^*(t)y\|_H^2
\le
-2\langle M(t)y,y\rangle_*
+\lambda\|y\|_H^2
\le C\|y\|_V^2,
\quad y\in V.
\]
By duality,
\(\|N(t)h\|_{V^*}\le C\|h\|_H, \  h\in H.\)

For \(\kappa>2\) and \(\beta>0\), set
\(\theta_{\beta,\ell}^{0}(t)=e^{\beta t+\kappa\ell\xi(t)}\), and define
\begin{equation}
 \|(p,q)\|_{\theta_{\beta,\ell}^{0}}^2
 :=\mathbb E\left[
 \sup_{0\le t\le T}\theta_{\beta,\ell}^{0}(t)\|p(t)\|_H^2
 +\int_0^T\theta_{\beta,\ell}^{0}(t)
       (\|p(t)\|_V^2+\|q(t)\|_H^2)dt\right].
\label{eq:BSEE-theta-norm}
\end{equation}

We first estimate the Lebesgue terms that enter the energy inequality.
For every \(\varepsilon>0\),
\begin{align*}
    2\langle N(s)q(s),p(s)\rangle_*
    =
    2\langle q(s),N^*(s)p(s)\rangle_H
    &\le
    \frac{1}{1+\varepsilon}\|q(s)\|_H^2
    +
    (1+\varepsilon)\|N^*(s)p(s)\|_H^2.
\end{align*}
Moreover,
\(-2\varepsilon\langle M(s)p(s),p(s)\rangle_* \le 2\varepsilon C_M\|p(s)\|_V^2.\)
Therefore, by \eqref{eq:BSEE-coercivity},
\begin{align*}
    &2\langle M(s)p(s)+N(s)q(s),p(s)\rangle_*
    \\
    &\le
    (1+\varepsilon)
    \Big[
        2\langle M(s)p(s),p(s)\rangle_*
        +
        \|N^*(s)p(s)\|_H^2
    \Big] 
    -2\varepsilon\langle M(s)p(s),p(s)\rangle_*
    +
    \frac{1}{1+\varepsilon}\|q(s)\|_H^2
    \\
    &\le
    -\big((1+\varepsilon)\delta_B-2\varepsilon C_M\big)\|p(s)\|_V^2
    +
    (1+\varepsilon)\lambda\|p(s)\|_H^2
    +
    \frac{1}{1+\varepsilon}\|q(s)\|_H^2.
\end{align*}
We choose \(\varepsilon>0\) so that
\((1+\varepsilon)\delta_B-2\varepsilon C_M\ge\delta_B/2\)
and set \(c_0:=\varepsilon/(1+\varepsilon)\). Then
\begin{equation}
    2\langle M(s)p(s)+N(s)q(s),p(s)\rangle_*
    \le
    -\frac{\delta_B}{2}\|p(s)\|_V^2
    +
    C\|p(s)\|_H^2
    +
    (1-c_0)\|q(s)\|_H^2.
    \label{eq:BSEE-linear-part-final}
\end{equation}
By \((B2)\) and Young's inequality,
\begin{align}
    2\langle g(s,p(s),q(s)),p(s)\rangle_H
    &\le
    \frac{\beta}{4}\|p(s)\|_H^2
    +
    \frac{4}{\beta}\|g(s,0,0)\|_H^2
    +C_L\|p(s)\|_H^2
    +\frac{c_0}{2}\|q(s)\|_H^2.
    \label{eq:BSEE-generator-estimate}
\end{align}

The following proposition gives weighted and unweighted a priori estimates.
\begin{proposition}
\label{prop:unbounded-xi-BSEE-estimate}
\label{prop:FV-BSEE-apriori}
Let \(\alpha\ge1\) and \(\kappa>2\), and assume
\((B1)\)--\((B3)\).
There exist constants \(\beta_\alpha\ge1\) and \(C_\alpha>0\), depending only on \(\alpha,\kappa,\delta_B,\lambda,C_M,L\) and the norm of the embedding
\(V\hookrightarrow H\), such that
for every \(\beta\ge\beta_\alpha\), the following estimates hold.
Let \((p,q)\) solve \eqref{eq:FV-BSEE} with
\(\|(p,q)\|_{\theta_{\beta,\ell}^{0}}<\infty\).  Let \(\psi_1,\psi_2\) be progressively
measurable \(H\)-valued processes satisfying
\(\psi(t,0)=\psi_1(t)+\psi_2(t)\) for \(d\xi\)-almost every \(t\),
\(\mathbb P\text{-a.s.}\) Suppose that
\[
\begin{aligned}
 &\theta_{\beta,\ell}^{0}(T)^{1/2}\varphi\in L^{2\alpha}(\mathcal F_T;H),
 \quad \bigl(\theta_{\beta,\ell}^{0}\bigr)^{1/2}g(\cdot,0,0)
       \in L^{2,2\alpha}_{\mathbb F}(0,T;H),\\
 &\bigl(\theta_{\beta,\ell}^{0}\bigr)^{1/2}\psi_1\in L^{1,2\alpha}_{\mathbb F,\xi}(0,T;H),
 \quad \bigl(\theta_{\beta,\ell}^{0}\bigr)^{1/2}\psi_2
       \in L^{2,2\alpha}_{\mathbb F,\xi}(0,T;H).
\end{aligned}
\]
When \(\ell=0\), take \(\psi_2=0\) and omit all terms with
\((\kappa-2)\ell\) in the denominator below and in the proof. Then
\begin{align}
&\mathbb E\Bigg[
 \sup_{0\le t\le T}\theta_{\beta,\ell}^{0}(t)^\alpha\|p(t)\|_H^{2\alpha}
 +\left(
    \int_0^T\theta_{\beta,\ell}^{0}(t)
       \bigl(\|p(t)\|_V^2+\|q(t)\|_H^2
                         +\beta\|p(t)\|_H^2\bigr)\,dt
  \right)^\alpha
 \Bigg]\notag\\
&\quad\le C_\alpha\,\mathbb E\Bigg[\Bigg(
 \theta_{\beta,\ell}^{0}(T)\|\varphi\|_H^2
 +\frac1\beta\int_0^T\theta_{\beta,\ell}^{0}(t)\|g(t,0,0)\|_H^2\,dt\notag\\
&\hspace{28mm}
 +\left(
    \int_{(0,T]}\theta_{\beta,\ell}^{0}(t)^{1/2}\|\psi_1(t)\|_H\,d\xi(t)
  \right)^2
 +\frac1{(\kappa-2)\ell}
    \int_{(0,T]}\theta_{\beta,\ell}^{0}(t)\|\psi_2(t)\|_H^2\,d\xi(t)
 \Bigg)^\alpha\Bigg].
\label{eq:unbounded-xi-weighted-BSEE-estimate}
\end{align}

If, in addition,
\[
\begin{aligned}
 \varphi\in L^{4\alpha}(\mathcal F_T;H),
 \quad g(\cdot,0,0)\in L^{2,4\alpha}_{\mathbb F}(0,T;H),\quad
 \psi_1\in L^{1,4\alpha}_{\mathbb F,\xi}(0,T;H),
 \quad\psi_2\in L^{2,4\alpha}_{\mathbb F,\xi}(0,T;H),
\end{aligned}
\]
and \(\mathbb E e^{2\alpha\kappa\ell\xi(T)}<\infty\), then
\begin{align}
&\mathbb E\Bigg[
    \sup_{0\le t\le T}\|p(t)\|_H^{2\alpha}
    +\left(\int_0^T\|p(t)\|_V^2\,dt\right)^\alpha
    +\left(\int_0^T\|q(t)\|_H^2\,dt\right)^\alpha
\Bigg]\notag\\
&\quad\le
C_{\alpha,\beta,\kappa}\,
\bigl(\mathbb E e^{2\alpha\kappa\ell\xi(T)}\bigr)^{1/2}
\Bigg\{\mathbb E\Bigg[
    \|\varphi\|_H^{4\alpha}
    +\frac1{\beta^{2\alpha}}
       \left(\int_0^T\|g(t,0,0)\|_H^2\,dt\right)^{2\alpha}
\notag\\
&\hspace{31mm}
    +\left(\int_{(0,T]}\|\psi_1(t)\|_H\,d\xi(t)\right)^{4\alpha}
    +\frac1{((\kappa-2)\ell)^{2\alpha}}
       \left(\int_{(0,T]}\|\psi_2(t)\|_H^2\,d\xi(t)\right)^{2\alpha}
\Bigg]\Bigg\}^{1/2}.
\label{eq:unbounded-xi-unweighted-BSEE-estimate}
\end{align}
The constant \(C_{\alpha,\beta,\kappa}\) depends only on
\(\alpha,\beta,\kappa,\delta_B,\lambda,C_M,L,T\) and the norm of the embedding
\(V\hookrightarrow H\). 
\end{proposition}

\begin{proof}
Set \(\theta=\theta_{\beta,\ell}^{0}\). To derive
\eqref{eq:unbounded-xi-BSEE-basic-energy}, we estimate the Lebesgue integral,
the \(d\xi^c\)-integral, and the jump sum separately. We then use
conditional estimates to control the higher moments of the maximum
and localization to control those of the energy integral.

Define
\begin{align*}
D
&=\theta(T)\|\varphi\|_H^2
  +\frac1\beta\int_0^T\theta(s)\|g(s,0,0)\|_H^2\,ds\\
&\quad+
  \left(\int_{(0,T]}\theta(s)^{1/2}\|\psi_1(s)\|_H\,d\xi(s)\right)^2
  +\frac1{(\kappa-2)\ell}
       \int_{(0,T]}\theta(s)\|\psi_2(s)\|_H^2\,d\xi(s).
\end{align*}
By the assumptions on the data, \(D\in L^\alpha(\mathcal F_T)\).

We first verify the integrability needed to apply the weighted It\^o
formula and take expectations. Define
\[
\mathcal M(t)=\int_0^t\theta(s-)
                  \langle p(s-),q(s)\rangle_H\,dW(s).
\]
For \(0\le t\le u\le T\), the quadratic variation of
\(\mathcal M\) satisfies
\begin{equation}
\begin{split}
\langle\mathcal M\rangle(u)-\langle\mathcal M\rangle(t)
&=\int_t^u\theta(s-)^2
             |\langle p(s-),q(s)\rangle_H|^2\,ds\\
&\le
 \left(\sup_{t\le v\le u}\theta(v)\|p(v)\|_H^2\right)
 \int_t^u\theta(s)\|q(s)\|_H^2\,ds.
\end{split}
\label{eq:BSEE-martingale-bracket-bound}
\end{equation}
By \eqref{eq:BSEE-theta-norm}, the assumption
\(\|(p,q)\|_\theta<\infty\) gives
\[
\mathbb E\langle\mathcal M\rangle(T)^{1/2}
\le
\left(\mathbb E\sup_{0\le t\le T}\theta(t)\|p(t)\|_H^2\right)^{1/2}
\left(\mathbb E\int_0^T\theta(s)\|q(s)\|_H^2\,ds\right)^{1/2}
<\infty.
\]
By the \(H^1\)-martingale criterion, \(\mathcal M\) has an
integrable maximum and is a uniformly integrable martingale.

To verify
\(\psi(\cdot,p)\in L^{1,2}_{\mathbb F,\xi}(0,T;H)\), suppose
first that \(\ell>0\). The finite-variation chain rule gives,
for \(c>0\),
\begin{align*}
c\int_{(0,T]}e^{-c\xi(s)}\,d\xi(s)
&=1-e^{-c\xi(T)}-
\sum_{0<s\le T}e^{-c\xi(s)}
       \bigl(e^{c\Delta\xi(s)}-1-c\Delta\xi(s)\bigr)
\le1.
\end{align*}
Taking \(c=\kappa\ell/2\) and \(c=\kappa\ell\), and using
\(e^{-\beta s}\le1\), we obtain
\begin{equation}
\ell\int_{(0,T]}\theta(s)^{-1/2}\,d\xi(s)\le\frac2\kappa,
\quad
\int_{(0,T]}\theta(s)^{-1}\,d\xi(s)\le\frac1{\kappa\ell}.
\label{eq:BSEE-reciprocal-weight-integrals}
\end{equation}
The Cauchy--Schwarz inequality with respect to \(d\xi\) gives
\[
\mathbb E\left(\int_{(0,T]}\|\psi_2(s)\|_H\,d\xi(s)\right)^2
\le\frac1{\kappa\ell}
       \mathbb E\int_{(0,T]}\theta(s)\|\psi_2(s)\|_H^2\,d\xi(s)
<\infty.
\]
By \((B3)\), the decomposition of \(\psi(\cdot,0)\), and the
first inequality in \eqref{eq:BSEE-reciprocal-weight-integrals},
\begin{align*}
\int_{(0,T]}\|\psi(s,p(s))\|_H\,d\xi(s)
&\le\int_{(0,T]}
       \bigl(\|\psi_1(s)\|_H+\|\psi_2(s)\|_H\bigr)\,d\xi(s)\\
&\quad+\frac2\kappa
             \sup_{0\le s\le T}\theta(s)^{1/2}\|p(s)\|_H.
\end{align*}
Every term on the right belongs to \(L^2(\mathcal F_T)\).

We apply \eqref{eq:positive-weight-backward-FV-Ito-M} to
\((p,q)\) on \([t,u]\), with \(F=\xi\) and \(\zeta=\psi(\cdot,p)\), to obtain
\begin{align}
&\theta(t)\|p(t)\|_H^2
 +\int_t^u\theta(s)\Big[
       \|q(s)\|_H^2+\beta\|p(s)\|_H^2\notag\\
&\hspace{35mm}
       -2\langle M(s)p(s)+N(s)q(s)
                    +g(s,p(s),q(s)),p(s)\rangle_*
                    \Big]\,ds\notag\\
&=\theta(u)\|p(u)\|_H^2
  +\int_{(t,u]}\theta(s)
       \bigl[2\langle p(s),\psi(s,p(s))\rangle_H
                      -\kappa\ell\|p(s)\|_H^2\bigr]\,d\xi^c(s)
                      \notag\\
&\quad+\sum_{t<s\le u}
       \bigl[\theta(s-)\|p(s-)\|_H^2
                         -\theta(s)\|p(s)\|_H^2\bigr]
       -2\bigl(\mathcal M(u)-\mathcal M(t)\bigr).
\label{eq:BSEE-weighted-energy-identity}
\end{align}
The jump relation \eqref{eq:backward jump indentity} also gives
\[
\sum_{0<s\le T}\|\Delta p(s)\|_H
\le\int_{(0,T]}\|\psi(s,p(s))\|_H\,d\xi(s)<\infty.
\]

For the Lebesgue integrand on the left-hand side of
\eqref{eq:BSEE-weighted-energy-identity}, we combine
\eqref{eq:BSEE-linear-part-final} and
\eqref{eq:BSEE-generator-estimate}. For sufficiently large \(\beta\),
we obtain
\begin{align}
&\|q(s)\|_H^2+\beta\|p(s)\|_H^2
 -2\langle M(s)p(s)+N(s)q(s)
                     +g(s,p(s),q(s)),p(s)\rangle_*\notag\\
&\quad\ge
 c\bigl(\|p(s)\|_V^2+\|q(s)\|_H^2+\beta\|p(s)\|_H^2\bigr)
 -\frac C\beta\|g(s,0,0)\|_H^2.
\label{eq:BSEE-coercive-integrand-bound}
\end{align}
Here \(c>0\), \(C\), and the lower bound on \(\beta\) depend
only on the operator constants, \(L\), and the norm of the embedding.

We next estimate the continuous and jump terms in
\eqref{eq:BSEE-weighted-energy-identity}. For the \(d\xi^c\)-integrand,
\((B3)\) and Young's inequality give, when \(\ell>0\),
\begin{align}
&2\langle p(s),\psi(s,p(s))\rangle_H
                         -\kappa\ell\|p(s)\|_H^2\notag\\
&\quad\le
 2\|p(s)\|_H\|\psi_1(s)\|_H
 +2\|p(s)\|_H\|\psi_2(s)\|_H
                  -(\kappa-2)\ell\|p(s)\|_H^2\notag\\
&\quad\le
 2\|p(s)\|_H\|\psi_1(s)\|_H
 +\frac1{(\kappa-2)\ell}\|\psi_2(s)\|_H^2.
\label{eq:BSEE-continuous-split-bound}
\end{align}
For the jump sum, fix \(s\) with \(\Delta\xi(s)>0\) and assume
\(\ell>0\). The backward jump identity
\eqref{eq:backward jump indentity} gives
\begin{align*}
p(s-)-\psi_1(s)\Delta\xi(s)
&=p(s)+[\psi(s,p(s))-\psi(s,0)]\Delta\xi(s)+\psi_2(s)\Delta\xi(s),
\end{align*}
and hence
\[
\|p(s-)-\psi_1(s)\Delta\xi(s)\|_H
\le(1+\ell\Delta\xi(s))\|p(s)\|_H
                   +\Delta\xi(s)\|\psi_2(s)\|_H.
\]
The identity
\(\|x\|_H^2=\|x-h\|_H^2+2\langle x,h\rangle_H-\|h\|_H^2\),
with \(x=p(s-)\) and \(h=\psi_1(s)\Delta\xi(s)\), together with
\(\theta(s)=\theta(s-)e^{\kappa\ell\Delta\xi(s)}\), gives
\begin{align}
&\theta(s-)\|p(s-)\|_H^2-\theta(s)\|p(s)\|_H^2\notag\\
&\quad\le
 2\theta(s-)\langle p(s-),\psi_1(s)\rangle_H\Delta\xi(s)
 -\theta(s-)\|\psi_1(s)\|_H^2(\Delta\xi(s))^2\notag\\
&\quad+\theta(s-)\Big[
 \bigl((1+\ell\Delta\xi(s))\|p(s)\|_H
                  +\Delta\xi(s)\|\psi_2(s)\|_H\bigr)^2
 -e^{\kappa\ell\Delta\xi(s)}\|p(s)\|_H^2\Big].
\label{eq:BSEE-jump-before-square-completion}
\end{align}
Put \(d=\Delta\xi(s)>0\). By
\eqref{eq:unbounded-xi-exponential-flow-lower-bound},
\[
a:=e^{\kappa\ell d}-(1+\ell d)^2
\ge(\kappa-2)\ell d>0.
\]
Completing the square gives the scalar inequality
\(-ax^2+2bx\le b^2/a\), valid for \(a>0\). Applying it with
\(x=\|p(s)\|_H\) and \(b=(1+\ell d)d\|\psi_2(s)\|_H\)
yields
\begin{align}
&\bigl((1+\ell\Delta\xi(s))\|p(s)\|_H
                   +\Delta\xi(s)\|\psi_2(s)\|_H\bigr)^2
 -e^{\kappa\ell\Delta\xi(s)}\|p(s)\|_H^2\notag\\
&\quad\le
 \frac{e^{\kappa\ell\Delta\xi(s)}(\Delta\xi(s))^2}
      {e^{\kappa\ell\Delta\xi(s)}-(1+\ell\Delta\xi(s))^2}
             \|\psi_2(s)\|_H^2.
\label{eq:BSEE-jump-square-completion-bound}
\end{align}
Substituting \eqref{eq:BSEE-jump-square-completion-bound} into
\eqref{eq:BSEE-jump-before-square-completion}, using
\(-\theta(s-)\|\psi_1(s)\|_H^2(\Delta\xi(s))^2\le0\)
and \(\theta(s-)e^{\kappa\ell\Delta\xi(s)}=\theta(s)\), we obtain
\begin{align}
&\theta(s-)\|p(s-)\|_H^2-\theta(s)\|p(s)\|_H^2\notag\\
&\quad\le
 2\theta(s-)\langle p(s-),\psi_1(s)\rangle_H\Delta\xi(s)
 +\frac{\theta(s)(\Delta\xi(s))^2}
        {e^{\kappa\ell\Delta\xi(s)}-(1+\ell\Delta\xi(s))^2}
              \|\psi_2(s)\|_H^2\notag\\
&\quad\le
 2\theta(s-)\langle p(s-),\psi_1(s)\rangle_H\Delta\xi(s)
 +\frac{\theta(s)\Delta\xi(s)}{(\kappa-2)\ell}
              \|\psi_2(s)\|_H^2.
\label{eq:BSEE-split-jump-estimate}
\end{align}

When \(\ell=0\), \(\psi(s,p(s))=\psi_1(s)\) and
\(\theta(s)=e^{\beta s}\) is continuous. Thus the left-hand
side of \eqref{eq:BSEE-continuous-split-bound} is
\(2\langle p(s),\psi_1(s)\rangle_H\), and
\begin{align*}
\theta(s-)\|p(s-)\|_H^2-\theta(s)\|p(s)\|_H^2
&=2\theta(s-)\langle p(s-),\psi_1(s)\rangle_H\Delta\xi(s)\\
&\quad-\theta(s-)\|\psi_1(s)\|_H^2(\Delta\xi(s))^2\\
&\leq 2\theta(s-)\langle p(s-),\psi_1(s)\rangle_H\Delta\xi(s)
\end{align*}
Consequently, \eqref{eq:BSEE-continuous-split-bound} and
\eqref{eq:BSEE-split-jump-estimate} remain valid with their
\(\psi_2\)-terms omitted.

We now integrate \eqref{eq:BSEE-continuous-split-bound} against
\(\theta(s)\,d\xi^c(s)\) over \((t,u]\), and sum
\eqref{eq:BSEE-split-jump-estimate} over \(t<s\le u\).
For these jump times, the left limits of \(\theta^{1/2}p\) give
\[
\theta(s-)\|p(s-)\|_H
\le
\left(\sup_{t\le v\le u}\theta(v)\|p(v)\|_H^2\right)^{1/2}
\theta(s)^{1/2},
\]
where we use \(\theta(s-)\le\theta(s)\).
Combining the integrals against the continuous part of \(\xi\)
with the sums over its atoms, we obtain
\begin{align}
&\int_{(t,u]}\theta(s)
 \bigl[2\langle p(s),\psi(s,p(s))\rangle_H
                          -\kappa\ell\|p(s)\|_H^2\bigr]\,d\xi^c(s)
                          +\sum_{t<s\le u}
 \bigl[\theta(s-)\|p(s-)\|_H^2-\theta(s)\|p(s)\|_H^2\bigr]
                          \notag\\
&\le
 2\left(\sup_{t\le v\le u}\theta(v)\|p(v)\|_H^2\right)^{1/2}
     \int_{(t,u]}\theta(s)^{1/2}\|\psi_1(s)\|_H\,d\xi(s)
                        +\frac1{(\kappa-2)\ell}
     \int_{(t,u]}\theta(s)\|\psi_2(s)\|_H^2\,d\xi(s).
\label{eq:BSEE-combined-Stieltjes-bound}
\end{align}
Setting \(u=T\) in \eqref{eq:BSEE-weighted-energy-identity} and applying
\eqref{eq:BSEE-coercive-integrand-bound} to the Lebesgue term and
\eqref{eq:BSEE-combined-Stieltjes-bound} to the finite-variation terms,
we obtain, by the definition of \(D\),
\begin{equation}
\begin{split}
&\theta(t)\|p(t)\|_H^2
 +c\int_t^T\theta(s)
       \bigl(\|p(s)\|_V^2+\|q(s)\|_H^2+\beta\|p(s)\|_H^2\bigr)\,ds\\
&\quad\le CD
 +C\left(\sup_{t\le v\le T}\theta(v)\|p(v)\|_H^2\right)^{1/2}D^{1/2}
 -2\bigl(\mathcal M(T)-\mathcal M(t)\bigr).
\end{split}
\label{eq:unbounded-xi-BSEE-basic-energy}
\end{equation}

We first close the estimate at order one. All expectations below
are finite by \eqref{eq:BSEE-theta-norm}, \(D\in L^1\), and
\[
\beta\int_0^T\theta(s)\|p(s)\|_H^2\,ds
\le\beta T\sup_{0\le s\le T}\theta(s)\|p(s)\|_H^2.
\]
Fix \(r\in[0,T]\). We take expectations in
\eqref{eq:unbounded-xi-BSEE-basic-energy} at \(t=r\), use
\(\mathbb E[\mathcal M(T)-\mathcal M(r)]=0\), and apply
Young's inequality to
\((\sup_{r\le u\le T}\theta(u)\|p(u)\|_H^2)^{1/2}D^{1/2}\).
For every \(\varepsilon>0\), this gives
\begin{align}
&\mathbb E\int_r^T\theta(s)
       \bigl(\|p(s)\|_V^2+\|q(s)\|_H^2+\beta\|p(s)\|_H^2\bigr)\,ds
\le\varepsilon\mathbb E
       \sup_{r\le t\le T}\theta(t)\|p(t)\|_H^2
       +C_\varepsilon\mathbb ED.
\label{eq:BSEE-expected-energy-bound}
\end{align}
To estimate \(\mathbb E\sup_{r\le t\le T}
\theta(t)\|p(t)\|_H^2\), we discard the nonnegative energy integral
in \eqref{eq:unbounded-xi-BSEE-basic-energy} and take the supremum
and expectation. For the martingale term, we use
\(\sup_{r\le t\le T}|\mathcal M(T)-\mathcal M(t)|
\le2\sup_{r\le t\le T}|\mathcal M(t)-\mathcal M(r)|\)
and apply the BDG inequality with \eqref{eq:BSEE-martingale-bracket-bound}.
We then apply Young's inequality to the products with \(D^{1/2}\)
and with \((\int_r^T\theta(s)\|q(s)\|_H^2\,ds)^{1/2}\). This gives
\begin{align}
\mathbb E\sup_{r\le t\le T}\theta(t)\|p(t)\|_H^2
&\le C\mathbb ED
 +C\mathbb E\left[
    \left(\sup_{r\le t\le T}\theta(t)\|p(t)\|_H^2\right)^{1/2}D^{1/2}
              \right]\notag\\
&\quad+C\mathbb E\left[
    \left(\sup_{r\le t\le T}\theta(t)\|p(t)\|_H^2\right)^{1/2}
    \left(\int_r^T\theta(s)\|q(s)\|_H^2\,ds\right)^{1/2}
              \right]\notag\\
&\le\frac12\mathbb E
       \sup_{r\le t\le T}\theta(t)\|p(t)\|_H^2
       +C\mathbb E\int_r^T\theta(s)\|q(s)\|_H^2\,ds
       +C\mathbb ED.
\label{eq:BSEE-expected-maximum-bound}
\end{align}
Using \eqref{eq:BSEE-expected-energy-bound} to estimate the integral
involving \(q\) in \eqref{eq:BSEE-expected-maximum-bound}, we obtain
\[
\mathbb E\sup_{r\le t\le T}\theta(t)\|p(t)\|_H^2
\le\left(\frac12+C\varepsilon\right)
       \mathbb E\sup_{r\le t\le T}\theta(t)\|p(t)\|_H^2
       +C_\varepsilon\mathbb ED.
\]
We choose \(C\varepsilon\le1/4\), absorb the maximum term into the
left-hand side, and substitute the resulting bound into
\eqref{eq:BSEE-expected-energy-bound} to obtain
\begin{equation}
\begin{split}
\mathbb E\Bigg[
 &\sup_{r\le t\le T}\theta(t)\|p(t)\|_H^2\\
 &+\int_r^T\theta(s)
       \bigl(\|p(s)\|_V^2+\|q(s)\|_H^2+\beta\|p(s)\|_H^2\bigr)\,ds
              \Bigg]
\le C\mathbb ED.
\end{split}
\label{eq:unbounded-xi-BSEE-closed-L2}
\end{equation}
With \(r=0\), this proves
\eqref{eq:unbounded-xi-weighted-BSEE-estimate} for \(\alpha=1\).

For \(\alpha>1\), we first derive a conditional version of
\eqref{eq:unbounded-xi-BSEE-closed-L2} to control the higher moment
of the maximum.
For \(E\in\mathcal F_r\), we apply the argument leading to
\eqref{eq:unbounded-xi-BSEE-closed-L2} on \([r,T]\), inserting
\(I_E\) in each expectation. The martingale estimate applies to the
\(H^1\)-martingale
\(I_E(\mathcal M(t)-\mathcal M(r))\), \(r\le t\le T\).
Absorbing the weighted maximum term gives
\[
\mathbb E[I_E\theta(r)\|p(r)\|_H^2]
\le C\mathbb E[I_ED],\quad E\in\mathcal F_r.
\]
Hence
\begin{equation}
\theta(r)\|p(r)\|_H^2
\le C\mathbb E[D\mid\mathcal F_r],\quad\mathbb P\text{-a.s.}
\label{eq:BSEE-conditional-pointwise-bound}
\end{equation}
We take the c\`adl\`ag version of
\(t\mapsto\mathbb E[D\mid\mathcal F_t]\).
Outside one null set, \eqref{eq:BSEE-conditional-pointwise-bound}
holds simultaneously for rational \(r\) and \(r=T\).
The right continuity of \(\theta\|p\|_H^2\) and of the
conditional-expectation martingale extends the inequality to every
\(r\in[0,T]\). Doob's \(L^\alpha\)-inequality applied to
\(t\mapsto\mathbb E[D\mid\mathcal F_t]\) yields
\begin{align}
\mathbb E\left(
  \sup_{0\le t\le T}\theta(t)\|p(t)\|_H^2\right)^\alpha
&\le C^\alpha\,
 \mathbb E\sup_{0\le t\le T}
          |\mathbb E[D\mid\mathcal F_t]|^\alpha\notag\\
&\le C^\alpha\left(\frac\alpha{\alpha-1}\right)^\alpha
          \mathbb ED^\alpha
 \le C_\alpha\mathbb ED^\alpha.
\label{eq:unbounded-xi-BSEE-maximal-higher-moment}
\end{align}

We now estimate
\(\mathbb E(\int_0^T\theta(s)
 (\|p(s)\|_V^2+\|q(s)\|_H^2+\beta\|p(s)\|_H^2)\,ds)^\alpha\).
Set
\[
\tau_n=\inf\{t\in[0,T]:\langle\mathcal M\rangle(t)\ge n\}\wedge T,
\quad \inf\varnothing:=+\infty.
\]
Then
\(\langle\mathcal M\rangle(\tau_n)\le n\). Thus the BDG inequality implies
\(\mathbb E|\mathcal M(\tau_n)|^\alpha\le C_\alpha n^{\alpha/2}<\infty\).
Also, \(\langle\mathcal M\rangle(T)<\infty\) almost surely,
so \(\tau_n=T\) eventually on almost every path.

We evaluate \eqref{eq:BSEE-weighted-energy-identity} at
\((t,u)=(0,\tau_n)\) and apply
\eqref{eq:BSEE-coercive-integrand-bound} on \([0,\tau_n]\)
and \eqref{eq:BSEE-combined-Stieltjes-bound} on \((0,\tau_n]\).
Dropping \(\theta(0)\|p(0)\|_H^2\), we obtain
\begin{align*}
&c\int_0^{\tau_n}\theta(s)
       \bigl(\|p(s)\|_V^2+\|q(s)\|_H^2+\beta\|p(s)\|_H^2\bigr)\,ds\\
&\quad\le\theta(\tau_n)\|p(\tau_n)\|_H^2+CD
 +C\left(\sup_{0\le t\le T}\theta(t)\|p(t)\|_H^2\right)^{1/2}D^{1/2}
 +2|\mathcal M(\tau_n)|.
\end{align*}
We bound \(\theta(\tau_n)\|p(\tau_n)\|_H^2\) by
\(\sup_{0\le t\le T}\theta(t)\|p(t)\|_H^2\) and use
\[
C\left(\sup_{0\le t\le T}\theta(t)\|p(t)\|_H^2\right)^{1/2}D^{1/2}
\le C\left[\sup_{0\le t\le T}\theta(t)\|p(t)\|_H^2+D\right].
\]
We obtain
\begin{align}
&\int_0^{\tau_n}\theta(s)
       \bigl(\|p(s)\|_V^2+\|q(s)\|_H^2+\beta\|p(s)\|_H^2\bigr)\,ds
       \notag\\
&\quad\le C\left[
       \sup_{0\le t\le T}\theta(t)\|p(t)\|_H^2
       +D+|\mathcal M(\tau_n)|\right].
\label{eq:BSEE-stopped-energy-bound}
\end{align}
By \eqref{eq:unbounded-xi-BSEE-maximal-higher-moment},
\(D\in L^\alpha\), and
\(\mathbb E|\mathcal M(\tau_n)|^\alpha<\infty\), both sides of
\eqref{eq:BSEE-stopped-energy-bound} belong to \(L^\alpha(\mathcal F_T)\).
We now apply the BDG inequality to \(\mathcal M(\cdot\wedge\tau_n)\),
use \eqref{eq:BSEE-martingale-bracket-bound}, and apply
Young's inequality to the product of powers of the maximum and the
integral involving \(q\). For every \(\varepsilon>0\), this gives
\begin{equation}
\begin{aligned}
\mathbb E|\mathcal M(\tau_n)|^\alpha
&\le C_\alpha\mathbb E\left[
       \left(\sup_{0\le t\le T}\theta(t)\|p(t)\|_H^2\right)^{\alpha/2}
       \left(\int_0^{\tau_n}\theta(s)\|q(s)\|_H^2\,ds\right)^{\alpha/2}
                          \right]\\
&\le\varepsilon\mathbb E
       \left(\int_0^{\tau_n}\theta(s)\|q(s)\|_H^2\,ds\right)^\alpha
       +C_{\alpha,\varepsilon}\mathbb E
       \left(\sup_{0\le t\le T}\theta(t)\|p(t)\|_H^2\right)^\alpha.
\end{aligned}
\label{eq:BSEE-stopped-martingale-moment}
\end{equation}
Raising \eqref{eq:BSEE-stopped-energy-bound} to the power
\(\alpha\), taking expectations, and substituting
\eqref{eq:BSEE-stopped-martingale-moment}, we obtain
\begin{align*}
&\mathbb E\left(\int_0^{\tau_n}\theta(s)
       \bigl(\|p(s)\|_V^2+\|q(s)\|_H^2+\beta\|p(s)\|_H^2\bigr)\,ds
                                                        \right)^\alpha\\
&\quad\le C_\alpha\mathbb ED^\alpha
       +C_{\alpha,\varepsilon}\mathbb E
          \left(\sup_{0\le t\le T}\theta(t)\|p(t)\|_H^2\right)^\alpha\\
&\quad+C_\alpha\varepsilon\mathbb E
          \left(\int_0^{\tau_n}\theta(s)\|q(s)\|_H^2\,ds\right)^\alpha.
\end{align*}
Since
\[
\int_0^{\tau_n}\theta(s)\|q(s)\|_H^2\,ds
\le\int_0^{\tau_n}\theta(s)
       \bigl(\|p(s)\|_V^2+\|q(s)\|_H^2+\beta\|p(s)\|_H^2\bigr)\,ds,
\]
choosing \(C_\alpha\varepsilon\le1/2\) absorbs the term
\(C_\alpha\varepsilon\mathbb E
 (\int_0^{\tau_n}\theta(s)\|q(s)\|_H^2\,ds)^\alpha\)
into the moment of the full energy on the left.
We then apply \eqref{eq:unbounded-xi-BSEE-maximal-higher-moment} to
\(\mathbb E(\sup_{0\le t\le T}\theta(t)\|p(t)\|_H^2)^\alpha\)
and obtain
\begin{equation}
\mathbb E\left(\int_0^{\tau_n}\theta(s)
       \bigl(\|p(s)\|_V^2+\|q(s)\|_H^2+\beta\|p(s)\|_H^2\bigr)\,ds
                                                    \right)^\alpha
\le C_\alpha\mathbb ED^\alpha.
\label{eq:BSEE-stopped-higher-energy}
\end{equation}
The constant is independent of \(n\). The integrals on the left-hand
side of \eqref{eq:BSEE-stopped-higher-energy} are nonnegative and
increase to the integral over \([0,T]\), so monotone
convergence gives
\[
\mathbb E\left(\int_0^T\theta(s)
       \bigl(\|p(s)\|_V^2+\|q(s)\|_H^2+\beta\|p(s)\|_H^2\bigr)\,ds
                                                    \right)^\alpha
\le C_\alpha\mathbb ED^\alpha.
\]
Together with \eqref{eq:unbounded-xi-BSEE-maximal-higher-moment},
this proves \eqref{eq:unbounded-xi-weighted-BSEE-estimate}.

Finally, we remove the weights under the additional data
assumptions. From \(1\le\theta(t)\le
 e^{\beta T+\kappa\ell\xi(T)}\), the definition of \(D\) gives
\begin{equation}
\begin{aligned}
D
&\le e^{\beta T+\kappa\ell\xi(T)}\Bigg[
 \|\varphi\|_H^2
 +\frac1\beta\int_0^T\|g(s,0,0)\|_H^2\,ds\\
&\hspace{28mm}
 +\left(\int_{(0,T]}\|\psi_1(s)\|_H\,d\xi(s)\right)^2
 +\frac1{(\kappa-2)\ell}
      \int_{(0,T]}\|\psi_2(s)\|_H^2\,d\xi(s)\Bigg].
\end{aligned}
\label{eq:BSEE-data-exponential-comparison}
\end{equation}
Raising \eqref{eq:BSEE-data-exponential-comparison} to the power
\(\alpha\) and applying the Cauchy--Schwarz inequality, we obtain
\begin{align}
\mathbb ED^\alpha
&\le C_\alpha e^{\alpha\beta T}
        \bigl(\mathbb Ee^{2\alpha\kappa\ell\xi(T)}\bigr)^{1/2}\notag\\
&\quad\times\Bigg\{\mathbb E\Bigg[
 \|\varphi\|_H^{4\alpha}
 +\frac1{\beta^{2\alpha}}
      \left(\int_0^T\|g(s,0,0)\|_H^2\,ds\right)^{2\alpha}\notag\\
&\hspace{25mm}
 +\left(\int_{(0,T]}\|\psi_1(s)\|_H\,d\xi(s)\right)^{4\alpha}
 +\frac1{((\kappa-2)\ell)^{2\alpha}}
      \left(\int_{(0,T]}\|\psi_2(s)\|_H^2\,d\xi(s)\right)^{2\alpha}
              \Bigg]\Bigg\}^{1/2}.
\label{eq:BSEE-unweighted-data-bound}
\end{align}
Since \(\theta\ge1\) and \(\alpha\ge1\), the left-hand side of
\eqref{eq:unbounded-xi-weighted-BSEE-estimate} dominates that of
\eqref{eq:unbounded-xi-unweighted-BSEE-estimate}. Combining
\eqref{eq:unbounded-xi-weighted-BSEE-estimate} and
\eqref{eq:BSEE-unweighted-data-bound} proves
\eqref{eq:unbounded-xi-unweighted-BSEE-estimate}.
\end{proof}

We now construct a solution in the weighted space by approximating the driver with bounded processes. 
\begin{theorem}
\label{thm:unbounded-xi-BSEE-finite-order}
Let \(\alpha\ge1\), \(\kappa>2\), and
\(\beta\ge\max\{\beta_\alpha,\beta_{2\alpha}\}\). Assume \((B1)\)--\((B3)\)
and suppose that
\begin{equation}
\begin{aligned}
 e^{\kappa\ell\xi(T)}\in L^{4\alpha}(\mathcal F_T),\quad
  \varphi\in L^{8\alpha}(\mathcal F_T;H),\quad
 g(\cdot,0,0)\in L^{2,8\alpha}_{\mathbb F}(0,T;H),\quad
  \psi(\cdot,0)\in L^{1,8\alpha}_{\mathbb F,\xi}(0,T;H).
\end{aligned}
\label{eq:unbounded-xi-BSEE-finite-order-data}
\end{equation}
Then \eqref{eq:FV-BSEE} has a unique solution with
\(\|(p,q)\|_{\theta_{\beta,\ell}^{0}}<\infty\). The solution \((p,q)\) satisfies
\eqref{eq:unbounded-xi-weighted-BSEE-estimate}
with \(\psi_1=\psi(\cdot,0)\) and \(\psi_2=0\).
The same estimate also holds with \(\alpha\) replaced
throughout by \(2\alpha\).
\end{theorem}
\begin{proof}
	Write
	\[
	\theta(t):=\theta_{\beta,\ell}^{0}(t)
	=e^{\beta t+\kappa\ell\xi(t)},
	\quad
	\xi^n(t):=\xi(t)\wedge n,
	\quad
	I_0:=\int_{(0,T]}\|\psi(t,0)\|_H\,d\xi(t).
	\]
	\noindent\emph{Step 1: Construction for the bounded drivers.}
	For every \(n\ge1\),
	\[
	0\le d\xi^n\le d\xi,
	\quad
	\xi^n(T)\le n,\quad
	\mathbb E
	\left(
	\int_{(0,T]}\|\psi(t,0)\|_H\,d\xi^n(t)
	\right)^2
	\le
	\mathbb E I_0^2
	<\infty.
	\]
	Hence Lemma~\ref{lem:BSEE-bounded-driver-construction}, with
	\(\chi=\xi^n\), gives a unique pair
	$
	(p^n,q^n)
	\in
	\bigl(
	\mathbb D^2_{\mathbb F}(0,T;H)
	\cap L^2_{\mathbb F}(0,T;V)
	\bigr)
	\times L^2_{\mathbb F}(0,T;H)
	$
	satisfying
	\begin{align*}
		p^n(t)
		&=
		\varphi
		+\int_t^T
		\bigl[
		M(s)p^n(s)
		+N(s)q^n(s)
		+g(s,p^n(s),q^n(s))
		\bigr]\,ds
		\\
		&\quad
		+\int_{(t,T]}\psi(s,p^n(s))\,d\xi^n(s)
		-\int_t^T q^n(s)\,dW(s).
	\end{align*}

	Fix
	\(\widehat\beta\ge\beta_2\), and set
	$
	\widehat\theta_n(t)
	:=
	e^{\widehat\beta t+\kappa\ell\xi^n(t)}.
	$
	Since \(\xi^n(T)\le n\),
	\[
	1\le\widehat\theta_n(t)
	\le e^{\widehat\beta T+\kappa\ell n},
	\quad 0\le t\le T.
	\]
	The bound on \(\widehat\theta_n\) implies
	\(\|(p^n,q^n)\|_{\widehat\theta_n}<\infty\).
	The assumptions \eqref{eq:unbounded-xi-BSEE-finite-order-data} and
	\(d\xi^n\le d\xi\) also imply the data conditions of
	Proposition~\ref{prop:unbounded-xi-BSEE-estimate} with moment
	parameter \(2\), driver \(\xi^n\), and
	\(\psi_1=\psi(\cdot,0)\), \(\psi_2=0\).
	Applying that proposition, we obtain
	\begin{equation}
\mathbb E\Bigg[
	\sup_{0\le t\le T}\|p^n(t)\|_H^4
	+
	\left(\int_0^T\|p^n(t)\|_V^2\,dt\right)^2
	+
	\left(\int_0^T\|q^n(t)\|_H^2\,dt\right)^2
	\Bigg]
	<\infty.
\label{eq:BSEE-fixed-driver-fourth-moment}
\end{equation}
	This bound may depend on \(n\).

	Estimate \eqref{eq:BSEE-fixed-driver-fourth-moment} implies
	\(\|(p^n,q^n)\|_\theta<\infty\), where the norm is defined in
	\eqref{eq:BSEE-theta-norm}. Indeed, since
	\(\theta(t)\le\theta(T)\),
	\begin{align*}
		\|(p^n,q^n)\|_{\theta_{\beta,\ell}^{0}}^2
		&\le
		\bigl(\mathbb E\theta(T)^2\bigr)^{1/2}
		\\
		&\quad\times
		\Bigg\{
		\mathbb E\left[
		\left(
		\sup_{0\le t\le T}\|p^n(t)\|_H^2
		+
		\int_0^T
		\bigl(
		\|p^n(t)\|_V^2+\|q^n(t)\|_H^2
		\bigr)\,dt
		\right)^2
		\right]
		\Bigg\}^{1/2}
		<\infty.
	\end{align*}
	Here \(\mathbb E\theta(T)^2<\infty\) follows from
	\eqref{eq:unbounded-xi-BSEE-finite-order-data}.
	
	To verify the weighted data conditions with moment parameter \(2\alpha\), set
	\[
	F_0:=\|\varphi\|_H+
	       \left(\int_0^T\|g(t,0,0)\|_H^2\,dt\right)^{1/2}+I_0
	       \in L^{8\alpha}(\mathcal F_T).
	\]
	Since \(\theta(T)\in L^{4\alpha}(\mathcal F_T)\),
	the Cauchy--Schwarz inequality gives
	\[
	\mathbb E[\theta(T)^{2\alpha}F_0^{4\alpha}]
	\le e^{2\alpha\beta T}
	   \bigl(\mathbb E e^{4\alpha\kappa\ell\xi(T)}\bigr)^{1/2}
	   \bigl(\mathbb E F_0^{8\alpha}\bigr)^{1/2}<\infty.
	\]
	Together with \(\theta(t)\le\theta(T)\), this gives
	\begin{equation}
	\theta(T)^{1/2}\varphi\in L^{4\alpha}(\mathcal F_T;H),
	\quad
	\theta^{1/2}g(\cdot,0,0)
	   \in L^{2,4\alpha}_{\mathbb F}(0,T;H),
	\label{eq:BSEE-weighted-terminal-and-drift-data}
	\end{equation}
	and, since the weighted Stieltjes integral is at most
	\(\theta(T)^{1/2}I_0\),
	\begin{equation}
	\mathbb E\left(
	\int_{(0,T]}\theta(t)^{1/2}\|\psi(t,0)\|_H\,d\xi(t)
	\right)^{4\alpha}
	\le\mathbb E[\theta(T)^{2\alpha}F_0^{4\alpha}]<\infty.
	\label{eq:BSEE-weighted-Stieltjes-data-moment}
	\end{equation}
	
	Since
	$
	\ell\,d\xi^n\le\ell\,d\xi,
    $
	we apply Remark~\ref{rem:multiple-FV-driver-reduction} with
	\(\xi\) as the reference driver and \(\ell\) as the Lipschitz
	upper bound. We may therefore estimate \(\int\psi(s,p^n(s))\,d\xi^n(s)\)
	with the weight \(\theta_{\beta,\ell}^{0}\). The data term satisfies
	\[
	\int_{(0,T]}
	\theta(t)^{1/2}\|\psi(t,0)\|_H\,d\xi^n(t)
	\le
	\int_{(0,T]}
	\theta(t)^{1/2}\|\psi(t,0)\|_H\,d\xi(t).
	\]
Applying Proposition~\ref{prop:unbounded-xi-BSEE-estimate}
with moment parameter \(2\alpha\) to the equation with
reference driver \(\xi\), as in
Remark~\ref{rem:multiple-FV-driver-reduction}, taking
\(\psi_2=0\), and using the preceding bound on the Stieltjes
datum, we obtain
	\begin{equation}
		\sup_{n\ge1}
		\mathbb E\left[
		S_n^{2\alpha}
		+
		\left(
		\int_0^T
		\theta(t)
		\bigl(
		\|p^n(t)\|_V^2+\|q^n(t)\|_H^2
		\bigr)\,dt
		\right)^{2\alpha}
		\right]
		<\infty,
		\quad
		S_n:=
		\sup_{0\le t\le T}
		\theta(t)\|p^n(t)\|_H^2.
		\label{eq:unbounded-xi-BSEE-uniform-approximation}
	\end{equation}
	
	\medskip
	\noindent\emph{Step 2: Cauchy property of the approximating solutions.}
	Fix \(m\ge n\), and set
	\[
	\widehat p^{m,n}:=p^m-p^n,
	\quad
	\widehat q^{m,n}:=q^m-q^n.
	\]
	Since the function
	\(x\mapsto(x\wedge m)-(x\wedge n)\) is nondecreasing,
	\[
	0\le d(\xi^m-\xi^n)\le d\xi,
	\quad
	\xi^m(T)-\xi^n(T)
	\le(\xi(T)-n)^+.
	\]
	Subtracting the equation with driver \(\xi^n\)
	from the one with driver \(\xi^m\) gives
	\begin{align}
		\widehat p^{m,n}(t)
		&=
		\int_t^T
		\Big[
		M(s)\widehat p^{m,n}(s)
		+N(s)\widehat q^{m,n}(s)
		\notag\\
		&\hspace{23mm}
		+g(s,p^m(s),q^m(s))
		-g(s,p^n(s),q^n(s))
		\Big]\,ds
		\notag\\
		&\quad+
		\int_{(t,T]}
		\bigl[
		\psi(s,p^m(s))-\psi(s,p^n(s))
		\bigr]\,d\xi^n(s)
		\notag\\
		&\quad+
		\int_{(t,T]}
		\psi(s,p^m(s))\,d(\xi^m-\xi^n)(s)
		-\int_t^T\widehat q^{m,n}(s)\,dW(s).
		\label{eq:BSEE-truncated-difference-equation}
	\end{align}
	
	The coefficient of \(d\xi^n\) in
	\eqref{eq:BSEE-truncated-difference-equation} is the map
	\(y\mapsto\psi(s,p^n(s)+y)-\psi(s,p^n(s))\), evaluated at
	\(y=\widehat p^{m,n}(s)\). It vanishes at zero and has
	Lipschitz constant \(\ell\). Since
	$
	\ell\,d\xi^n\le\ell\,d\xi,
	$
	Remark~\ref{rem:multiple-FV-driver-reduction} permits the use of
	\(\theta\) when estimating the first Stieltjes integral in
	\eqref{eq:BSEE-truncated-difference-equation}.
	In the integral against \(d(\xi^m-\xi^n)\), the process
	\(\psi(\cdot,p^m)\) is fixed as the difference state varies
	and thus enters the difference equation as an additive coefficient.
	To estimate
	\(\mathbb E(\int_{(0,T]}\theta(t)^{1/2}\|\psi(t,p^m(t))\|_H
	\,d(\xi^m-\xi^n)(t))^{2\alpha}\), we use \((B3)\) to obtain
	\begin{align}
		&\int_{(0,T]}
		\theta(t)^{1/2}
		\|\psi(t,p^m(t))\|_H\,d(\xi^m-\xi^n)(t)
		\notag\\
		&\quad\le
		\ell S_m^{1/2}
		\bigl[\xi^m(T)-\xi^n(T)\bigr]
		+
		\int_{(0,T]}
		\theta(t)^{1/2}
		\|\psi(t,0)\|_H\,d(\xi^m-\xi^n)(t)
		\notag\\
		&\quad\le
		\ell S_m^{1/2}(\xi(T)-n)^+
		+
		I_{\{\xi(T)>n\}}
		\int_{(0,T]}
		\theta(t)^{1/2}
		\|\psi(t,0)\|_H\,d\xi(t).
		\label{eq:BSEE-truncated-tail-bound}
	\end{align}
	If \(\ell>0\), the exponential moment assumption implies
	\(\xi(T)\in L^{4\alpha}(\mathcal F_T)\). Hence
	\[
	\mathbb E\bigl[((\xi(T)-n)^+)^{4\alpha}\bigr]
	\longrightarrow0.
	\]
	By the Cauchy--Schwarz inequality and
	\eqref{eq:unbounded-xi-BSEE-uniform-approximation},
	\begin{align*}
		&\sup_{m\ge n}
		\mathbb E
		\left[
		\ell S_m^{1/2}(\xi(T)-n)^+
		\right]^{2\alpha}
		\le
		\ell^{2\alpha}
		\left(
		\sup_m\mathbb E S_m^{2\alpha}
		\right)^{1/2}
		\left\{
		\mathbb E
		\bigl[((\xi(T)-n)^+)^{4\alpha}\bigr]
		\right\}^{1/2}
		\longrightarrow0.
	\end{align*}
	When \(\ell=0\),
	\(\ell S_m^{1/2}(\xi(T)-n)^+\) is identically zero.
	
	For the second term on the right-hand side of
	\eqref{eq:BSEE-truncated-tail-bound},
	\eqref{eq:BSEE-weighted-Stieltjes-data-moment} and dominated
	convergence give
	\[
	\mathbb E\left[
	I_{\{\xi(T)>n\}}
	\left(
	\int_{(0,T]}
	\theta(t)^{1/2}
	\|\psi(t,0)\|_H\,d\xi(t)
	\right)^{2\alpha}
	\right]
	\longrightarrow0.
	\]
	Consequently,
	\begin{equation}
		\sup_{m\ge n}
		\mathbb E
		\left(
		\int_{(0,T]}
		\theta(t)^{1/2}
		\|\psi(t,p^m(t))\|_H\,d(\xi^m-\xi^n)(t)
		\right)^{2\alpha}
		\longrightarrow0.
		\label{eq:BSEE-truncated-tail-convergence}
	\end{equation}
	
	Estimate \eqref{eq:unbounded-xi-BSEE-uniform-approximation} gives
	\(\|(\widehat p^{m,n},\widehat q^{m,n})\|_\theta<\infty\)
	in the norm \eqref{eq:BSEE-theta-norm}. The Lebesgue generator of
	\eqref{eq:BSEE-truncated-difference-equation} vanishes at
	\((\widehat p^{m,n},\widehat q^{m,n})=(0,0)\) and has the same
	Lipschitz constant \(L\). Thus
	Remark~\ref{rem:multiple-FV-driver-reduction} and
	Proposition~\ref{prop:unbounded-xi-BSEE-estimate}, applied with
	moment parameter \(\alpha\), zero terminal datum, and the additive
	Stieltjes datum in
	\eqref{eq:BSEE-truncated-difference-equation}, give
	\begin{align}
		&\mathbb E\Bigg[
		\sup_{0\le t\le T}
		\theta(t)^\alpha
		\|\widehat p^{m,n}(t)\|_H^{2\alpha}
		\notag\\
		&\quad+
		\left(
		\int_0^T
		\theta(t)
		\bigl(
		\|\widehat p^{m,n}(t)\|_V^2
		+\|\widehat q^{m,n}(t)\|_H^2
		+\beta\|\widehat p^{m,n}(t)\|_H^2
		\bigr)\,dt
		\right)^\alpha
		\Bigg]
		\notag\\
		&\quad\le
		C_\alpha
		\mathbb E
		\left(
		\int_{(0,T]}
		\theta(t)^{1/2}
		\|\psi(t,p^m(t))\|_H\,d(\xi^m-\xi^n)(t)
		\right)^{2\alpha}.
		\label{eq:BSEE-truncated-difference-estimate}
	\end{align}
Estimates \eqref{eq:BSEE-truncated-difference-estimate} and
\eqref{eq:BSEE-truncated-tail-convergence} show that
\((\theta^{1/2}p^n,\theta^{1/2}q^n)\)
is Cauchy in the Banach space
\[
\bigl(
\mathbb D^{2\alpha}_{\mathbb F}(0,T;H)
\cap L^{2,2\alpha}_{\mathbb F}(0,T;V)
\bigr)
\times L^{2,2\alpha}_{\mathbb F}(0,T;H).
\]
Its limit can therefore be written as
\((\theta^{1/2}p,\theta^{1/2}q)\). Thus
\begin{align}
&\mathbb E\Bigg[
\sup_{0\le t\le T}
\theta(t)^\alpha
\|p^n(t)-p(t)\|_H^{2\alpha}
\notag\\
&\quad+
\left(
\int_0^T
\theta(t)
\bigl(
\|p^n(t)-p(t)\|_V^2
+\|q^n(t)-q(t)\|_H^2
\bigr)\,dt
\right)^\alpha
\Bigg]
\longrightarrow0.
\label{eq:BSEE-approximation-convergence}
\end{align}
Since \(\theta^{-1/2}\) is adapted and c\`adl\`ag,
\(p\) has an adapted c\`adl\`ag version.
Since \(\theta\ge1\) and \(\alpha\ge1\),
\eqref{eq:BSEE-approximation-convergence} also implies convergence
in \(\mathbb D^2_{\mathbb F}(0,T;H)\cap
L^2_{\mathbb F}(0,T;V)\) for \(p^n\) and
\(L^2_{\mathbb F}(0,T;H)\) for \(q^n\).

	\medskip
	\noindent\emph{Step 3: Identification of the limit.}
	The bounds on \(M,N\) and the Lipschitz property of \(g\) imply
	\begin{equation}
\begin{aligned}
		&\mathbb E\int_0^T
		\Big\|
		M(t)(p^n(t)-p(t))
		+N(t)(q^n(t)-q(t))
		\\
		&\hspace{29mm}
		+g(t,p^n(t),q^n(t))
		-g(t,p(t),q(t))
		\Big\|_{V^*}^2\,dt
		\\
		&\quad\le
		C\mathbb E\int_0^T
		\bigl(
		\|p^n(t)-p(t)\|_V^2
		+\|q^n(t)-q(t)\|_H^2
		\bigr)\,dt
		\longrightarrow0.
	\end{aligned}
\label{eq:BSEE-drift-difference-convergence}
\end{equation}
	The Cauchy--Schwarz inequality in time and
	\eqref{eq:BSEE-drift-difference-convergence} imply that the integrals
	from \(0\) to \(t\) of
	\(M(p^n-p)+N(q^n-q)+g(\cdot,p^n,q^n)-g(\cdot,p,q)\)
	converge to zero in \(L^2(\Omega;C([0,T];V^*))\).
	
	For the stochastic integrals, the BDG inequality gives
	\[
	\mathbb E
	\sup_{0\le t\le T}
	\left\|
	\int_0^t
	(q^n(s)-q(s))\,dW(s)
	\right\|_H^2
	\le
	C\mathbb E\int_0^T
	\|q^n(s)-q(s)\|_H^2\,ds
	\longrightarrow0.
	\]
	
	It remains to identify the Stieltjes term. If \(\ell>0\),
	\eqref{eq:BSEE-reciprocal-weight-integrals} gives
	\[
	\ell
	\int_{(0,T]}\|p(s)\|_H\,d\xi(s)
	\le
	\frac2\kappa
	\sup_{0\le t\le T}
	\theta(t)^{1/2}\|p(t)\|_H.
	\]
	Hence
	\[
	\int_{(0,T]}
	\|\psi(s,p(s))\|_H\,d\xi(s)
	\le
	\frac2\kappa
	\sup_{0\le t\le T}
	\theta(t)^{1/2}\|p(t)\|_H
	+I_0
	\in L^2(\mathcal F_T).
	\]
	If \(\ell=0\), \(\psi(t,y)=\psi(t,0)\); the square-integrable
	total variation of the limiting Stieltjes integral follows from
	\[
	\int_{(0,T]}
	\|\psi(s,p(s))\|_H\,d\xi(s)
	=I_0\in L^2(\mathcal F_T).
	\]
	
	Now
	\begin{align*}
		&\int_{(0,t]}
		\psi(s,p^n(s))\,d\xi^n(s)
		-
		\int_{(0,t]}
		\psi(s,p(s))\,d\xi(s)
		\\
		&\quad=
		\int_{(0,t]}
		\bigl[
		\psi(s,p^n(s))-\psi(s,p(s))
		\bigr]\,d\xi^n(s)
		-
		\int_{(0,t]}
		\psi(s,p(s))\,d(\xi-\xi^n)(s).
	\end{align*}
	For \(\ell>0\),
	\eqref{eq:BSEE-reciprocal-weight-integrals} and
	\(d\xi^n\le d\xi\) yield
	\[
	\sup_{0\le t\le T}
	\left\|
	\int_{(0,t]}
	\bigl[
	\psi(s,p^n(s))-\psi(s,p(s))
	\bigr]\,d\xi^n(s)
	\right\|_H
	\le
	\frac2\kappa
	\sup_{0\le t\le T}
	\theta(t)^{1/2}\|p^n(t)-p(t)\|_H.
	\]
	When \(\ell=0\), the integral on the left-hand side is identically
	zero. Moreover,
$
	d(\xi-\xi^n)
	\le
	I_{\{\xi(T)>n\}}\,d\xi.
$
	Therefore
	\begin{align*}
		&\mathbb E
		\sup_{0\le t\le T}
		\left\|
		\int_{(0,t]}
		\psi(s,p^n(s))\,d\xi^n(s)
		-
		\int_{(0,t]}
		\psi(s,p(s))\,d\xi(s)
		\right\|_H^2
		\\
		&\quad\le
		\frac8{\kappa^2}
		\mathbb E
		\sup_{0\le t\le T}
		\theta(t)\|p^n(t)-p(t)\|_H^2
		\\
		&\quad+
		2\mathbb E\left[
		I_{\{\xi(T)>n\}}
		\left(
		\int_{(0,T]}
		\|\psi(s,p(s))\|_H\,d\xi(s)
		\right)^2
		\right]
		\longrightarrow0.
	\end{align*}
	When \(\ell=0\), we may omit the first term on the right.
	The expected weighted maximum tends to zero by
	\eqref{eq:BSEE-approximation-convergence}, while the expectation
	containing \(I_{\{\xi(T)>n\}}\) tends to zero by dominated
	convergence because
	\((\int_{(0,T]}\|\psi(s,p(s))\|_H\,d\xi(s))^2\in L^1\).
	
	Since \(p^n(T)=\varphi\) and
	\(p^n\to p\) in
	\(\mathbb D^2_{\mathbb F}(0,T;H)\), we have
	\(p(T)=\varphi\) almost surely.
	The convergence of \(p^n,q^n\), together with that of the drift,
	stochastic, and Stieltjes integrals, yields \eqref{eq:FV-BSEE-weak-form} for each
	fixed time and test vector. We intersect the full-measure sets for
	rational times, the endpoint \(T\), and a countable dense subset
	of \(V\). The c\`adl\`ag \(V^*\)-valued versions and continuity
	in the test vector then extend the weak identity to every time and
	every test vector on the resulting full-measure set. Hence \((p,q)\) solves
	\eqref{eq:FV-BSEE}.
	
	The bounds on the terminal, Lebesgue, and Stieltjes data in
	\eqref{eq:BSEE-weighted-terminal-and-drift-data} and
	\eqref{eq:BSEE-weighted-Stieltjes-data-moment} allow us to apply
	Proposition~\ref{prop:unbounded-xi-BSEE-estimate} to \((p,q)\)
	with moment parameters \(\alpha\) and \(2\alpha\) to obtain
	the two asserted estimates.
	
	\medskip
	\noindent\emph{Step 4: Uniqueness in the weighted quadratic class.}
	Let \((\widetilde p,\widetilde q)\) be another solution satisfying
	$
	\|(\widetilde p,\widetilde q)\|_{\theta_{\beta,\ell}^{0}}
	<\infty.
	$
	Set \(\beta':=\max\{\beta,\beta_1\}\). Both solutions remain
	in the weighted quadratic class for \(\beta'\), since the ratio of
	the new weight to the old one is bounded by \(e^{(\beta'-\beta)T}\).
	Their difference has zero terminal datum; its generators vanish
	at zero and retain the Lipschitz bounds \(L,\ell\). Applying
	\eqref{eq:unbounded-xi-weighted-BSEE-estimate} with moment parameter
	\(1\) and time-weight parameter \(\beta'\) gives
	\[
	\|(p-\widetilde p,q-\widetilde q)\|
	_{\theta_{\beta',\ell}^{0}}^2
	=0.
	\]
	Hence \(p\) and \(\widetilde p\) are indistinguishable and
	\(q=\widetilde q\)
	\(dt\otimes d\mathbb P\)-almost everywhere.
	
\end{proof}

The next corollary establishes well-posedness in the unweighted
\(4\alpha\)-moment space.
\begin{corollary}
\label{cor:FV-BSEE-unweighted-wellposed}
Assume the hypotheses of
Theorem~\ref{thm:unbounded-xi-BSEE-finite-order}.
Then \eqref{eq:FV-BSEE} has a unique solution in
\[
 \bigl(\mathbb D_{\mathbb F}^{4\alpha}(0,T;H)
       \cap L_{\mathbb F}^{2,4\alpha}(0,T;V)\bigr)
 \times L_{\mathbb F}^{2,4\alpha}(0,T;H).
\]
The solution satisfies \eqref{eq:unbounded-xi-unweighted-BSEE-estimate}
with moment parameter \(2\alpha\), \(\psi_1=\psi(\cdot,0)\), and
\(\psi_2=0\).
\end{corollary}
\equationlabelalias{eq:FV-BSEE-finite-order-unweighted-estimate}{eq:unbounded-xi-unweighted-BSEE-estimate}

\begin{proof}
Theorem~\ref{thm:unbounded-xi-BSEE-finite-order} gives a solution
in the weighted quadratic class. The data conditions
\eqref{eq:unbounded-xi-BSEE-finite-order-data} allow us to apply
\eqref{eq:unbounded-xi-unweighted-BSEE-estimate} with moment parameter
\(2\alpha\), \(\psi_1=\psi(\cdot,0)\), and \(\psi_2=0\).
This gives the stated estimate and the required integrability.

For uniqueness, let \((\widetilde p,\widetilde q)\) be another
solution in the stated unweighted class. Since \(\alpha\ge1\),
its squared maximum and integrated squared norms belong to
\(L^2(\mathcal F_T)\). Also,
\(\theta_{\beta,\ell}^{0}(T)\in L^2(\mathcal F_T)\) by
\eqref{eq:unbounded-xi-BSEE-finite-order-data}. The Cauchy--Schwarz inequality gives
\begin{align*}
 \|(\widetilde p,\widetilde q)\|_{\theta_{\beta,\ell}^{0}}^2
 &\le \|\theta_{\beta,\ell}^{0}(T)\|_{L^2(\mathcal F_T)}\\
 &\quad\times\left\|
    \sup_{0\le t\le T}\|\widetilde p(t)\|_H^2
     +\int_0^T\bigl(\|\widetilde p(t)\|_V^2
                      +\|\widetilde q(t)\|_H^2\bigr)\,dt
    \right\|_{L^2(\mathcal F_T)}<\infty.
\end{align*}
Uniqueness in
Theorem~\ref{thm:unbounded-xi-BSEE-finite-order} therefore implies
that \(\widetilde p\) and \(p\) are indistinguishable and
\(\widetilde q=q\), \(dt\otimes d\mathbb P\)-almost everywhere.
\end{proof}
\section{Operator-valued BSIEs with finite-variation drivers}
\label{subsec:FV-BSIE-well-posedness}

We extend the conditional-expectation formulation of operator-valued BSIEs in
\cite{LT21} to equations with a finite-variation driver:
\begin{equation}
\begin{aligned}
P(t)=\mathbb E\Bigg[
&L^*(t,T)\Xi L(t,T)
+\int_t^T L^*(t,s)F(s,P(s))L(t,s)\,ds
\\
&+\int_{(t,T]}L^*(t,s-)G(s)L(t,s-)\,d\xi(s)
\,\Bigg|\,\mathcal F_t\Bigg].
\end{aligned}
\label{eq:FV-BSIE}
\end{equation}

\subsection{Assumptions on the stochastic evolution family and the data}

Throughout this section, \(\xi\) is a real-valued adapted nondecreasing
c\`adl\`ag process with \(\xi(0)=0\) and \(\xi(T)<\infty\) \(\mathbb P\text{-a.s.}\) Set
$
\Delta:=\{(t,s)\in[0,T]^2:0\le t\le s\le T\}.
$ Let $L$ be a stochastic evolution operator on $H$, that is, \[L=
\{L(t,s)\in\mathfrak{L}(L^{2}(\mathcal{F}_{t},H);L^{2}(\mathcal{F}%
_{s},H)):(t,s)\in\Delta\}.
\]
We use \(L^{\ast}\) as formal adjoint notation. For fixed \((t,s)\in\Delta\)
and \(u\in L^{1}(\mathcal F_s;H)\), its action on
\(v\in L^{2}(\mathcal F_t;H)\) is defined through the scalar pairing
\[
(L^{\ast}(t,s)u)(v):=\langle u,L(t,s)v\rangle\ \ \mathbb P\text{-a.s.},\quad\text{for
each}\ v\in L^{2}(\mathcal{F}_{t},H).
\]
We impose the following two conditions.

\begin{description}
\item[\((C1)\)]
For each \((t,s)\in\Delta\),
$
L(t,s)\in
\mathfrak L\bigl(L^4(\mathcal F_t;H);L^4(\mathcal F_s;H)\bigr).
$
Moreover,
$
L(t,t)\eta=\eta
\ \text{for}\
\eta\in L^4(\mathcal F_t;H)\ \text{and}\ t\in[0,T],
$
 and  $
L(r,s)L(t,r)\eta=L(t,s)\eta$ for
$ \eta\in L^4(\mathcal F_t;H)
$ and \(0\le t\le r\le s\le T\).
The evolution family is local with respect to the initial
\(\sigma\)-field: $
L(t,s)(I_A\eta)=I_A L(t,s)\eta,
\ \mathbb P\text{-a.s.}
$ for every \(A\in\mathcal F_t\),
\(\eta\in L^4(\mathcal F_t;H)\), and \(t\le s\le T\).

For every \(t\in[0,T]\) and every
\(\eta\in L^4(\mathcal F_t;H)\), the process
\(s\mapsto L(t,s)\eta\) admits a version in
\(\mathbb D^4_{\mathbb F}(t,T;H)\).
There exists a constant \(C_L>0\), independent of \(t\), such that
\begin{equation}
\mathbb E\left[
\sup_{t\le s\le T}\|L(t,s)\eta\|_H^4
\,\middle|\,\mathcal F_t
\right]
\le C_L\|\eta\|_H^4,
\quad\mathbb P\text{-a.s.}
\label{eq:BSIE-random-maximal-estimate}
\end{equation}

For every \(u\in H\), the family
\(\{L(t,s)u:(t,s)\in\Delta\}\) admits an \(H\)-valued version such that
$
(\omega,t,s)\longmapsto (L(t,s)u)(\omega)
$
is jointly measurable on \(\Omega\times\Delta\), and, for every
\((\omega,t)\), the path \(s\mapsto (L(t,s)u)(\omega)\) is
c\`adl\`ag on \([t,T]\).

\item[\((C2)\)]
The terminal datum satisfies
\(\Xi\in L_w^2(\mathcal F_T;\mathfrak L(H))\). The mapping
$
F:\Omega\times[0,T]\times\mathfrak L(H)\longrightarrow\mathfrak L(H)
$
is measurable from
\(\operatorname{Prog}\otimes\mathcal B_w(\mathfrak L(H))\) to
\(\mathcal B_w(\mathfrak L(H))\), and there exists \(L>0\) such that, for
every \((t,\omega)\in[0,T]\times\Omega\),
\begin{equation}
\|F(t,Q_1)-F(t,Q_2)\|_{\mathfrak L(H)}
\le L\|Q_1-Q_2\|_{\mathfrak L(H)},
\quad Q_1,Q_2\in\mathfrak L(H),
\label{eq:BSIE-F-Lipschitz}
\end{equation}
and $
F(\cdot,0)\in L^2_{\mathbb F,w}(0,T;\mathfrak L(H)).
$ Moreover,
\(G\in L^{1,2}_{\mathbb F,\xi,w}(0,T;\mathfrak L(H)).\)
\end{description}

\subsection{The solution space and the fixed-time conditional expectation}

For \(\beta>0\), set
\begin{equation}
\|P\|_\beta^2
:=
\mathbb E\int_0^T e^{\beta t}
\|P(t)\|_{\mathfrak L(H)}^2\,dt.
\label{eq:BSIE-Hbeta-norm}
\end{equation}
Let \(\mathcal H_\beta\) denote
\(L^2_{\mathbb F,w}(0,T;\mathfrak L(H))\), equipped with the norm
\(\|\cdot\|_\beta\), where processes equal
\((dt\otimes d\mathbb P)\)-almost everywhere are identified.
Since \(1\le e^{\beta t}\le e^{\beta T}\), this norm is equivalent to
the usual \(L^2_{\mathbb F,w}(0,T;\mathfrak L(H))\)-norm.
In particular, \(\mathcal H_\beta\) is a Banach space for every
\(\beta>0\). 
The deterministic weight \(e^{\beta t}\) suffices here because the
finite-variation term \(G(s)\,d\xi(s)\) in \eqref{eq:FV-BSIE}
does not contain \(P\).

For \(R\in\mathcal H_\beta\) and \(t\in[0,T]\), define
\(\Phi_R(t):H\times H\to L^1(\mathcal F_T)\) by
\begin{align}
\Phi_R(t)(u,v)
:={}&
\langle\Xi L(t,T)u,L(t,T)v\rangle_H
\notag\\
&+\int_t^T
\langle F(s,R(s))L(t,s)u,L(t,s)v\rangle_H\,ds
\notag\\
&+\int_{(t,T]}
\langle G(s)L(t,s-)u,L(t,s-)v\rangle_H\,d\xi(s),
\quad u,v\in H.
\label{eq:BSIE-Phi-R}
\end{align}

We first justify the conditional expectation in \eqref{eq:FV-BSIE}.
\begin{lemma}
\label{lem:BSIE-fixed-time-aggregation}
For every \(R\in\mathcal H_\beta\) and \(t\in[0,T]\),
\(\Phi_R(t)\) is a bounded bilinear mapping
from \(H\times H\) to \(L^1(\mathcal F_T)\). Moreover,
$
\Gamma_t(R)
:=
\mathbb E\big[\Phi_R(t)\mid\mathcal F_t\big]
\in L_w^2(\mathcal F_t;\mathfrak L(H)),
$
and
\begin{align}
\|\Gamma_t(R)\|_{\mathfrak L(H)}^2
&\le
3C_L\,\mathbb E\Bigg[
\|\Xi\|_{\mathfrak L(H)}^2
+T\int_t^T\|F(s,R(s))\|_{\mathfrak L(H)}^2\,ds
\notag\\
&\hspace{32mm}
+\left(\int_{(t,T]}\|G(s)\|_{\mathfrak L(H)}\,d\xi(s)\right)^2
\,\Bigg|\,\mathcal F_t\Bigg],
\quad\mathbb P\text{-a.s.}
\label{eq:BSIE-Gamma-pointwise}
\end{align}
The family \(\{\Gamma_t(R):0\le t\le T\}\) admits a weakly progressively
measurable modification.
\end{lemma}
\begin{proof}
Fix \(t\in[0,T]\). For each fixed finite family of vectors and
scalars, the linearity identities for \(L(t,s)\) hold simultaneously
at all rational times and at the endpoints outside one null set.
The c\`adl\`ag versions extend these identities to every
\(s\in[t,T]\) and to all left limits on \((t,T]\). 
Thus each term of \eqref{eq:BSIE-Phi-R}, including the
Stieltjes term, is bilinear in \(u\) and \(v\), with the
identities understood as almost-sure equalities.

Since \(R\in\mathcal H_\beta\) and \(e^{\beta s}\ge1\),
\eqref{eq:BSIE-F-Lipschitz} gives
\begin{align}
\mathbb E\int_t^T\|F(s,R(s))\|_{\mathfrak L(H)}^2\,ds
 &\le2\mathbb E\int_t^T\|F(s,0)\|_{\mathfrak L(H)}^2\,ds
       +2L^2\mathbb E\int_t^T\|R(s)\|_{\mathfrak L(H)}^2\,ds
 <\infty.
\label{eq:BSIE-F-R-square-integrability}
\end{align}
For every \(s\in(t,T]\), the left limit of
\(L(t,\cdot)u\) satisfies
\(\|L(t,s-)u\|_H\le\sup_{t\le r\le T}\|L(t,r)u\|_H\).
Consequently, \eqref{eq:BSIE-Phi-R} yields the pathwise bound
\begin{align*}
 |\Phi_R(t)(u,v)|
 &\le\left(
     \|\Xi\|_{\mathfrak L(H)}
     +\int_t^T\|F(s,R(s))\|_{\mathfrak L(H)}\,ds
     +\int_{(t,T]}\|G(s)\|_{\mathfrak L(H)}\,d\xi(s)
   \right)\\
 &\quad\times
     \sup_{t\le s\le T}\|L(t,s)u\|_H
     \sup_{t\le s\le T}\|L(t,s)v\|_H.
\end{align*}
H\"older's inequality with exponents \((2,4,4)\),
\eqref{eq:BSIE-random-maximal-estimate}, the data conditions,
and \eqref{eq:BSIE-F-R-square-integrability} show that
\(\Phi_R(t)\) is a bounded bilinear map from
\(H\times H\) to \(L^1(\mathcal F_T)\).

To construct the conditional expectation of \(\Phi_R(t)\), we apply
conditional H\"older's inequality and
\eqref{eq:BSIE-random-maximal-estimate} to obtain, for \(u,v\in H\),
\begin{align}
&\left|\mathbb E\big[\Phi_R(t)(u,v)\mid\mathcal F_t\big]\right|
\notag\\
&\quad\le
\sqrt{3C_L}\,\|u\|_H\|v\|_H
\Bigg(\mathbb E\Bigg[
\|\Xi\|_{\mathfrak L(H)}^2
+T\int_t^T\|F(s,R(s))\|_{\mathfrak L(H)}^2\,ds
\notag\\
&\hspace{54mm}
+\left(\int_{(t,T]}\|G(s)\|_{\mathfrak L(H)}\,d\xi(s)\right)^2
\,\Bigg|\,\mathcal F_t\Bigg]\Bigg)^{1/2}.
\label{eq:BSIE-scalar-conditional-domination}
\end{align}
The square root of the conditional expectation on the right-hand side of
\eqref{eq:BSIE-scalar-conditional-domination} belongs to
\(L^2(\mathcal F_t)\) by \((C2)\) and
\eqref{eq:BSIE-F-R-square-integrability}. The domination criterion
\eqref{eq:operator-valued-conditional-expectation-domination}
therefore yields \(\Gamma_t(R)\in
L_w^2(\mathcal F_t;\mathfrak L(H))\), uniquely up to \(\mathbb P\text{-a.s.}\)
equality, and gives \eqref{eq:BSIE-Gamma-pointwise}.

For fixed \(u,v\in H\), condition \((C1)\) and the weak progressive
measurability of \(F(\cdot,R(\cdot))\) and \(G\) imply that the terminal and
Lebesgue integrands are jointly measurable in their parameters. Taking
left limits in the jointly measurable c\`adl\`ag version in \((C1)\) also makes
\(\langle G(s)L(t,s-)u,L(t,s-)v\rangle_H\) jointly measurable
in \((\omega,t,s)\). The parameter-integration theorem shows that
\((\omega,t)\mapsto\Phi_R(t)(u,v)(\omega)\) is jointly measurable.
The scalar projection theorem \cite[Corollary~7.6.8]{CE15}
therefore gives a progressively measurable modification of
\(\{\mathbb E[\Phi_R(t)(u,v)\mid\mathcal F_t]:0\le t\le T\}\)
for each \(u,v\in H\). By \cite[Lemma~2.9]{LT21}, the weakly
adapted operator-valued family \(\{\Gamma_t(R):0\le t\le T\}\)
has a weakly progressively measurable modification.
\end{proof}

With \(\Gamma_t(R)\) defined for every deterministic \(t\) by
Lemma~\ref{lem:BSIE-fixed-time-aggregation}, we define a solution of
\eqref{eq:FV-BSIE} as follows.
\begin{definition}
\label{def:FV-BSIE-solution}
A process \(P\in\mathcal H_\beta\) is called a solution of \eqref{eq:FV-BSIE}
if, for every \(t\in[0,T]\),
\begin{equation}
P(t)=\Gamma_t(P),
\quad \mathbb P\text{-a.s.}
\label{eq:FV-BSIE-fixed-point}
\end{equation}
\end{definition}

\subsection{Well-posedness and conditional estimates}
The next theorem constructs the BSIE solution used in the quadratic
duality formula and gives its second-moment estimate.
\begin{theorem}
\label{thm:FV-BSIE-wellposed}\label{thm:unbounded-xi-BSIE}
Assume \((C1)\)--\((C2)\). For every sufficiently large \(\beta\),
BSIE \eqref{eq:FV-BSIE} has a solution
\(P\in\mathcal H_\beta\), unique up to modification. Moreover,
\begin{align}
\|P\|_\beta^2\le C_\beta\mathbb E\Bigg[
 &\|\Xi\|_{\mathfrak L(H)}^2
 +\int_0^T\|F(s,0)\|_{\mathfrak L(H)}^2ds+\left(\int_{(0,T]}\|G(s)\|_{\mathfrak L(H)}d\xi(s)\right)^2\Bigg],
\label{eq:FV-BSIE-estimate}
\end{align}
where \(C_\beta\) depends only on \(\beta,C_L,L,T\).
\end{theorem}
\begin{proof}
We use the weakly progressively measurable modification of \(\Gamma(R)\)
constructed in Lemma~\ref{lem:BSIE-fixed-time-aggregation}.
The Stieltjes term in the estimate of \(\|\Gamma(R)\|_\beta^2\)
satisfies
\[
\int_0^T e^{\beta t}\mathbb E
 \left(\int_{(t,T]}\|G(s)\|_{\mathfrak L(H)}d\xi(s)\right)^2dt
\le Te^{\beta T}\mathbb E
 \left(\int_{(0,T]}\|G(s)\|_{\mathfrak L(H)}d\xi(s)\right)^2.
\]
We integrate \eqref{eq:BSIE-Gamma-pointwise} against
\(e^{\beta t}dt\). Applying Tonelli's theorem to the term containing
\(\int_t^T\|R(s)\|^2ds\) and using
\(\int_0^s e^{\beta t}dt\le e^{\beta s}/\beta\), we obtain
\begin{equation}
\|\Gamma(R)\|_\beta^2 \le CTe^{\beta T}\mathbb E\Bigg[ \|\Xi\|_{\mathfrak L(H)}^2+ \int_0^T\|F(s,0)\|_{\mathfrak L(H)}^2ds +\left(\int_{(0,T]}\|G(s)\|_{\mathfrak L(H)}d\xi(s)\right)^2\Bigg] +\frac{6C_LTL^2}{\beta}\|R\|_\beta^2.
\label{eq:BSIE-self-map-estimate}
\end{equation}
Thus \(\Gamma\) maps \(\mathcal H_\beta\) to itself.
In the difference \(\Phi_{R_1}(t)-\Phi_{R_2}(t)\), the terms involving
\(\Xi\) and \(G\) in \eqref{eq:BSIE-Phi-R} cancel.
Applying \eqref{eq:BSIE-F-Lipschitz} to
\(F(s,R_1(s))-F(s,R_2(s))\) therefore gives
\begin{equation}
\|\Gamma(R_1)-\Gamma(R_2)\|_\beta^2
\le\frac{C_LTL^2}{\beta}\|R_1-R_2\|_\beta^2.
\label{eq:BSIE-contraction-final}
\end{equation}
We choose \(\beta\) such that \(6C_LTL^2/\beta\le1/2\). Banach's
fixed-point theorem gives a fixed point \(R\) in the space of
\(dt\otimes d\mathbb P\)-equivalence classes. Put \(P=\Gamma(R)\). Then \(P(t)=\Gamma_t(P)\ \mathbb P\text{-a.s.}\) for each $t\in[0,T]$.
If \(P_1\) and \(P_2\) are solutions, uniqueness of the fixed point gives
\(P_1=P_2\), \(dt\otimes d\mathbb P\)-almost everywhere. Then
\(\Gamma_t(P_1)=\Gamma_t(P_2)\) \(\mathbb P\text{-a.s.}\) for each $t\in[0,T]$. The identity \eqref{eq:FV-BSIE-fixed-point} therefore gives
\(P_1(t)=P_2(t)\) \(\mathbb P\text{-a.s.}\) for each \(t\), which proves uniqueness up to
modification. We substitute \(R=P\) into \eqref{eq:BSIE-self-map-estimate}.
Since \(P=\Gamma(P)\), the bound
\((6C_LTL^2/\beta)\|P\|_\beta^2\le\tfrac12\|P\|_\beta^2\)
allows us to absorb the last term of
\eqref{eq:BSIE-self-map-estimate} into the left-hand side and obtain
\eqref{eq:FV-BSIE-estimate}. 
\end{proof}

To control quadratic pairings with random states, we next bound
\(\|P(t)\|_{\mathfrak L(H)}\) by conditional moments of the data.
\begin{proposition}
\label{cor:FV-BSIE-conditional-estimate}
Under \((C1)\)--\((C2)\), for every deterministic \(t\in[0,T]\),
\begin{align}
\|P(t)\|_{\mathfrak L(H)}^2
&\le C\mathbb E\Bigg[
 \|\Xi\|_{\mathfrak L(H)}^2+
 \int_t^T\|F(s,0)\|_{\mathfrak L(H)}^2ds
 \notag\\
&\hspace{26mm}
 +\left(\int_{(t,T]}\|G(s)\|_{\mathfrak L(H)}d\xi(s)\right)^2
 \,\Big|\,\mathcal F_t\Bigg],\quad\mathbb P\text{-a.s.},
\label{eq:FV-BSIE-conditional-estimate}
\end{align}
where \(C\) depends only on \(C_L,L,T\).
\end{proposition}
\begin{proof}
Applying \eqref{eq:BSIE-Gamma-pointwise} with \(R=P\), and then using
\eqref{eq:BSIE-F-Lipschitz}, gives, for every deterministic \(r\in[0,T]\),
\begin{align}
\|P(r)\|^2\le C\mathbb E\Bigg[
 &\|\Xi\|^2+\int_r^T\|F(s,0)\|^2ds+
 \left(\int_{(r,T]}\|G(s)\|d\xi(s)\right)^2
 \notag\\
 &+\int_r^T\|P(s)\|^2ds\,\Big|\,\mathcal F_r\Bigg].
\label{eq:BSIE-preconditional-weighted}
\end{align}
Fix \(t\in[0,T]\) and \(A\in\mathcal F_t\). For each \(r\ge t\), we multiply
\eqref{eq:BSIE-preconditional-weighted} by \(I_A\) and take expectations.
Since \(A\in\mathcal F_r\), the tower property and the bound
\(\int_{(r,T]}\|G(s)\|d\xi(s)\le\int_{(t,T]}\|G(s)\|d\xi(s)\)
give
\begin{equation}
\mathbb E[I_A\|P(r)\|^2]
\le C\mathbb E\!\left[I_A\!\left(
\|\Xi\|^2+\int_t^T\|F(s,0)\|^2ds+
\left(\int_{(t,T]}\|G(s)\|d\xi(s)\right)^2\right)\right]
+C\int_r^T\mathbb E[I_A\|P(s)\|^2]ds.
\label{eq:BSIE-tested-Gronwall-bound}
\end{equation}
We apply the backward Gronwall inequality to
\(r\mapsto\mathbb E[I_A\|P(r)\|^2]\) in
\eqref{eq:BSIE-tested-Gronwall-bound} and evaluate the bound at \(r=t\).
Since \(A\in\mathcal F_t\) is arbitrary, this gives
\eqref{eq:FV-BSIE-conditional-estimate}.
\end{proof}

For every \(t\in[0,T]\) and \(u,v\in H\), the
operator-valued identity \eqref{eq:FV-BSIE-fixed-point}, together with
Lemma~\ref{lem:BSIE-fixed-time-aggregation}, yields
\begin{align}
	&\langle P(t)u,v\rangle_H
	\notag\\
	&=
	\mathbb E\Bigg[
	\langle\Xi L(t,T)u,L(t,T)v\rangle_H
	+\int_t^T
	\langle F(s,P(s))L(t,s)u,L(t,s)v\rangle_H\,ds
	\notag\\
	&\hspace{27mm}
	+\int_{(t,T]}
	\langle G(s)L(t,s-)u,L(t,s-)v\rangle_H\,d\xi(s)
	\,\Bigg|\,\mathcal F_t\Bigg],
	\quad\mathbb P\text{-a.s.}
	\label{eq:BSIE-Ito-random-duality}
\end{align}
By the locality property in \((C1)\), the scalar representation
\eqref{eq:BSIE-Ito-random-duality} extends first to
\(\mathcal F_t\)-simple \(H\)-valued random variables and then, by
\(L^4\)-approximation together with \eqref{eq:BSIE-random-maximal-estimate}
and the integrability conditions
\[
\begin{aligned}
\Xi\in L_w^2(\mathcal F_T;\mathfrak L(H)),\quad
F(\cdot,P(\cdot))\in L^2_{\mathbb F,w}(0,T;\mathfrak L(H)),\quad
G\in L^{1,2}_{\mathbb F,\xi,w}(0,T;\mathfrak L(H)),
\end{aligned}
\]
to all \(u,v\in L^4(\mathcal F_t;H)\).

The following weak right-continuity property of \(P\) is needed in the
short-time moving-spike estimate.
\begin{lemma}
\label{lem:FV-BSIE-weak-right-continuity}
Assume \((C1)\)--\((C2)\), and let \(P\) be the solution of
BSIE \eqref{eq:FV-BSIE}. Let \(\alpha\ge1\), and suppose that
\begin{equation}
\begin{aligned}
(\Xi,F(\cdot,0),G)\in{}&
L_w^{2\alpha}(\mathcal F_T;\mathfrak L(H))
\times L^{2,2\alpha}_{\mathbb F,w}(0,T;\mathfrak L(H))
\times L^{1,2\alpha}_{\mathbb F,\xi,w}(0,T;\mathfrak L(H)).
\end{aligned}
\label{eq:BSIE-Ito-high-moment-data}
\end{equation}
Assume, in addition, that there exists \(C_{L,\alpha}>0\) such that,
for every \(t\in[0,T]\) and every
\(\eta\in L^{4\alpha}(\mathcal F_t;H)\),
\begin{equation}
\mathbb E\left[
\sup_{t\le s\le T}
\|L(t,s)\eta\|_H^{4\alpha}
\,\middle|\,\mathcal F_t
\right]
\le
C_{L,\alpha}\|\eta\|_H^{4\alpha},
\quad\mathbb P\text{-a.s.}
\label{eq:BSIE-high-order-maximal-estimate}
\end{equation}
Then the solution satisfies
\begin{equation}
\begin{aligned}
\sup_{0\le t\le T}
\|P(t)\|_{L_w^{2\alpha}(\mathcal F_t;\mathfrak L(H))}<\infty,
\quad F(\cdot,P(\cdot))
\in L^{2,2\alpha}_{\mathbb F,w}(0,T;\mathfrak L(H)).
\end{aligned}
\label{eq:unbounded-xi-BSIE-high-moments}
\end{equation}
Moreover, for every \(r\in[0,T)\) and
\(u,v\in L^{4\alpha}(\mathcal F_r;H)\),
\begin{equation}
\langle P(r+h)u,v\rangle_H
\longrightarrow
\langle P(r)u,v\rangle_H
\quad\text{in }L^\alpha(\Omega),
\quad h\downarrow0.
\label{eq:FV-BSIE-weak-right-continuity}
\end{equation}
\end{lemma}
The proof is given in
Appendix~\ref{subsec:proof-FV-BSIE-weak-right-continuity}.
\subsection{An It\^o-type formula for the operator-valued BSIE}
\label{subsec:FV-BSIE-Ito}
To establish the It\^o-type formula, we take \(L\) to be the solution
operator of the following homogeneous linear SEE.
Let \(A,B\) satisfy \((A1)\), and let
\(D:[0,T]\times\Omega\to\mathfrak L(H;H)\)
be weakly progressively measurable and uniformly bounded: there exists
\(C_D\ge0\) such that
\begin{equation}\label{bound-D-condition}
\|D(t)x\|_H\le C_D\|x\|_H,
\quad (t,\omega,x)\in[0,T]\times\Omega\times H.
\end{equation}
Fix \(\alpha\ge1\), choose \(\kappa>2\) and
\(\beta\ge\max\{\beta_{2\alpha},\beta_{4\alpha}\}\) in the forward
weight, and assume
\begin{equation}
\mathfrak E_{4\alpha\kappa,C_D}(\xi)<\infty.
\label{eq:BSIE-flow-exponential-condition}
\end{equation}
For \(r\in[0,T]\) and \(\eta\in L^{2m}(\mathcal F_r;H)\),
\(1\le m\le2\alpha\), let
\(X^{r,\eta}\) be the unique solution on \([r,T]\) of
\begin{equation}
\left\{
\begin{aligned}
 dX^{r,\eta}(s)
 &=A(s)X^{r,\eta}(s)\,ds
   +B(s)X^{r,\eta}(s)\,dW(s)
   +D(s)X^{r,\eta}(s-)\,d\xi(s),
   \quad s\in[r,T],
 \\
 X^{r,\eta}(r)&=\eta.
\end{aligned}
\right.
\label{eq:BSIE-deterministic-start-homogeneous-SEE}
\end{equation}
We set 
\begin{equation}
L(r,s)\eta:=X^{r,\eta}(s),
\quad r\le s\le T.
\label{eq:SEE-deterministic-start-L-definition}
\end{equation}
Proposition~\ref{lem:SEE-joint-cadlag-flow-version} in
Appendix~\ref{subsec:homogeneous-SEE-evolution-family} proves that
\(L\) satisfies \((C1)\) and
\eqref{eq:SEE-deterministic-start-conditional-estimate}. In particular,
\eqref{eq:BSIE-random-maximal-estimate} and
\eqref{eq:BSIE-high-order-maximal-estimate} follow by taking
\(m=2\) and \(m=2\alpha\), respectively.

Let \(\zeta\in L_{\mathbb F}^{8\alpha}(0,T;H)\).
For \(t_0\in[0,T)\) and \(0<\rho<T-t_0\), set
\(E_\rho=[t_0,t_0+\rho),\)
and let \(x^\rho\) be the solution of
\begin{equation}
	\left\{
	\begin{aligned}
		dx^\rho(t)
		&=
		A(t)x^\rho(t)\,dt
		+
		\big[B(t)x^\rho(t)+\zeta(t)I_{E_\rho}(t)\big]\,dW(t)
		\\
		&\quad+
		D(t)x^\rho(t-)\,d\xi(t),
		\quad t\in[0,T],
		\\
		x^\rho(0)&=0.
	\end{aligned}
	\right.
	\label{eq:BSIE-Ito-xrho-general}
\end{equation}
The following identity will be used to evaluate the quadratic term
in the spike expansion.
\begin{theorem}
	\label{thm:FV-BSIE-Ito}
	Assume \((A1)\) and \((C2)\), and let \(D\) be weakly
	progressively measurable with \(\|D(t)\|_{\mathfrak L(H)}\le C_D\).
	Let \(L\) be defined by
	\eqref{eq:SEE-deterministic-start-L-definition} and \(P\) be the
	solution of \eqref{eq:FV-BSIE}. Let \(\alpha\ge1\), and suppose that
	\eqref{eq:BSIE-Ito-high-moment-data} and
    \eqref{eq:BSIE-flow-exponential-condition} hold. Let
    \(\zeta\in L_{\mathbb F}^{8\alpha}(0,T;H)\).
    Then, for Lebesgue-almost every \(t_0\in[0,T)\), there exists a unique
    pair of progressively measurable real-valued processes
    \((\sigma^\rho,Z^\rho)\) such that
	\((\sigma^\rho,Z^\rho) \in L^\alpha_{\mathbb F}(0,T) \times L^{2,\alpha}_{\mathbb F}(0,T)\)
    and
    \begin{align}
        &\big\langle
        P(t)x^\rho(t),
        x^\rho(t)
        \big\rangle_H
        +
        \sigma^\rho(t)=
        \big\langle
        \Xi x^\rho(T),
        x^\rho(T)
        \big\rangle_H
        \notag\\
        &\quad+
        \int_t^T
        \Big[
        \big\langle
        F(s,P(s))x^\rho(s),
        x^\rho(s)
        \big\rangle_H
        -
        \big\langle
        P(s)\zeta(s),
        \zeta(s)
        \big\rangle_H
        I_{E_\rho}(s)
        \Big]\,ds
        \notag\\
        &\quad+
        \int_{(t,T]}
        \big\langle
        G(s)x^\rho(s-),
        x^\rho(s-)
        \big\rangle_H\,d\xi(s)-
        \int_t^T
        Z^\rho(s)\,dW(s),
        \label{eq:FV-BSIE-Ito}
    \end{align}
    up to evanescence.
    Moreover,
    \begin{equation}
        \sup_{0\le t\le T}
        \mathbb E|\sigma^\rho(t)|^\alpha
        =
        o(\rho^\alpha),
        \quad
        \rho\downarrow0,
        \label{eq:FV-BSIE-Ito-sigma}
    \end{equation}
    and
    \begin{equation}
        \mathbb E
        \left[
        \left(
        \int_0^T
        |Z^\rho(s)|^2\,ds
        \right)^{\alpha/2}
        \right]
        =
        O(\rho^\alpha),
        \quad
        \rho\downarrow0.
        \label{eq:FV-BSIE-Ito-Z}
    \end{equation}
\end{theorem}

The proof is given in Appendix~\ref{subsec:BSIE-Ito-proofs}.

\section{Singular stochastic control and maximum principles}
\label{sec:control-problem}

\subsection{Formulation of the control problem}
\label{subsec:control-formulation}
Let \(U\) be a nonempty separable metric space with metric
\(d\). Fix a reference point \(u_0\in U\) and write
\(|u|_U:=d(u,u_0)\).

We impose each of the polynomial-moment and conditional
exponential-moment requirements below for every \(m\ge1\).
The set of admissible control pairs is
\begin{align}
    \mathcal U[0,T]
    :=&
    \bigg\{
    (u,\xi):
    \ u \ \text{is a \(U\)-valued progressively measurable process satisfying} \ \mathbb E\int_0^T |u(t)|_U^{m}\,dt<\infty,     \notag\\
    & \xi \ \text{is a real-valued adapted nondecreasing c\`adl\`ag process with} \  \xi(0)=0,\ \xi(T)\in L^m(\mathcal F_T),
    \label{eq:admissible-control-set}\\
	   & \text{and}\
 \mathfrak E_{m,L_\gamma}(\xi)<\infty, \ \text{for every}\ m \geq 1              \notag
    \bigg\}.
\end{align}
Here \(L_\gamma\) is the uniform Lipschitz constant of the singular
coefficient with respect to the state variable. The quantity
\(\mathfrak E_{m,L_\gamma}(\xi)\) is defined in
\eqref{eq:unbounded-xi-conditional-exponential-moment}.
The conditional exponential-moment requirement holds automatically when \(L_\gamma=0\), since then
\(\mathfrak E_{m,0}(\xi)=1\).

For $x_0\in H$ and each \((u,\xi)\in\mathcal U[0,T]\), the state process \(x^{u,\xi}\) satisfies
\begin{equation}
\begin{aligned}
    x^{u,\xi}(t)
    &=
    x_0
    +
    \int_0^t
    \big[
        A(s)x^{u,\xi}(s)
        +
        a(s,x^{u,\xi}(s),u(s))
    \big]\,ds
    \\
    &\quad
    +
    \int_0^t
    \big[
        B(s)x^{u,\xi}(s)
        +
        b(s,x^{u,\xi}(s),u(s))
    \big]\,dW(s)
    +
    \int_{(0,t]}
        \gamma(s,x^{u,\xi}(s-))\,d\xi(s),
    \quad t\in[0,T].
\end{aligned}
\label{eq:controlled-state-equation}
\end{equation}
The cost functional is defined by
\begin{equation}
    J(u,\xi)
    :=
    \mathbb E\left[
        h(x^{u,\xi}(T))
        +
        \int_0^T
            k(t,x^{u,\xi}(t),u(t))\,dt
        +
        \int_{(0,T]}
            g(t,x^{u,\xi}(t-))\,d\xi(t)
    \right].
    \label{eq:cost-functional}
\end{equation}
The nonlinear coefficients are
\[
    a,b:\Omega\times[0,T]\times H\times U\to H,
    \quad
    k:\Omega\times[0,T]\times H\times U\to\mathbb R,
\] 
\[
    \gamma:\Omega\times[0,T]\times H\to H,
    \quad
    g:\Omega\times[0,T]\times H\to\mathbb R,
    \quad
    h:\Omega\times H\to\mathbb R.
\]
The control problem is to minimize \(J(u,\xi)\) over
\((u,\xi)\in\mathcal U[0,T]\).

Let
\((\bar u,\bar\xi)\in\mathcal U[0,T]\) be an optimal control and write
\(\bar x:=x^{\bar u,\bar\xi}\) for the corresponding state.

The following differentiability and growth conditions permit second-order
expansions of the state and cost and ensure that the cost is finite for
every admissible control.
\begin{assumption}
\label{ass:control-coefficients}
The maps \(a,b,k\) are
$
    \operatorname{Prog}
    \otimes\mathcal B(H)
    \otimes\mathcal B(U)
$-measurable, \(\gamma,g\) are
$
    \operatorname{Prog}\otimes\mathcal B(H)
$-measurable, and \(h\) is
$
    \mathcal F_T\otimes\mathcal B(H)
$-measurable.
For every \((t,\omega,u)\), the mappings
$
    a(t,\omega,\cdot,u),\
    b(t,\omega,\cdot,u),\
    k(t,\omega,\cdot,u)
$
are twice continuously Fr\'echet differentiable on \(H\). For every
\((t,\omega)\),
$
    \gamma(t,\omega,\cdot),\
    g(t,\omega,\cdot)
$
are twice continuously Fr\'echet differentiable on \(H\), and
\(h(\omega,\cdot)\in C^2(H;\mathbb R)\). Moreover, for every
\((t,\omega,x)\), the mappings
$
    a(t,\omega,x,\cdot),\
    b(t,\omega,x,\cdot),\
    k(t,\omega,x,\cdot)
$
are continuous on \(U\).
The first- and second-order derivatives of $a,b,k,\gamma,g,h$ with respect to \(x\) are
uniformly bounded.  
$
    h(\cdot,0)\in L^1(\mathcal F_T),
$
and there exists \(C>0\) such that for all \((t,\omega,u)\)
\[
\|a(t,0,u)\|_H
    +
    \|b(t,0,u)\|_H
    +
    |k(t,0,u)|
    \le
    C\bigl(1+|u|_U\bigr),\quad
\|\gamma(t,0)\|_H+|g(t,0)|\le C.
\]

\end{assumption}

For every admissible control \((u,\xi)\),
Corollary~\ref{thm:FV-SEE-wellposed} yields a unique solution of
\eqref{eq:controlled-state-equation}. 
Proposition~\ref{prop:unbounded-xi-SEE-estimate} gives, for
all \(m\ge1\),
\begin{equation}
 x^{u,\xi}\in
 \mathbb D^{2m}_{\mathbb F}(0,T;H)
 \cap L^{2,2m}_{\mathbb F}(0,T;V).
 \label{eq:controlled-state-high-order-moments}
\end{equation}

Along the optimal trajectory \((\bar x,\bar u,\bar\xi)\), set, for $\psi=a,b,k$,
\[\bar \psi_x(t):=\psi_x(t,\bar x(t),\bar u(t)),\quad 
\bar \psi_{xx}(t):=\psi_{xx}(t,\bar x(t),\bar u(t)),
\]
and, for $\phi=\gamma,g$, set
\[
    \bar\phi(t):=\phi(t,\bar x(t-)),
    \quad
    \bar\phi_x(t):=\phi_x(t,\bar x(t-)),
    \quad
    \bar\phi_{xx}(t):=\phi_{xx}(t,\bar x(t-)).
\]

Set
\[
\bar A(t):=A(t)+\bar a_x(t),\qquad
\bar B(t):=B(t)+\bar b_x(t).
\]

\subsection{Variational equations and estimates}
\label{subsec:variational-equations}
We first perturb the regular control while keeping the singular component fixed.

Let \(v(\cdot)\) be an admissible regular control.
For \(t_0\in[0,T)\) and \(0<\rho<T-t_0\), put
\(E_\rho:=[t_0,t_0+\rho)\). The spike perturbation is
\[
    u^\rho(t)
    :=
    \begin{cases}
        v(t), & t\in E_\rho,\\
        \bar u(t), & t\notin E_\rho.
    \end{cases}
\]
Then \(u^\rho\) is admissible. Let
\(x^\rho:=x^{u^\rho,\bar\xi}\) denote the corresponding state process.
For $\psi=a,b,k$, write
\[
 \delta \psi(t):=\psi(t,\bar x(t),v(t))-\psi(t,\bar x(t),\bar u(t)), \quad \delta \psi_x(t):=\psi_x(t,\bar x(t),v(t))-\psi_x(t,\bar x(t),\bar u(t)).
\]

The first-order variational equation is
\begin{equation}
\left\{
\begin{aligned}
    dx_1^\rho(t)
    &=
    \bar A(t)x_1^\rho(t)\,dt
    +
    \big[
        \bar B(t)x_1^\rho(t)
        +
        \delta b(t)I_{E_\rho}(t)
    \big]\,dW(t)
    +
    \bar\gamma_x(t)x_1^\rho(t-)\,d\bar\xi(t),
    \quad t\in[0,T],
    \\
    x_1^\rho(0)&=0.
\end{aligned}
\right.
\label{eq:first-variation-regular}
\end{equation}
The second-order variational equation is
\begin{equation}
\left\{
\begin{aligned}
    dx_2^\rho(t)
    &=
    \bigg[
        \bar A(t)x_2^\rho(t)
        +
        \delta a(t)I_{E_\rho}(t)
        +
        \frac12
        \bar a_{xx}(t)
        \big(x_1^\rho(t),x_1^\rho(t)\big)
    \bigg]\,dt
    \\
    &\quad
    +
    \bigg[
        \bar B(t)x_2^\rho(t)
        +
        \delta b_x(t)x_1^\rho(t)I_{E_\rho}(t)
        +
        \frac12
        \bar b_{xx}(t)
        \big(x_1^\rho(t),x_1^\rho(t)\big)
    \bigg]\,dW(t)
    \\
    &\quad
    +
    \bigg[
        \bar\gamma_x(t)x_2^\rho(t-)
        +
        \frac12
        \bar\gamma_{xx}(t)
        \big(x_1^\rho(t-),x_1^\rho(t-)\big)
    \bigg]\,d\bar\xi(t),
    \quad t\in[0,T],
    \\
    x_2^\rho(0)&=0.
\end{aligned}
\right.
\label{eq:second-variation-regular}
\end{equation}
The next lemma estimates \(x_1^\rho\), \(x_2^\rho\), and
\(x^\rho-\bar x-x_1^\rho-x_2^\rho\) at the orders needed
for the cost expansion.
\begin{lemma}
\label{lem:regular-variation-estimates}
Assume \((A1)\) and Assumption~\ref{ass:control-coefficients}.
For Lebesgue-almost every \(t_0\in[0,T)\) and every
\(\alpha\ge1\), there is a constant \(C_\alpha\) such that
\begin{align}
 \mathbb E\sup_{t\le T}\|x_1^\rho(t)\|_H^{2\alpha}
 &\le C_\alpha\rho^\alpha,\label{eq:first-variation-estimate}\\
 \mathbb E\sup_{t\le T}\|x_2^\rho(t)\|_H^{2\alpha}
 &\le C_\alpha\rho^{2\alpha},\label{eq:second-variation-estimate}\\
 \mathbb E\sup_{t\le T}\|x^\rho(t)-\bar x(t)-x_1^\rho(t)\|_H^{2\alpha}
 &\le C_\alpha\rho^{2\alpha},\label{eq:first-remainder-estimate}\\
 \mathbb E\sup_{t\le T}
 \|x^\rho(t)-\bar x(t)-x_1^\rho(t)-x_2^\rho(t)\|_H^{2\alpha}
 &=o(\rho^{2\alpha}).\label{eq:second-remainder-estimate}
\end{align}
\end{lemma}

The proof is given in Appendix~\ref{subsec:variational-adjoint-estimate-proofs}.

We now vary the singular component while keeping the regular control fixed.
Let \(\eta\) be an admissible singular control, and put
\(\nu=\eta-\bar\xi\) and \(K=\bar\xi+\eta\). Then \(\nu\) is
adapted, c\`adl\`ag, and of finite variation, and
\begin{equation}
 d|\nu|_v\le dK,\quad |\nu|_v(T)\le K(T)=\eta(T)+\bar\xi(T).
 \label{eq:nu-total-variation-bound}
\end{equation}
In particular, \(K(T)\in L^m(\mathcal F_T)\) for every \(m\ge1\).
Conditional H\"older's inequality yields
\[
 \mathfrak E_{m,L_\gamma}(K)
 \le\mathfrak E_{2m,L_\gamma}(\bar\xi)^{1/2}
      \mathfrak E_{2m,L_\gamma}(\eta)^{1/2}<\infty.
\]
For \(0\le\varepsilon\le1\), set
\[
 \xi^\varepsilon=(1-\varepsilon)\bar\xi+\varepsilon\eta,
 \quad u^\varepsilon=\bar u,
 \quad x^\varepsilon=x^{\bar u,\xi^\varepsilon}.
\]
Since \(d\xi^\varepsilon\le dK\), these perturbations are admissible,
and the moment and conditional exponential bounds for these perturbations
are uniform in \(\varepsilon\).

The first variation with respect to the singular component is the solution of
\begin{equation}
\left\{
\begin{aligned}
    dy(t)
    &=
    \bar A(t)y(t)\,dt
    +
    \bar B(t)y(t)\,dW(t)
    +
    \bar\gamma_x(t)y(t-)\,d\bar\xi(t)
    +
    \bar\gamma(t)\,d\nu(t),
    \quad t\in[0,T],
    \\
    y(0)&=0,
\end{aligned}
\right.
\label{eq:first-variation-singular}
\end{equation}

The next lemma justifies differentiation of the state with respect to
the singular control in the moment spaces needed for the cost expansion.
\begin{lemma}
\label{lem:singular-variation-estimate}
Under \((A1)\) and Assumption~\ref{ass:control-coefficients}, for every
\(\alpha\ge1\),
\(y\in\mathbb D^{2\alpha}_{\mathbb F}(0,T;H) \cap L^{2,2\alpha}_{\mathbb F}(0,T;V).\)
Moreover,
\[
 \mathbb E\sup_{t\le T}
 \left\|(x^\varepsilon(t)-\bar x(t))/\varepsilon-y(t)\right\|_H^{2\alpha}
 \le C_\alpha\varepsilon^{2\alpha}.
\]
\end{lemma}

The proof is given in Appendix~\ref{subsec:variational-adjoint-estimate-proofs}.

\subsection{Adjoint equations and duality relations}
\label{subsec:adjoint-equations}
The adjoint equations express the linear and quadratic terms in the
cost expansion directly through the control perturbations.
The first-order adjoint equation is the backward stochastic evolution equation
with a finite-variation driver
\begin{equation}
\left\{
\begin{aligned}
    p(t)
    &=
    h_x(\bar x(T))
    +
    \int_t^T
    \Big[
        \bar A(s)^*p(s)
        +
        \bar B(s)^*q(s)
        +
        \bar k_x(s)
    \Big]\,ds
    \\
    &\quad
    -
    \int_t^T q(s)\,dW(s)
    +
    \int_{(t,T]}
    \Big[
        \bar\gamma_x(s)^*p(s)
        +
        \bar g_x(s)
    \Big]\,d\bar\xi(s),
    \quad t\in[0,T].
\end{aligned}
\right.
\label{eq:first-adjoint}
\end{equation}
Fix \(\kappa>2\). Theorem~\ref{thm:unbounded-xi-BSEE-finite-order}
yields a solution \((p,q)\) of \eqref{eq:first-adjoint}, unique
in the weighted quadratic class
\(\|(p,q)\|_{\theta_{\beta,L_\gamma}^{0}}<\infty\), for sufficiently
large deterministic \(\beta\). Boundedness of the adjoint data and
the admissibility of \(\bar\xi\) give the required data moments
of every finite order. Applying
\eqref{eq:unbounded-xi-unweighted-BSEE-estimate} to this solution
gives, for every \(\alpha\ge1\),
\[
 (p,q)\in
 \bigl(\mathbb D^{2\alpha}_{\mathbb F}(0,T;H)
       \cap L^{2,2\alpha}_{\mathbb F}(0,T;V)\bigr)
 \times L^{2,2\alpha}_{\mathbb F}(0,T;H)
\]
and
\begin{align}
 &\mathbb E\left[
   \sup_{0\le t\le T}\|p(t)\|_H^{2\alpha}
   +\left(\int_0^T\|p(t)\|_V^2\,dt\right)^\alpha
   +\left(\int_0^T\|q(t)\|_H^2\,dt\right)^\alpha
 \right]\notag\\
 &\quad\le C_\alpha
 \left(\mathbb E e^{2\alpha\kappa L_\gamma\bar\xi(T)}\right)^{1/2}
 \left(1+\mathbb E\bar\xi(T)^{4\alpha}\right)^{1/2}<\infty.
 \label{eq:first-adjoint-moment-estimate}
\end{align}

We identify the second derivatives of real-valued coefficients with
bounded self-adjoint operators on \(H\). 
For \(\psi=a,b,\gamma\) and \(z\in H\), let
\(\bar\psi_{xx}(t)^*z\in\mathfrak L(H)\)
denote the bounded operator determined by
\[
\big\langle(\bar\psi_{xx}(t)^*z)e,f\big\rangle_H
=
\big\langle z,\bar\psi_{xx}(t)(e,f)\big\rangle_H,
\quad e,f\in H.
\]

The second-order adjoint coefficients are
\begin{align*}
    \bar H_{xx}(t)
    :=\bar k_{xx}(t)+\bar a_{xx}(t)^*p(t)
    +\bar b_{xx}(t)^*q(t),\quad
    \bar G_{xx}(t)
    :=\bar g_{xx}(t)+\bar\gamma_{xx}(t)^*p(t).
\end{align*}
The second-order adjoint equation is the operator-valued BSIE
\begin{equation}
\begin{aligned}
    P(t)
    =
    \mathbb E\bigg[
    &L(t,T)^*
        h_{xx}(\bar x(T))
        L(t,T)
   +
    \int_t^T
        L(t,s)^*
        \bar H_{xx}(s)
        L(t,s)\,ds
    \\
    &+
    \int_{(t,T]}
        L(t,s-)^*
        \bar G_{xx}(s)
        L(t,s-)\,d\bar\xi(s)
    \,\bigg|\,\mathcal F_t
    \bigg],
    \quad t\in[0,T].
\end{aligned}
\label{eq:second-adjoint}
\end{equation}
Here \(L(t,s)\) is the stochastic evolution operator generated by the
homogeneous linearized equation
\begin{equation}
\left\{
\begin{aligned}
    dX(s)
    &=
    \bar A(s)X(s)\,ds
    +
    \bar B(s)X(s)\,dW(s)
    +
    \bar\gamma_x(s)X(s-)\,d\bar\xi(s),
    \quad s\in[t,T],
    \\
    X(t)&=\eta.
\end{aligned}
\right.
\label{eq:homogeneous-linear-flow}
\end{equation}
Thus \(L(t,s)\eta=X(s)\) and \(L(t,s-)\eta=X(s-)\).
The conditional flow estimates in
\eqref{eq:SEE-deterministic-start-conditional-estimate}, applied to
\((\bar A,\bar B,\bar\gamma_x,\bar\xi)\), give the required moment
bounds for \(L\).
The adjoint data satisfy
\begin{equation}
 \|\bar H_{xx}(s)\|_{\mathfrak L(H)}
 \le C(1+\|p(s)\|_H+\|q(s)\|_H),\quad
 \|\bar G_{xx}(s)\|_{\mathfrak L(H)}
 \le C(1+\|p(s)\|_H),
\label{eq:second-adjoint-data-bounds}
\end{equation}
so Theorem~\ref{thm:unbounded-xi-BSIE} gives a weakly progressively
measurable solution \(P\in\mathcal H_\beta\) of
\eqref{eq:second-adjoint}, unique up to modification, for sufficiently
large \(\beta\). Proposition~\ref{cor:FV-BSIE-conditional-estimate} gives
\begin{align}
 \|P(t)\|_{\mathfrak L(H)}^2
 &\le C\mathbb E\Bigg[
    1+\int_t^T(\|p(s)\|_H^2+\|q(s)\|_H^2)\,ds\notag\\
 &\hspace{28mm}
    +\left(\int_{(t,T]}(1+\|p(s)\|_H)\,d\bar\xi(s)\right)^2
    \,\Big|\,\mathcal F_t\Bigg],\quad\mathbb P\text{-a.s.}
 \label{eq:second-adjoint-conditional-moment-bound}
\end{align}
Conditional Jensen's inequality then gives, for every \(\alpha\ge1\),
\begin{align}
 \sup_{0\le t\le T}\mathbb E\|P(t)\|_{\mathfrak L(H)}^{2\alpha}
 &\le C_\alpha\mathbb E\Bigg[
 (1+\bar\xi(T)^{2\alpha})
 \left(1+\sup_{s\le T}\|p(s)\|_H^{2\alpha}\right)
 +\left(\int_0^T\|q(s)\|_H^2\,ds\right)^\alpha
 \Bigg]<\infty.
 \label{eq:second-adjoint-uniform-moment-bound}
\end{align}

The next identity expresses the pairings with
\(h_x(\bar x(T)),\bar k_x,\bar g_x\) in terms of \(p,q\)
and the inhomogeneous coefficients of the variational equation.
\begin{lemma}
\label{lem:first-order-duality}
Let
$
    f\in L^{1,2}_{\mathbb F}(0,T;H),\ 
    g\in L^2_{\mathbb F}(0,T;H),\
    \eta\in L^{1,2}_{\mathbb F,\bar\xi}(0,T;H).
	$
Let \(\nu\) be an adapted c\`adl\`ag finite-variation process with
\(\nu(0)=0\), and let
$
    \zeta\in L^{1,2}_{\mathbb F,\nu}(0,T;H).
$
Let
\(X\in\mathbb D^2_{\mathbb F}(0,T;H) \cap L^2_{\mathbb F}(0,T;V)\)
be a solution of
\begin{equation}
\left\{
\begin{aligned}
    dX(t)
    &=
    \big[
        \bar A(t)X(t)+f(t)
    \big]dt
    +
    \big[
        \bar B(t)X(t)+g(t)
    \big]dW(t)
    \\
    &\quad+
    \big[
        \bar\gamma_x(t)X(t-)+\eta(t)
    \big]d\bar\xi(t)
    +
    \zeta(t)\,d\nu(t),
    \quad t\in[0,T],
    \\
    X(0)&=0.
\end{aligned}
\right.
\label{eq:linear-forward-duality}
\end{equation}
Then
\begin{align}
    &\mathbb E\bigg[
        \big\langle h_x(\bar x(T)),X(T)\big\rangle_H
        +
        \int_0^T
            \big\langle \bar k_x(t),X(t)\big\rangle_H\,dt
        +
        \int_{(0,T]}
            \big\langle \bar g_x(t),X(t-)\big\rangle_H\,d\bar\xi(t)
    \bigg]
    \notag\\
    &=
    \mathbb E\bigg[
        \int_0^T
            \big\langle p(t),f(t)\big\rangle_H\,dt
        +
        \int_0^T
            \big\langle q(t),g(t)\big\rangle_H\,dt
        +
        \int_{(0,T]}
            \big\langle p(t),\eta(t)\big\rangle_H\,d\bar\xi(t)
        +
        \int_{(0,T]}
            \big\langle p(t),\zeta(t)\big\rangle_H\,d\nu(t)
    \bigg].
    \label{eq:first-order-duality}
\end{align}
\end{lemma}

\begin{proof}
At a jump time \(s\),
\[
    X(s)
    =
    X(s-)
    +
    \big[
        \bar\gamma_x(s)X(s-)+\eta(s)
    \big]\Delta\bar\xi(s)
    +
    \zeta(s)\Delta\nu(s),
\]
whereas
\[
    p(s-)
    =
    p(s)
    +
    \big[
        \bar\gamma_x(s)^*p(s)+\bar g_x(s)
    \big]\Delta\bar\xi(s).
\]
Therefore
\begin{equation}
 \big\langle p(s),X(s)\big\rangle_H -\big\langle p(s-),X(s-)\big\rangle_H = \big[ \langle p(s),\eta(s)\rangle_H -\langle\bar g_x(s),X(s-)\rangle_H \big]\Delta\bar\xi(s) + \langle p(s),\zeta(s)\rangle_H\Delta\nu(s).
\label{eq:first-duality-jump-cancellation}
\end{equation}
Applying the product formula to \(\langle p(t),X(t)\rangle_H\)
and using \eqref{eq:first-adjoint} and
\eqref{eq:linear-forward-duality}, we obtain
\begin{align}
    d\langle p(t),X(t)\rangle_H
    &=
    -\big\langle \bar k_x(t),X(t)\big\rangle_Hdt
    +\big\langle p(t),f(t)\big\rangle_Hdt
    +\big\langle q(t),g(t)\big\rangle_Hdt
    \notag\\
    &\quad
    -\big\langle \bar g_x(t),X(t-)\big\rangle_Hd\bar\xi(t)
    +\big\langle p(t),\eta(t)\big\rangle_Hd\bar\xi(t)
    +\big\langle p(t),\zeta(t)\big\rangle_Hd\nu(t)
    +d\mathcal M(t),
    \label{eq:product-p-X}
\end{align}
where
\[
    \mathcal M(t)
    :=
    \int_0^t
    \Big[
        \big\langle p(s),\bar B(s)X(s)+g(s)\big\rangle_H
        +\langle q(s),X(s)\rangle_H
    \Big]dW(s).
\]
The assumption \(X\in\mathbb D^2_{\mathbb F}\cap
	L^2_{\mathbb F}(0,T;V)\), estimate
	\eqref{eq:first-adjoint-moment-estimate} for \((p,q)\), and the
bound \(\|\bar B(s)X(s)\|_H\le C\|X(s)\|_V\) give
\begin{equation}
\begin{aligned}
    \mathbb E\langle\mathcal M\rangle_T^{1/2}
    &\le
    C\mathbb E\Bigg[
        \sup_{0\le t\le T}\|p(t)\|_H
        \left(
            \int_0^T
            [\|X(t)\|_V^2+\|g(t)\|_H^2]dt
        \right)^{1/2}
    \\
    &\hspace{2.4cm}
        +
        \sup_{0\le t\le T}\|X(t)\|_H
        \left(
            \int_0^T\|q(t)\|_H^2dt
        \right)^{1/2}
    \Bigg]
    <\infty.
\end{aligned}
\label{eq:first-duality-martingale-integrability}
\end{equation}
Thus \(\mathcal M\) is a uniformly integrable martingale. Integrating
\eqref{eq:product-p-X} over \([0,T]\), using \(X(0)=0\) and
\(p(T)=h_x(\bar x(T))\), and taking expectations gives
\eqref{eq:first-order-duality}.
\end{proof}

\subsection{Stochastic maximum principle}
\label{subsec:SMP}
Define the Hamiltonian associated with the regular control by
\[
\mathcal H(t,x,v,p,q)
:=
\langle p,a(t,x,v)\rangle_H
+
\langle q,b(t,x,v)\rangle_H
+
k(t,x,v),
\]
and the singular Hamiltonian by
\[
\mathcal G(t,x)
:=
\langle p(t),\gamma(t,x)\rangle_H+g(t,x).
\]
Set
$
Z(t):=\mathcal G(t,\bar x(t-)),
\ t\in(0,T].
$
The growth bounds for
\(\gamma,g\) and the moment bounds for \(\bar x,p\) give
\[
 |Z(t)|
 \le C\bigl(1+\sup_{s\le T}\|\bar x(s)\|_H\bigr)
        \bigl(1+\sup_{s\le T}\|p(s)\|_H\bigr),
 \quad t\in(0,T],
\]
where the right-hand side belongs to \(L^m(\mathcal F_T)\) for
every finite \(m\ge1\). In particular,
\(\sup_\tau\|Z(\tau)\|_{L^1(\mathcal F_\tau)}<\infty\), with
\(\tau\) ranging over \((0,T]\)-valued stopping times.
The optional projection theorem
\cite[Chapter~VI, Theorem~43]{DellacherieMeyer1982} therefore gives
an optional projection \({}^oZ\), unique up to evanescence. Since
\(Z(\tau)\in\mathcal F_\tau\),
\[
 {}^oZ(\tau)=\mathbb E[Z(\tau)\mid\mathcal F_\tau]=Z(\tau),
 \quad\mathbb P\text{-a.s.}
\]

The bound on \(Z\) and the moment estimates imply
\(\mathbb E\int_{(0,T]}|Z(t)|\,d\chi(t)<\infty\)
for every admissible singular control or bounded adapted
nondecreasing c\`adl\`ag process \(\chi\) with \(\chi(0)=0\).
The integration property of optional projection gives
\[
\mathbb E\int_{(0,T]}Z(t)\,d\chi(t)
=
\mathbb E\int_{(0,T]}{}^oZ(t)\,d\chi(t).
\]
Taking \(\chi=\eta\) and \(\chi=\bar\xi\) shows that replacing
\(Z\) by \({}^oZ\) preserves the singular variational inequality.
We use this optional modification throughout the theorem and its
proof and continue to denote it by
\(\mathcal G(t,\bar x(t-))\).

The maximum principle combines pointwise minimization in the regular
control with nonnegativity of the singular Hamiltonian and its
vanishing \(d\bar\xi\)-almost everywhere.
\begin{theorem}
\label{thm:SMP}
\label{thm:SMP-unbounded-xi}
Assume \((A1)\) and Assumption~\ref{ass:control-coefficients}, and let
\((\bar u,\bar\xi)\) be an optimal admissible control with corresponding
state \(\bar x\). Let \((p,q)\) be the solution of
the first-order adjoint BSEE \eqref{eq:first-adjoint}, and let \(P\) be the
solution of the second-order adjoint BSIE \eqref{eq:second-adjoint}. Then,
for \((dt\otimes d\mathbb P)\)-a.e. \((t,\omega)\),
\begin{align}
    &\mathcal H(t,\bar x(t),\bar u(t),p(t),q(t))
    \notag\\
    &=
    \inf_{v\in U}
    \bigg\{
        \mathcal H(t,\bar x(t),v,p(t),q(t))
        +
        \frac12
        \Big\langle
            P(t)
            \big[
                b(t,\bar x(t),v)-b(t,\bar x(t),\bar u(t))
            \big],
    \notag\\
    &\hspace{5.7cm}
            b(t,\bar x(t),v)-b(t,\bar x(t),\bar u(t))
        \Big\rangle_H
    \bigg\}.
    \label{eq:SMP-regular-final}
\end{align}
Moreover, 
\begin{equation}
    \mathbb P\left\{
        \mathcal G(t,\bar x(t-))\ge0,
        \quad t\in(0,T]
    \right\}=1,
    \label{eq:SMP-singular-nonnegative}
\end{equation}
and
\begin{equation}
    \mathcal G(t,\bar x(t-))\,d\bar\xi(t)=0,\quad \mathbb P\text{-a.s.}
    \label{eq:SMP-singular-complementarity}
\end{equation}
\end{theorem}
\begin{remark}
	\label{rem:singular-Hamiltonian-optionality}
	In the finite-dimensional formulations of singular stochastic
	control in \cite{BahlaliMezerdi2005,Bahlali2007,JiXue09}, the coefficients
	of the singular terms are deterministic, continuous in time, and
	independent of the state. The adjoint processes entering the singular
	Hamiltonians are continuous and adapted, so the singular Hamiltonians
	are optional.
	
	Here \(\gamma\) and \(g\) may be random and state-dependent,
	with only progressive measurability imposed. The optional projection
	constructed above provides the optional process needed for the
	stopping-time tests that yield
	\eqref{eq:SMP-singular-nonnegative} and
	\eqref{eq:SMP-singular-complementarity}.
\end{remark}

\subsection{Proof of Theorem~\ref{thm:SMP}}
\label{subsec:SMP-proof}

\subsubsection{Regular perturbations}

We first expand the cost under a spike perturbation of the regular control.

\begin{lemma}
\label{lem:regular-cost-expansion}
Assume \((A1)\) and Assumption~\ref{ass:control-coefficients}.
Then
\begin{align}
    &J(u^\rho,\bar\xi)-J(\bar u,\bar\xi)
    =
    \mathbb E\bigg[
        \int_{E_\rho}\delta k(t)\,dt
        +
        \big\langle h_x(\bar x(T)),
            x_1^\rho(T)+x_2^\rho(T)
        \big\rangle_H
        +
        \frac12
        \big\langle
            h_{xx}(\bar x(T))x_1^\rho(T),
            x_1^\rho(T)
        \big\rangle_H
    \notag\\
    &\quad
        +
        \int_0^T
            \big\langle
                \bar k_x(t),
                x_1^\rho(t)+x_2^\rho(t)
            \big\rangle_H\,dt
        +
        \frac12
        \int_0^T
            \big\langle
                \bar k_{xx}(t)x_1^\rho(t),
                x_1^\rho(t)
            \big\rangle_H\,dt
    \notag\\
    &\quad
        +
        \int_{(0,T]}
            \big\langle
                \bar g_x(t),
                x_1^\rho(t-)+x_2^\rho(t-)
            \big\rangle_H\,d\bar\xi(t)
        +
        \frac12
        \int_{(0,T]}
            \big\langle
                \bar g_{xx}(t)x_1^\rho(t-),
                x_1^\rho(t-)
            \big\rangle_H\,d\bar\xi(t)
    \bigg]
    +
    o(\rho).
    \label{eq:regular-cost-expansion}
\end{align}
\end{lemma}
\begin{proof}
	Write
	$
	\Delta x^\rho:=x^\rho-\bar x,
	\
	r_2^\rho
	:=
	\Delta x^\rho-x_1^\rho-x_2^\rho.
	$
	Taylor's formula gives
	\begin{align}
		&J(u^\rho,\bar\xi)-J(\bar u,\bar\xi)
		\notag\\
		&=
		\mathbb E\bigg[
		\int_{E_\rho}\delta k(t)\,dt
		+
		\big\langle
		h_x(\bar x(T)),
		x_1^\rho(T)+x_2^\rho(T)
		\big\rangle_H
		+
		\frac12
		\big\langle
		h_{xx}(\bar x(T))x_1^\rho(T),
		x_1^\rho(T)
		\big\rangle_H
		\notag\\
		&\quad
		+
		\int_0^T
		\big\langle
		\bar k_x(t),
		x_1^\rho(t)+x_2^\rho(t)
		\big\rangle_H\,dt
		+
		\frac12
		\int_0^T
		\big\langle
		\bar k_{xx}(t)x_1^\rho(t),
		x_1^\rho(t)
		\big\rangle_H\,dt
		\notag\\
		&\quad
		+
		\int_{(0,T]}
		\big\langle
		\bar g_x(t),
		x_1^\rho(t-)+x_2^\rho(t-)
		\big\rangle_H\,d\bar\xi(t)
		+
		\frac12
		\int_{(0,T]}
		\big\langle
		\bar g_{xx}(t)x_1^\rho(t-),
		x_1^\rho(t-)
		\big\rangle_H\,d\bar\xi(t)
		\bigg]
		\notag\\
		&\quad
		+
		\mathbb E\bigg[
		R_h^\rho
		+
		\int_0^T R_k^\rho(t)\,dt
		+
		\int_{(0,T]}R_g^\rho(t)\,d\bar\xi(t)
		\bigg],
		\label{eq:regular-cost-expansion-with-remainders}
	\end{align}
	where
	\begin{align*}
		R_h^\rho
		&=
		\big\langle
		h_x(\bar x(T)),r_2^\rho(T)
		\big\rangle_H
		+
		\big\langle
		h_{xx}(\bar x(T))x_1^\rho(T),
		x_2^\rho(T)+r_2^\rho(T)
		\big\rangle_H
		\\
		&\quad+
		\frac12
		\big\langle
		h_{xx}(\bar x(T))
		\big(x_2^\rho(T)+r_2^\rho(T)\big),
		x_2^\rho(T)+r_2^\rho(T)
		\big\rangle_H
		\\
		&\quad+
		\int_0^1(1-\theta)
		\Big\langle
		\Big[
		h_{xx}\big(
		\bar x(T)+\theta\Delta x^\rho(T)
		\big)
		-
		h_{xx}(\bar x(T))
		\Big]\Delta x^\rho(T),
		\Delta x^\rho(T)
		\Big\rangle_H\,d\theta,
	\end{align*}
	\begin{align*}
		R_k^\rho(t)
		&=
		\big\langle
		\bar k_x(t),r_2^\rho(t)
		\big\rangle_H
		+
		I_{E_\rho}(t)
		\Big\langle
		k_x(t,\bar x(t),v(t))
		-
		\bar k_x(t),
		\Delta x^\rho(t)
		\Big\rangle_H
		\\
		&\quad+
		\big\langle
		\bar k_{xx}(t)x_1^\rho(t),
		x_2^\rho(t)+r_2^\rho(t)
		\big\rangle_H
		+
		\frac12
		\big\langle
		\bar k_{xx}(t)
		\big(x_2^\rho(t)+r_2^\rho(t)\big),
		x_2^\rho(t)+r_2^\rho(t)
		\big\rangle_H
		\\
		&\quad+
		\frac12 I_{E_\rho}(t)
		\Big\langle
		\big[
		k_{xx}(t,\bar x(t),v(t))
		-
		\bar k_{xx}(t)
		\big]\Delta x^\rho(t),
		\Delta x^\rho(t)
		\Big\rangle_H
		\\
		&\quad+
		\int_0^1(1-\theta)
		\Big\langle
		\Big[
		k_{xx}\big(
		t,\bar x(t)+\theta\Delta x^\rho(t),u^\rho(t)
		\big)
		-
		k_{xx}\big(
		t,\bar x(t),u^\rho(t)
		\big)
		\Big]\Delta x^\rho(t),
		\Delta x^\rho(t)
		\Big\rangle_H\,d\theta,
	\end{align*}
	and
	\begin{align*}
		R_g^\rho(t)
		&=
		\big\langle
		\bar g_x(t),r_2^\rho(t-)
		\big\rangle_H
		+
		\big\langle
		\bar g_{xx}(t)x_1^\rho(t-),
		x_2^\rho(t-)+r_2^\rho(t-)
		\big\rangle_H
		\\
		&\quad+
		\frac12
		\big\langle
		\bar g_{xx}(t)
		\big(x_2^\rho(t-)+r_2^\rho(t-)\big),
		x_2^\rho(t-)+r_2^\rho(t-)
		\big\rangle_H
		\\
		&\quad+
		\int_0^1(1-\theta)
		\Big\langle
		\Big[
		g_{xx}\big(
		t,\bar x(t-)+\theta\Delta x^\rho(t-)
		\big)
		-
		\bar g_{xx}(t)
		\Big]\Delta x^\rho(t-),
		\Delta x^\rho(t-)
		\Big\rangle_H\,d\theta.
	\end{align*}

	For the integral remainder in \(R_h^\rho\), H\"older's inequality
and the convergence \(\Delta x^\rho(T)\to0\) in probability give
	\begin{align*}
		&\mathbb E
		\int_0^1(1-\theta)
		\left|
		\Big\langle
		\Big[
		h_{xx}\big(
		\bar x(T)+\theta\Delta x^\rho(T)
		\big)
		-
		h_{xx}(\bar x(T))
		\Big]\Delta x^\rho(T),
		\Delta x^\rho(T)
		\Big\rangle_H
		\right|d\theta
		\\
		&\quad\le
		\left(
		\mathbb E\|\Delta x^\rho(T)\|_H^4
		\right)^{1/2}
		\int_0^1(1-\theta)
		\left(
		\mathbb E
		\left\|
		h_{xx}\big(
		\bar x(T)+\theta\Delta x^\rho(T)
		\big)
		-
		h_{xx}(\bar x(T))
		\right\|_{\mathfrak L(H)}^2
		\right)^{1/2}d\theta=o(\rho).
	\end{align*}
	Consequently,
	\begin{equation}
		\mathbb E|R_h^\rho|=o(\rho).
		\label{eq:R-h-regular}
	\end{equation}
	
	For the integral remainder in \(R_k^\rho\), since
\(\Delta x^\rho\rightarrow 0\) in \(dt\otimes d\mathbb P\)-measure,
we obtain
	\begin{align*}
		&\mathbb E\int_0^T\int_0^1(1-\theta)
		\Big|
		\Big\langle
		\Big[
		k_{xx}\big(
		t,\bar x(t)+\theta\Delta x^\rho(t),u^\rho(t)
		\big)
		-
		k_{xx}\big(
		t,\bar x(t),u^\rho(t)
		\big)
		\Big]\Delta x^\rho(t),
		\Delta x^\rho(t)
		\Big\rangle_H
		\Big|\,d\theta\,dt
		\\
		&\quad\le
		\left(
		\mathbb E\int_0^T
		\|\Delta x^\rho(t)\|_H^4\,dt
		\right)^{1/2}
		\\
		&\qquad\quad\times
		\int_0^1(1-\theta)
		\left[
		\mathbb E\int_0^T
		\left\|
		k_{xx}\big(
		t,\bar x(t)+\theta\Delta x^\rho(t),u^\rho(t)
		\big)
		-
		k_{xx}\big(
		t,\bar x(t),u^\rho(t)
		\big)
		\right\|_{\mathfrak L(H)}^2dt
		\right]^{1/2}d\theta=o(\rho).
	\end{align*}
     Consequently,
	\begin{equation}
		\mathbb E\int_0^T|R_k^\rho(t)|\,dt=o(\rho).
		\label{eq:R-k-regular}
	\end{equation}
	
	It remains to estimate \(R_g^\rho\). For every \(\varepsilon>0\),
	\begin{align*}
		\mu^{\bar\xi}
		\big(
		\|\Delta x^\rho(t-)\|_H>\varepsilon
		\big)
		&\le
		\bigl(\mathbb E\bar\xi(T)^2\bigr)^{1/2}
		\mathbb P\left(
		\sup_{0\le t\le T}
		\|\Delta x^\rho(t)\|_H>\varepsilon
		\right)^{1/2}
		\longrightarrow0.
	\end{align*}
	Hence
	$
	\Delta x^\rho(t-)\rightarrow0
	\
	\text{in }\mu^{\bar\xi}\text{-measure}.
	$
    H\"older's inequality then gives
	\begin{align*}
		&\mathbb E\int_{(0,T]}\int_0^1(1-\theta)
		\left|
		\Big\langle
		\Big[
		g_{xx}\big(
		t,\bar x(t-)+\theta\Delta x^\rho(t-)
		\big)
		-
		\bar g_{xx}(t)
		\Big]\Delta x^\rho(t-),
		\Delta x^\rho(t-)
		\Big\rangle_H
		\right|d\theta\,d\bar\xi(t)
		\\
		&\quad\le
		\left[
		\mathbb E\int_{(0,T]}
		\|\Delta x^\rho(t-)\|_H^4
		\,d\bar\xi(t)
		\right]^{1/2}
		\\
		&\qquad\quad\times
		\int_0^1(1-\theta)
		\left[
		\mathbb E\int_{(0,T]}
		\left\|
		g_{xx}\big(
		t,\bar x(t-)+\theta\Delta x^\rho(t-)
		\big)
		-
		\bar g_{xx}(t)
		\right\|_{\mathfrak L(H)}^2
		\,d\bar\xi(t)
		\right]^{1/2}d\theta=o(\rho),
	\end{align*}
	where the first factor satisfies
	\begin{align*}
		\left[
		\mathbb E\int_{(0,T]}
		\|\Delta x^\rho(t-)\|_H^4
		\,d\bar\xi(t)
		\right]^{1/2}
		&\le
		\bigl(\mathbb E\bar\xi(T)^2\bigr)^{1/4}
        \left[\mathbb E\sup_{t\le T}
          \|\Delta x^\rho(t)\|_H^8\right]^{1/4}
		=
		O(\rho).
	\end{align*}
	 Thus,
\begin{equation}
\mathbb E\int_{(0,T]}
		|R_g^\rho(t)|\,d\bar\xi(t)
		=
		o(\rho).
		\label{eq:R-g-regular}
	\end{equation}
	
	Combining
	\eqref{eq:R-h-regular},
	\eqref{eq:R-k-regular}, and
	\eqref{eq:R-g-regular} with
	\eqref{eq:regular-cost-expansion-with-remainders}
	yields
	\eqref{eq:regular-cost-expansion}.
\end{proof}

Applying Lemma~\ref{lem:first-order-duality} to \(X=x_1^\rho\), with
\(f=0\), \(g(t)=\delta b(t)I_{E_\rho}(t)\), \(\eta=0\), and
\(\nu=0\), we obtain
\begin{align}
    &\mathbb E\bigg[
        \big\langle h_x(\bar x(T)),x_1^\rho(T)\big\rangle_H
        +
        \int_0^T
            \big\langle \bar k_x(t),x_1^\rho(t)\big\rangle_H\,dt
        +
        \int_{(0,T]}
            \big\langle \bar g_x(t),x_1^\rho(t-)\big\rangle_H\,d\bar\xi(t)
    \bigg]
    \notag\\
    &=
    \mathbb E\int_{E_\rho}
        \big\langle q(t),\delta b(t)\big\rangle_H\,dt.
    \label{eq:duality-x1}
\end{align}
Applying Lemma~\ref{lem:first-order-duality} to \(X=x_2^\rho\), with
\begin{align*}
f(t)&=
    \delta a(t)I_{E_\rho}(t)
    +
    \frac12\bar a_{xx}(t)
    \big(x_1^\rho(t),x_1^\rho(t)\big),\\
 g(t)&=
    \delta b_x(t)x_1^\rho(t)I_{E_\rho}(t)
    +
    \frac12\bar b_{xx}(t)
    \big(x_1^\rho(t),x_1^\rho(t)\big),
\end{align*}
and
\[
    \eta(t)=
    \frac12\bar\gamma_{xx}(t)
    \big(x_1^\rho(t-),x_1^\rho(t-)\big),
\]
with \(\nu=0\), we obtain
\begin{align}
    &\mathbb E\bigg[
        \big\langle h_x(\bar x(T)),x_2^\rho(T)\big\rangle_H
        +
        \int_0^T
            \big\langle \bar k_x(t),x_2^\rho(t)\big\rangle_H\,dt
        +
        \int_{(0,T]}
            \big\langle \bar g_x(t),x_2^\rho(t-)\big\rangle_H\,d\bar\xi(t)
    \bigg]
    \notag\\
    &=
    \mathbb E\int_{E_\rho}
        \big\langle p(t),\delta a(t)\big\rangle_H\,dt
    +
    \mathbb E\int_{E_\rho}
        \big\langle q(t),\delta b_x(t)x_1^\rho(t)\big\rangle_H\,dt
    \notag\\
    &\quad
    +
    \frac12
    \mathbb E\int_0^T
        \big\langle
            \bar a_{xx}(t)^*p(t)x_1^\rho(t),
            x_1^\rho(t)
        \big\rangle_H\,dt
    +
    \frac12
    \mathbb E\int_0^T
        \big\langle
            \bar b_{xx}(t)^*q(t)x_1^\rho(t),
            x_1^\rho(t)
        \big\rangle_H\,dt
    \notag\\
    &\quad
    +
    \frac12
    \mathbb E\int_{(0,T]}
        \big\langle
            \bar\gamma_{xx}(t)^*p(t)x_1^\rho(t-),
            x_1^\rho(t-)
        \big\rangle_H\,d\bar\xi(t).
    \label{eq:duality-x2}
\end{align}
The term \(\mathbb E\int_{E_\rho}
	\langle q(t),\delta b_x(t)x_1^\rho(t)\rangle_H\,dt\)
	in \eqref{eq:duality-x2} satisfies
\begin{align}
\left|
\mathbb E\int_{E_\rho}
\langle q(t),\delta b_x(t)x_1^\rho(t)\rangle_H\,dt
\right|
&\le C
\left(\mathbb E\int_{E_\rho}\|q(t)\|_H^2\,dt\right)^{1/2}
\left(\mathbb E\int_{E_\rho}\|x_1^\rho(t)\|_H^2\,dt\right)^{1/2}
=o(\rho).
    \label{eq:delta-bx-x1-small}
\end{align}

Substituting \eqref{eq:duality-x1} and \eqref{eq:duality-x2}
into \eqref{eq:regular-cost-expansion} and using
\eqref{eq:delta-bx-x1-small} yields
\begin{align}
    J(u^\rho,\bar\xi)-J(\bar u,\bar\xi)
    &=
    \mathbb E\int_{E_\rho}
    \Big[
        \delta k(t)
        +
        \big\langle p(t),\delta a(t)\big\rangle_H
        +
        \big\langle q(t),\delta b(t)\big\rangle_H
    \Big]\,dt
    \notag\\
    &\quad
    +
    \frac12
    \mathbb E\bigg[
        \big\langle
            h_{xx}(\bar x(T))x_1^\rho(T),
            x_1^\rho(T)
        \big\rangle_H
        +
        \int_0^T
            \big\langle
                \bar H_{xx}(t)x_1^\rho(t),
                x_1^\rho(t)
            \big\rangle_H\,dt
    \notag\\
    &\hspace{2.5cm}
        +
        \int_{(0,T]}
            \big\langle
                \bar G_{xx}(t)x_1^\rho(t-),
                x_1^\rho(t-)
            \big\rangle_H\,d\bar\xi(t)
    \bigg]
    +
    o(\rho).
    \label{eq:regular-cost-before-second-adjoint}
\end{align}

Since \(\delta b\in L^{16}_{\mathbb F}(0,T;H)\) and
\(\mathfrak E_{8\kappa,L_\gamma}(\bar\xi)<\infty\), we apply
Theorem~\ref{thm:FV-BSIE-Ito} with \(\alpha=2\) to
\eqref{eq:second-adjoint} and the first variation \(x_1^\rho\) to obtain
\begin{align}
    &\mathbb E\bigg[
        \big\langle
            h_{xx}(\bar x(T))x_1^\rho(T),
            x_1^\rho(T)
        \big\rangle_H
        +
        \int_0^T
            \big\langle
                \bar H_{xx}(t)x_1^\rho(t),
                x_1^\rho(t)
            \big\rangle_H\,dt
    \notag\\
    &\hspace{2.5cm}
        +
        \int_{(0,T]}
            \big\langle
                \bar G_{xx}(t)x_1^\rho(t-),
                x_1^\rho(t-)
            \big\rangle_H\,d\bar\xi(t)
    \bigg]
    =
    \mathbb E\int_{E_\rho}
        \big\langle P(t)\delta b(t),\delta b(t)\big\rangle_H\,dt
    +
    o(\rho).
    \label{eq:second-adjoint-duality-x1}
\end{align}
Substituting \eqref{eq:second-adjoint-duality-x1} into
\eqref{eq:regular-cost-before-second-adjoint} yields
\begin{align}
   0\leq J(u^\rho,\bar\xi)-J(\bar u,\bar\xi)
    &=
    \mathbb E\int_{E_\rho}
    \bigg[
        \delta k(t)
        +
        \big\langle p(t),\delta a(t)\big\rangle_H
        +
        \big\langle q(t),\delta b(t)\big\rangle_H
    \notag\\
    &\hspace{3.1cm}
        +
        \frac12
        \big\langle P(t)\delta b(t),\delta b(t)\big\rangle_H
    \bigg]\,dt
    +
    o(\rho).
    \label{eq:regular-spike-expansion-final}
\end{align}
From \eqref{eq:regular-spike-expansion-final}, we obtain the first-order
necessary condition \eqref{eq:SMP-regular-final} for the regular control.

\subsubsection{Singular perturbations}
\label{subsec:singular-perturbations}

We next expand the cost under a convex perturbation of the singular control.

\begin{lemma}
\label{lem:singular-cost-expansion}
Under the assumptions of Theorem~\ref{thm:SMP}, we have
\begin{align}
    \frac{
        J(\bar u,\xi^\varepsilon)-J(\bar u,\bar\xi)
    }{\varepsilon}
    &=
    \mathbb E\bigg[
        \big\langle h_x(\bar x(T)),y(T)\big\rangle_H
        +
        \int_0^T
            \big\langle \bar k_x(t),y(t)\big\rangle_H\,dt
    \notag\\
    &\quad
        +
        \int_{(0,T]}
            \big\langle \bar g_x(t),y(t-)\big\rangle_H\,d\bar\xi(t)
        +
        \int_{(0,T]}
            \bar g(t)\,d\nu(t)
    \bigg]
    +o(1),
    \label{eq:singular-cost-expansion}
\end{align}
as \(\varepsilon\downarrow0\), where \(y\) is the solution of
\eqref{eq:first-variation-singular}.
\end{lemma}
\begin{proof}
	Set
	\(z^\varepsilon(t) := \frac{x^\varepsilon(t)-\bar x(t)}{\varepsilon}.\)
	Applying Taylor's formula to \(h(x^\varepsilon(T))\),
	\(k(t,x^\varepsilon(t),\bar u(t))\), and
	\(g(t,x^\varepsilon(t-))\), and using
	\(d\xi^\varepsilon=d\bar\xi+\varepsilon d\nu\), we obtain
	\begin{equation}
\begin{aligned}
		\frac{J(\bar u,\xi^\varepsilon)-J(\bar u,\bar\xi)}
		{\varepsilon}
		&=
		\mathbb E\bigg[
		\big\langle h_x(\bar x(T)),z^\varepsilon(T)\big\rangle_H
		+
		\int_0^T
		\big\langle \bar k_x(t),z^\varepsilon(t)\big\rangle_H\,dt
		\\
		&\quad
		+
		\int_{(0,T]}
		\big\langle \bar g_x(t),z^\varepsilon(t-)\big\rangle_H
		\,d\bar\xi(t)
		+
		\int_{(0,T]}
		g(t,x^\varepsilon(t-))\,d\nu(t)
		\bigg]
		+I_\varepsilon,
	\end{aligned}
\label{eq:singular-cost-Taylor-quotient}
\end{equation}
Here \(I_\varepsilon\) is the sum of the quadratic Taylor remainders
from the terminal cost and the Lebesgue and Stieltjes integrals, divided by
\(\varepsilon\). Boundedness of the second derivatives gives
\[
 |I_\varepsilon|
 \le C\varepsilon\,
 \|1+\bar\xi(T)\|_{L^2}
 \|z^\varepsilon\|_{\mathbb D^4_{\mathbb F}(0,T;H)}^2
 =O(\varepsilon),
\]
where Lemma~\ref{lem:singular-variation-estimate} bounds the
\(\mathbb D^4_{\mathbb F}\)-norm uniformly in \(\varepsilon\).
The error from replacing \(g(t,x^\varepsilon(t-))\) by
\(\bar g(t)\) in the \(d\nu\)-integral of
\eqref{eq:singular-cost-Taylor-quotient} is estimated using
\eqref{eq:nu-total-variation-bound} and the bounded derivative \(g_x\):
\begin{align*}
 &\left|\mathbb E\int_{(0,T]}
       [g(t,x^\varepsilon(t-))-\bar g(t)]\,d\nu(t)\right|\\
 &\quad\le
 C\varepsilon\,
 \||\nu|_v(T)\|_{L^2}
 \|z^\varepsilon\|_{\mathbb D^2_{\mathbb F}(0,T;H)}
 =O(\varepsilon).
\end{align*}
The errors in the terminal, Lebesgue, and Stieltjes linear terms are
bounded by
\begin{align*}
 &\left|\mathbb E\langle h_x(\bar x(T)),z^\varepsilon(T)-y(T)\rangle_H\right|\\
 &\quad+
 \left|\mathbb E\int_0^T
       \langle\bar k_x(t),z^\varepsilon(t)-y(t)\rangle_H\,dt\right|\\
 &\quad+
 \left|\mathbb E\int_{(0,T]}
       \langle\bar g_x(t),z^\varepsilon(t-)-y(t-)\rangle_H
       \,d\bar\xi(t)\right|\\
 &\quad\le
 C(1+T+\|\bar\xi(T)\|_{L^2})
 \|z^\varepsilon-y\|_{\mathbb D^2_{\mathbb F}(0,T;H)}
 =O(\varepsilon).
\end{align*}
Substituting these \(O(\varepsilon)\) error bounds into
\eqref{eq:singular-cost-Taylor-quotient} proves
\eqref{eq:singular-cost-expansion}.
\end{proof}

Applying Lemma~\ref{lem:first-order-duality} to
\(y\) in \eqref{eq:first-variation-singular},
whose only inhomogeneous term is
\(\bar\gamma(t)\,d\nu(t)\), with
\(\nu=\eta-\bar\xi\), we obtain
\begin{align}
    &\mathbb E\bigg[
        \big\langle h_x(\bar x(T)),y(T)\big\rangle_H
        +
        \int_0^T
            \big\langle \bar k_x(t),y(t)\big\rangle_H\,dt
        +
        \int_{(0,T]}
            \big\langle \bar g_x(t),y(t-)\big\rangle_H\,d\bar\xi(t)
    \bigg]
    \notag\\
    &=
    \mathbb E\int_{(0,T]}
        \big\langle p(t),\bar\gamma(t)\big\rangle_H\,d\nu(t).
    \label{eq:singular-duality}
\end{align}
Substituting \eqref{eq:singular-duality} into
\eqref{eq:singular-cost-expansion} yields
\begin{align}
    \frac{
        J(\bar u,\xi^\varepsilon)-J(\bar u,\bar\xi)
    }{\varepsilon}
    =
    \mathbb E\int_{(0,T]}
       \mathcal G(t,\bar x(t-))
        \,d(\eta-\bar\xi)(t)
    +o(1).
    \label{eq:singular-cost-after-duality}
\end{align}
Optimality of \((\bar u,\bar\xi)\) and admissibility of
\((\bar u,\xi^\varepsilon)\) therefore imply
\begin{equation}
    \mathbb E\int_{(0,T]}
        \mathcal G(t,\bar x(t-))
        \,d(\eta-\bar\xi)(t)
    \ge0,
    \quad
    \text{for every admissible singular control }\eta.
    \label{eq:singular-variational-inequality}
\end{equation}
The optional-projection and stopping-time localization argument in
Appendix~\ref{subsec:singular-VI-pointwise-conditions} shows that
\eqref{eq:singular-variational-inequality} implies
\eqref{eq:SMP-singular-nonnegative} and
\eqref{eq:SMP-singular-complementarity}.

\begin{remark}
The results of this paper extend, with direct modifications, to
equations driven by a cylindrical Wiener process on a separable
Hilbert space \(K_W\), with Hilbert--Schmidt diffusion coefficients
and the bound
\[
    \|(B(t)x)^*x\|_{K_W}\le C_B\|x\|_H^2,
    \qquad x\in V.
\]
\end{remark}

\subsection{A stochastic heat equation with regular and singular controls}
\label{subsec:compact-mixed-control-example}
Let \(H=L^2(0,\pi)\) and \(V=H_0^1(0,\pi)\), with
\(\|z\|_V=\|z'\|_{L^2}\). The Dirichlet Laplacian
\(A:V\to V^*\) is defined by
\[
\langle Az,v\rangle_*
=-\int_0^\pi z'(\theta)v'(\theta)\,d\theta,
\quad z,v\in V.
\]
The realization of \(A\) on \(H\) acts as \(Az=z''\), with
domain \(H^2(0,\pi)\cap H_0^1(0,\pi)\). Its normalized first
eigenfunction is
\[
e(\theta)=\sqrt{\frac2\pi}\sin\theta,
\quad \|e\|_H=1,\quad Ae=-e.
\]

Fix \(r,\sigma>0\), \(0<c<1\), \(K\in\mathbb R\), and
\(0<m<2c/r\). For \(x_0\in H\), the state equation and cost are
\begin{align}
	dx(t)&=\bigl[Ax(t)+u(t)e\bigr]dt
	+\sigma e\,dW(t)+e\,d\xi(t),
	\quad x(0)=x_0,\notag\\
	J(u,\xi)&=\mathbb E\left[
	\varphi\bigl(\langle x(T),e\rangle_H\bigr)
	+\frac r2\int_0^T u(t)^2\,dt+c\xi(T)\right],
	\quad \varphi(y)=\sqrt{1+(y-K)^2}.
	\label{eq:compact-example-model}
\end{align} 
Both controls act with spatial distribution \(e\). The regular control takes values in the nonconvex
set \(U=\{-m,0,m\}\), corresponding to cooling, no regular input,
and heating. The singular control \(\xi\) is an adapted
nondecreasing c\`adl\`ag process with \(\xi(0)=0\).
The regular control \(u\) gives a bounded heating or cooling
rate, whereas the singular control \(\xi\)
supplies nonnegative heat inputs and allows instantaneous
heating: a jump \(\Delta\xi(t)\) adds
\(e\,\Delta\xi(t)\) to the state.

The scalar process \(y(t)=\langle x(t),e\rangle_H\) is the
coefficient of \(x(t)\) along \(e\), and \(y(t)e\) is the
corresponding spatial component. Projection onto \(e\) gives
\[
dy(t)=\bigl[-y(t)+u(t)\bigr]dt+\sigma\,dW(t)+d\xi(t),
\quad y(0)=\langle x_0,e\rangle_H.
\]
The terminal penalty measures the deviation of \(y(T)\) from
\(K\), and \(c\) is the cost per unit of singular input.

The coefficients satisfy Assumption~\ref{ass:control-coefficients}.
Since \(L_\gamma=0\), the conditional exponential-moment requirement
is automatic.

Let \((\bar u,\bar\xi)\) be an optimal pair and write
\(\bar y(t)=\langle\bar x(t),e\rangle_H\).
Define the scalar adjoint processes \((p_1,q_1)\) by
\begin{equation}
	dp_1(t)=p_1(t)\,dt+q_1(t)\,dW(t),
	\quad p_1(T)=\varphi'(\bar y(T)).
	\label{eq:compact-example-adjoint}
\end{equation}
Since \(h_x(\bar x(T))=\varphi'(\bar y(T))e\) and \(Ae=-e\),
uniqueness for the first-order adjoint equation gives
\(p(t)=p_1(t)e\) and \(q(t)=q_1(t)e\).
The process \(p_1\) is continuous and satisfies
\begin{equation}
	p_1(t)=e^{-(T-t)}
	\mathbb E\bigl[p_1(T)\mid\mathcal F_t\bigr]
	=e^{-(T-t)}
	\mathbb E\bigl[\varphi'(\bar y(T))\mid\mathcal F_t\bigr],
	\quad |p_1(t)|\le e^{-(T-t)}.
	\label{eq:compact-example-adjoint-representation}
\end{equation}

The diffusion does not depend on \(u\), so the second-order
correction in Theorem~\ref{thm:SMP} vanishes.
The regular Hamiltonian is \(p_1v+\sigma q_1+rv^2/2\).
Since \(\sigma q_1\) does not depend on \(v\), the regular
minimization condition becomes
\begin{equation}
	\bar u(t)\in
	\operatorname*{arg\,min}_{v\in U}
	\left\{p_1(t)v+\frac r2v^2\right\},
	\quad (dt\otimes d\mathbb P)\text{-a.e.}
	\label{eq:compact-example-regular-minimization}
\end{equation}
Evaluating the Hamiltonian at the three control values gives
\begin{equation}
	\bar u(t)=
	\begin{cases}
		m,&p_1(t)<-rm/2,\\
		0,&|p_1(t)|<rm/2,\\
		-m,&p_1(t)>rm/2.
	\end{cases}
	\label{eq:nonconvex-heat-regular-control}
\end{equation}
At \(p_1=-rm/2\), both \(0\) and \(m\) minimize the Hamiltonian;
at \(p_1=rm/2\), both \(-m\) and \(0\) minimize it.

The singular Hamiltonian \(\mathcal G(t)=p_1(t)+c\) is
continuous. The singular nonnegativity and contact conditions
in Theorem~\ref{thm:SMP} read
\begin{equation}
	p_1(t)+c\ge0,\quad
	\bigl(p_1(t)+c\bigr)d\bar\xi(t)=0.
	\label{eq:compact-example-SMP}
\end{equation}
Evaluating the nonnegativity condition at \(T\) gives
\(p_1(T)\ge-c\). By
\eqref{eq:compact-example-adjoint-representation},
\[
p_1(t)+c
=e^{-(T-t)}\mathbb E[p_1(T)\mid\mathcal F_t]+c
\ge c\bigl(1-e^{-(T-t)}\bigr)>0,
\quad 0\le t<T.
\]
The contact condition in \eqref{eq:compact-example-SMP},
together with \(\bar\xi(0)=0\), therefore gives
\(\bar\xi(t)=0\) for \(t<T\).
Any optimal singular input occurs at \(T\).

To determine \(\Delta\bar\xi(T)\), let \(Y\) solve the scalar
state equation with regular control \(\bar u\) and no singular
input:
\[
dY(t)=\bigl[-Y(t)+\bar u(t)\bigr]dt+\sigma\,dW(t),
\quad Y(0)=\langle x_0,e\rangle_H.
\]
Then \(\bar y(T)=Y(T)+\Delta\bar\xi(T)\).
At \(T\), \eqref{eq:compact-example-adjoint} and
\eqref{eq:compact-example-SMP} give
\[
\varphi'(\bar y(T))+c\ge0,\quad
\bigl(\varphi'(\bar y(T))+c\bigr)\Delta\bar\xi(T)=0.
\]
Define the terminal heating threshold \(y_c\) as the unique
solution of \(\varphi'(y_c)+c=0\).
Since \(\varphi'(y)=(y-K)/\sqrt{1+(y-K)^2}\) is strictly
increasing, the terminal nonnegativity condition requires
\(\bar y(T)\ge y_c\), while the terminal contact condition requires
\(\bar y(T)=y_c\) whenever \(\Delta\bar\xi(T)>0\).
Consequently,
\begin{equation}
	y_c=K-\frac{c}{\sqrt{1-c^2}},\quad
	\Delta\bar\xi(T)=\bigl(y_c-Y(T)\bigr)^+,\quad
	\bar y(T)=\max\{Y(T),y_c\}.
	\label{eq:compact-example-terminal-control}
\end{equation}
When \(Y(T)<y_c\), the terminal input raises the scalar state
from \(Y(T)\) to \(y_c\). When \(Y(T)\ge y_c\), the terminal
input is zero.

Conversely, let \((\bar u,\bar\xi)\) be an admissible pair,
and let \((p_1,q_1)\) solve \eqref{eq:compact-example-adjoint}
for its state \(\bar x\).
If the regular minimization condition
\eqref{eq:compact-example-regular-minimization} and the singular
nonnegativity and contact conditions
\eqref{eq:compact-example-SMP} hold, then
\((\bar u,\bar\xi)\) is optimal.
Indeed, convexity of \(\varphi\) and adjoint duality give,
for every admissible \((u,\xi)\),
\begin{equation}
	J(u,\xi)-J(\bar u,\bar\xi)
	\ge
	\mathbb E\int_0^T
	\left[p_1(u-\bar u)+\frac r2(u^2-\bar u^2)\right]dt
	+\mathbb E\int_{(0,T]}(p_1+c)\,d\xi
	\ge0.
	\label{eq:compact-example-optimality}
\end{equation}
The first integrand is nonnegative by
\eqref{eq:compact-example-regular-minimization}.
The second integral is nonnegative by
\eqref{eq:compact-example-SMP}, which also gives
\(\int_{(0,T]}(p_1+c)\,d\bar\xi=0\).

\section{Appendix}

\subsection{Well-posedness of the SEE with a bounded driver}
\label{subsec:localized-SEE-construction}

The following lemma provides the bounded-driver solutions used in
Theorem~\ref{thm:FV-SEE-weighted-wellposed}.
\begin{lemma}
\label{lem:localized-FV-SEE-construction}
\label{lem:additive-FV-SEE-wellposed}
Let \(r\in[0,T]\) and assume \((A1)\)--\((A3)\).
Let \(\chi\) be adapted, nondecreasing, and c\`adl\`ag on
\([r,T]\), with \(\chi(T)-\chi(r)\le K_0\) \(\mathbb P\text{-a.s.}\) for some constant
 \(K_0<\infty\). Suppose that
\[
 X_r\in L^2(\mathcal F_r;H),\quad
 a(\cdot,0)\in L^{1,2}_{\mathbb F}(r,T;H),\quad
 b(\cdot,0)\in L^2_{\mathbb F}(r,T;H),\quad
 \gamma(\cdot,0)\in L^{1,2}_{\mathbb F,\chi}(r,T;H).
\]
Then \eqref{eq:unbounded-xi-forward-SEE}, with \(\chi\) in place
of \(\xi\), has a unique solution in
\(\mathbb D^2_{\mathbb F}(r,T;H) \cap L^2_{\mathbb F}(r,T;V).\)
\end{lemma}
\begin{proof}
We first solve the equation with prescribed inhomogeneous coefficients
and then solve the nonlinear equation by contraction.
Set
\begin{equation}
 \rho_\beta^r(t)
 =\exp\{-\beta[(t-r)+\chi(t)-\chi(r)]\},
 \quad r\le t\le T.
 \label{eq:K-beta}
\end{equation}
For fixed \(\beta>0\), write
\begin{align*}
    \|y\|_{\rho_\beta^r}^2
    &:={}
    \mathbb E\bigg[
        \sup_{r\le t\le T}
            \rho_\beta^r(t)\|y(t)\|_H^2
        +
        \int_r^T
            \rho_\beta^r(t)\|y(t)\|_V^2\,dt
    \\
    &\hspace{18mm}
        +
        \beta\int_r^T
            \rho_\beta^r(t)\|y(t)\|_H^2\,dt
        +
        \beta\int_{(r,T]}
            \rho_\beta^r(t)\|y(t-)\|_H^2\,d\chi(t)
    \bigg].
\end{align*}
Since \(\chi(T)-\chi(r)\le K_0\),
\(e^{-\beta(T-r+K_0)}\le\rho_\beta^r(t)\le1,\ r\le t\le T.\)
The two additional energy terms are bounded by
\(\beta(T-r+K_0)\mathbb E\sup_t\|y(t)\|_H^2\).
Thus, for fixed \(\beta,K_0\), this norm is equivalent to the usual
norm on the complete space
\(\mathbb D^2_{\mathbb F}(r,T;H)\cap L^2_{\mathbb F}(r,T;V)\).

\medskip\noindent\emph{Step 1: The linear equation with prescribed inhomogeneous coefficients.}
Let
\[
    X_r\in L^2(\mathcal F_r;H),
    \quad
    f\in L^{1,2}_{\mathbb F}(r,T;H),
    \quad
    g\in L^{2,2}_{\mathbb F}(r,T;H),\quad
    \eta\in L^{1,2}_{\mathbb F,\chi}(r,T;H).
\]
Then the equation
\begin{equation}
\left\{
\begin{aligned}
    dx(t)
    &=
    [A(t)x(t)+f(t)]\,dt
    +
    [B(t)x(t)+g(t)]\,dW(t)
    +
    \eta(t)\,d\chi(t),
    \\
    x(r)&=X_r,
\end{aligned}
\right.
\label{eq:additive-FV-SEE}
\end{equation}
admits a unique solution
$
    x\in
    \mathbb D^2_{\mathbb F}(r,T;H)
    \cap
    L^2_{\mathbb F}(r,T;V).
$

First suppose that
\(f,g\in L^2_{\mathbb F}(r,T;H)\) and
\(\eta\in L^2_{\mathbb F,\chi}(r,T;H)\).
Let \(\{e_k\}_{k\ge1}\subset V\) be an orthonormal basis of \(H\)
whose linear span is dense in \(V\). Set
\(H_n:=\operatorname{span}\{e_1,\ldots,e_n\},\)
and denote by \(\Pi_n\) the orthogonal projection from \(H\) onto
\(H_n\).
We seek an \(H_n\)-valued process \(x_n\) satisfying, for
\(i=1,\ldots,n\),
\begin{align}
    \langle x_n(t),e_i\rangle_H
    &=
    \langle\Pi_nX_r,e_i\rangle_H
    +
    \int_r^t
        \langle A(s)x_n(s),e_i\rangle_*\,ds
    +
    \int_r^t
        \langle f(s),e_i\rangle_H\,ds
    \notag\\
    &\quad
    +
    \int_r^t
        \langle B(s)x_n(s)+g(s),e_i\rangle_H\,dW(s)
    +
    \int_{(r,t]}
        \langle\eta(s),e_i\rangle_H\,d\chi(s).
    \label{eq:SEE-Galerkin}
\end{align}
This finite-dimensional linear equation has a unique adapted c\`adl\`ag
solution.

With the inherited \(V\)- and \(H\)-norms, the projected
coefficients satisfy \((A1)\) uniformly in \(n\). Since the finite-variation
term is independent of the state variable, its
Lipschitz constant in the state variable is zero. Hence Proposition~\ref{prop:unbounded-xi-SEE-estimate}
may be applied with any positive auxiliary Lipschitz upper bound
$\bar\ell$.
We choose $\bar\ell=\beta/\kappa$, so that
$\kappa\bar\ell=\beta$ and the weight $\rho_{\beta,\bar\ell}^r$
coincides with $\rho_\beta^r$.
Proposition~\ref{prop:unbounded-xi-SEE-estimate}, applied with
\(\alpha=1\), \(\xi=\chi\), and the auxiliary Lipschitz upper bound
\(\bar\ell=\beta/\kappa\), together
with \(\rho_\beta^r\ge e^{-\beta(T-r+K_0)}\), gives uniform bounds in
\(\mathbb D^2_{\mathbb F}(r,T;H)\cap L^2_{\mathbb F}(r,T;V)\)
for sufficiently large \(\beta\).

We extract a weakly convergent subsequence in
\(L^2_{\mathbb F}(r,T;V)\), with limit \(x\).
Passing to the limit, using countability and density, yields the
weak identity for \eqref{eq:additive-FV-SEE},
\(dt\otimes d\mathbb P\)-almost everywhere.
Lemma~\ref{Ito-lemma} gives an adapted c\`adl\`ag
\(H\)-valued representative for which the weak identity holds
at every time in \([r,T]\), with \(x(r)=X_r\).
Applying Proposition~\ref{prop:unbounded-xi-SEE-estimate} to this
representative gives the required maximal and energy bounds.

We now extend the construction to general
\(f\in L^{1,2}_{\mathbb F}(r,T;H)\) and
\(\eta\in L^{1,2}_{\mathbb F,\chi}(r,T;H)\).
For $n\ge1$, define the radial truncations
\[
f_n(t):=
\begin{cases}
f(t), & \|f(t)\|_H\le n,\\[1mm]
\dfrac{n}{\|f(t)\|_H}f(t), & \|f(t)\|_H>n,
\end{cases}
\qquad
\eta_n(t):=
\begin{cases}
\eta(t), & \|\eta(t)\|_H\le n,\\[1mm]
\dfrac{n}{\|\eta(t)\|_H}\eta(t), & \|\eta(t)\|_H>n.
\end{cases}
\]
We keep $X_r$ and $g$ unchanged. Then \(f_n\in L^2_{\mathbb F}(r,T;H)\) and
\(\eta_n\in L^2_{\mathbb F,\chi}(r,T;H)\), and the dominated convergence theorem gives
\[
 f_n\to f\quad\hbox{in }L^{1,2}_{\mathbb F}(r,T;H),\quad
 \eta_n\to\eta\quad\hbox{in }L^{1,2}_{\mathbb F,\chi}(r,T;H).
\]
Let \(x_n\) denote the corresponding solutions. Applying Proposition~\ref{prop:nonlinear-FV-SEE-estimate}
with \(\alpha=1\) to \(x_n-x_m\) shows that
\(\{x_n\}\) is Cauchy in the weighted norm \(\|\cdot\|_{\rho_\beta^r}\).
Passing to the limit in \eqref{eq:additive-FV-SEE} for $x_n$,
viewed as an identity in \(V^*\), shows that the limit solves
\eqref{eq:additive-FV-SEE}.

Uniqueness follows by applying \eqref{eq:unbounded-xi-weighted-SEE-estimate}
with \(\alpha=1\) to the difference of two solutions.

\medskip\noindent\emph{Step 2: The nonlinear equation.}
For \(y\in\mathbb D^2_{\mathbb F}(r,T;H)\cap
L^2_{\mathbb F}(r,T;V)\), let \(\Gamma y\) be the solution of
\eqref{eq:additive-FV-SEE} with inhomogeneous coefficients
\(f(t)=a(t,y(t)),\ g(t)=b(t,y(t)),\ \eta(t)=\gamma(t,y(t-)).\)
Thus \eqref{eq:additive-FV-SEE} defines \(\Gamma\) on the complete space
equipped with \(\|\cdot\|_{\rho_\beta^r}\).

We apply
\eqref{eq:unbounded-xi-weighted-SEE-estimate} to \(\Gamma y_1-\Gamma y_2\), for \(y_1,y_2\) in this space, with \(\alpha=1\),
\(\xi=\chi\), and 
\(\bar\ell=\beta/\kappa\). The left-hand side has the same four terms as
\(\|\cdot\|_{\rho_\beta^r}^2\), except that the Stieltjes-energy
coefficient is \((\kappa-2)\beta/\kappa\) in place of \(\beta\).
The left-hand side and this squared norm are therefore equivalent
up to constants depending only on \(\kappa\). Hence, for sufficiently
large \(\beta\),
\begin{align*}
 \|\Gamma y_1-\Gamma y_2\|_{\rho_\beta^r}^2
 &\le C L^2(T-r+1)\mathbb E\int_r^T
       \rho_\beta^r(t)\|y_1(t)-y_2(t)\|_H^2\,dt\\
 &+C\ell^2K_0\mathbb E\int_{(r,T]}
       \rho_\beta^r(t)\|y_1(t-)-y_2(t-)\|_H^2\,d\chi(t)\\
 &\le
 \frac{C\{L^2(T-r+1)+\ell^2K_0\}}{\beta}
       \|y_1-y_2\|_{\rho_\beta^r}^2.
\end{align*}
Choose \(\beta\) large enough for the estimate to hold and such that
\[
 \frac{C\{L^2(T-r+1)+\ell^2K_0\}}{\beta}\le\frac12.
\]
The map \(\Gamma\) is then a contraction, and its unique fixed point
solves \eqref{eq:unbounded-xi-forward-SEE} with driver \(\chi\).
Conversely, a solution in the stated unweighted space belongs to
this weighted space and, by uniqueness for the linear equation, is a
fixed point. This proves uniqueness.
\end{proof}

\subsection{Well-posedness of the BSEE with a bounded driver}
\label{subsec:bounded-BSEE-construction}

The next lemma establishes well-posedness of the approximating
equations with bounded drivers used in
Theorem~\ref{thm:unbounded-xi-BSEE-finite-order}. 
\begin{lemma}
\label{lem:BSEE-bounded-driver-construction}
Let \(\chi\) be adapted, nondecreasing, and c\`adl\`ag, with
\(\chi(0)=0\) and \(\chi(T)\le n\) for a deterministic \(n\).
Assume \((B1)\)--\((B3)\),
\(\varphi\in L^2(\mathcal F_T;H)\),
\(g(\cdot,0,0)\in L^2_{\mathbb F}(0,T;H)\), and
\(\psi(\cdot,0)\in L^{1,2}_{\mathbb F,\chi}(0,T;H)\).
Then \eqref{eq:FV-BSEE}, with \(\xi\) replaced by
 \(\chi\), has a unique solution
$
 (p,q)\in
 \bigl(\mathbb D^2_{\mathbb F}(0,T;H)
       \cap L^2_{\mathbb F}(0,T;V)\bigr)
 \times L^2_{\mathbb F}(0,T;H).
$
\end{lemma}

\begin{proof}
\medskip\noindent\emph{Step 1: The linear equation with frozen generators.}
Fix \(P\in\mathbb D^2_{\mathbb F}(0,T;H)
\cap L^2_{\mathbb F}(0,T;V)\) and
\(Q\in L^2_{\mathbb F}(0,T;H)\).
Then
\(g(\cdot,P,Q)\in L^2_{\mathbb F}(0,T;H)
\subset L^{1,2}_{\mathbb F}(0,T;H)\), and 
\(\psi(\cdot,P)\in L^{1,2}_{\mathbb F,\chi}(0,T;H)\).

For 
\(\zeta\in L^{1,2}_{\mathbb F,\chi}(0,T;H)\), let \(\zeta^m\) be
its radial truncation at radius \(m\). Since \(\chi(T)\le n\),
\(\zeta^m\in L^2_{\mathbb F,\chi}(0,T;H)\), and dominated convergence gives
\[
 \mathbb E\left(
   \int_{(0,T]}\|\zeta^m(t)-\zeta(t)\|_H\,d\chi(t)
 \right)^2\longrightarrow0.
\]
For each \(m\), \cite[Theorem~3.4]{LSW25}, with zero anticipation
interval, \(F=\chi\), terminal datum \(\varphi\), and prescribed
coefficients \(g(\cdot,P,Q)\) and \(\zeta^m\), gives the unique solution of the frozen
linear equation. Letting \(m\to\infty\) extends solvability to \(\zeta\).
Uniqueness follows by applying
\eqref{eq:unbounded-xi-weighted-BSEE-estimate} to the difference of two solutions.
We may therefore define \(\Gamma(P,Q)=(p,q)\).

\medskip\noindent\emph{Step 2: Contraction for the nonlinear equation.}
Fix \(\kappa>2\). Write
$
 \theta(t):=e^{\beta(t+\chi(t))}
           =\theta_{\beta,\beta/\kappa}^{0}(t),
$
and use the norm
\[
 \|(p,q)\|_{\beta,\chi}^2
 =\mathbb E\left[
    \sup_{0\le t\le T}\theta(t)\|p(t)\|_H^2
    +\int_0^T\theta(t)
       \bigl(\|p(t)\|_V^2+\|q(t)\|_H^2
                         +\beta\|p(t)\|_H^2\bigr)dt\right].
\]
Since \(1\le\theta\le e^{\beta(T+n)}\), this norm is equivalent,
for fixed \(\beta,n\), to the usual norm on the complete product
space in Definition~\ref{def:FV-BSEE-solution}.

Let \((p_i,q_i)=\Gamma(P_i,Q_i)\), \(i=1,2\), and write
\(\widehat p=p_1-p_2,\ \widehat q=q_1-q_2,\)
\[
\begin{aligned}
 \widehat g(t)
 &=g(t,P_1(t),Q_1(t))-g(t,P_2(t),Q_2(t)),\\
 \widehat\psi(t)
 &=\psi(t,P_1(t))-\psi(t,P_2(t)).
\end{aligned}
\]
The pair \((\widehat p,\widehat q)\) satisfies the linear equation
\[
 \widehat p(t)
 =\int_t^T[M(s)\widehat p(s)+N(s)\widehat q(s)
                              +\widehat g(s)]ds
   +\int_{(t,T]}\widehat\psi(s)d\chi(s)
   -\int_t^T\widehat q(s)dW(s).
\]

The integrability of \(g(\cdot,P_i,Q_i)\) gives
\(\widehat g\in L^2_{\mathbb F}(0,T;H)\), and the Lipschitz
condition on \(\psi\) gives
\[
 \mathbb E\int_{(0,T]}\|\widehat\psi(t)\|_H^2\,d\chi(t)
 \le\ell^2n\,\mathbb E\sup_{0\le t\le T}
                   \|P_1(t)-P_2(t)\|_H^2<\infty.
\]
Hence \(\widehat\psi\in L^2_{\mathbb F,\chi}(0,T;H)\).
To apply \eqref{eq:unbounded-xi-weighted-BSEE-estimate}, note that \(\widehat\psi\) is independent of \(\widehat p\).
We may therefore choose \(\bar{\ell}=\beta/\kappa\), for which the weight is
\(\theta(t)=e^{\beta(t+\chi(t))}\). The estimate then gives
\begin{equation}
 \|(\widehat p,\widehat q)\|_{\beta,\chi}^2
 \le\frac C\beta\,\mathbb E\left[
       \int_0^T\theta(t)\|\widehat g(t)\|_H^2\,dt
       +\int_{(0,T]}\theta(t)\|\widehat\psi(t)\|_H^2\,d\chi(t)
                                 \right],
 \label{eq:BSEE-linear-L2-source-estimate}
\end{equation}
where \(C\) is independent of \(\beta,n\) and
\((P_i,Q_i)\), \(i=1,2\). 
Since \(\chi(T)\le n\),
\[
 \int_{(0,T]}\theta(t)\|P_1(t)-P_2(t)\|_H^2d\chi(t)
 \le n\sup_t\theta(t)\|P_1(t)-P_2(t)\|_H^2.
\]
Substituting this bound into \eqref{eq:BSEE-linear-L2-source-estimate}, we obtain
\[
 \|\Gamma(P_1,Q_1)-\Gamma(P_2,Q_2)\|_{\beta,\chi}^2
 \le\frac{C(L^2+\ell^2n)}{\beta}
             \|(P_1-P_2,Q_1-Q_2)\|_{\beta,\chi}^2.
\]
We choose \(\beta\ge\beta_1\), as required by
\eqref{eq:unbounded-xi-weighted-BSEE-estimate} with moment parameter
\(1\), and \(\beta>2C(L^2+\ell^2n)\). The solution map is then a strict contraction, and its unique fixed
point solves \eqref{eq:FV-BSEE} with driver \(\chi\). Conversely,
by uniqueness for the frozen linear equation, every solution in
the stated class is a fixed point. This proves uniqueness.

\end{proof}

\subsection{The evolution family of the homogeneous SEE}
\label{subsec:homogeneous-SEE-evolution-family}

The next result verifies \((C1)\) for the homogeneous SEE evolution family used in
Section~\ref{subsec:FV-BSIE-Ito}.
\begin{proposition}
\label{lem:SEE-joint-cadlag-flow-version}
Let the family \(L\) be defined by
\eqref{eq:SEE-deterministic-start-L-definition} under assumptions
\((A1)\), \eqref{bound-D-condition}, and
\eqref{eq:BSIE-flow-exponential-condition}.
For every \(0\le r\le s\le T\) and \(1\le m\le2\alpha\),
\(L(r,s)\) maps \(L^{2m}(\mathcal F_r;H)\) into
\(L^{2m}(\mathcal F_s;H)\). The family \(L\) satisfies \((C1)\).
\end{proposition}

\begin{proof}
We first establish the conditional moment bound, locality, and the flow
identity. We then construct the jointly measurable version using right
continuity with respect to the starting time.

Fix \(1\le m\le2\alpha\) and choose an auxiliary time-weight parameter
\(\beta\ge\max\{\beta_m,\beta_{2\alpha},\beta_{4\alpha}\}\).
For \(E\in\mathcal F_r\), multiplying the weak equation by \(I_E\)
shows that \(I_EX^{r,\eta}\) is the solution with initial value
\(I_E\eta\). Hence
\[
 L(r,s)(I_E\eta)=I_E L(r,s)\eta,\quad r\le s\le T,
\]
up to indistinguishability. The flow identity $L(t,s)L(r,t)\eta=L(r,s)\eta,\ r\le t\le s\le T,$ follows from the uniqueness of solutions.

For bounded \(\eta\), the conditional form of
\eqref{eq:unbounded-xi-unweighted-SEE-estimate} gives
\begin{align}
 &\mathbb E\left[
   \sup_{r\le s\le T}\|L(r,s)\eta\|_H^{2m}
   +\left(\int_r^T\|L(r,s)\eta\|_V^2\,ds\right)^m
   \,\middle|\,\mathcal F_r\right]
 \le C_m\|\eta\|_H^{2m},\quad\mathbb P\text{-a.s.}
 \label{eq:SEE-deterministic-start-conditional-estimate}
\end{align}
To extend this bound to \(\eta\in L^{2m}(\mathcal F_r;H)\), let
\(X^{r,\eta}\) be the corresponding weighted solution, and put
\(E_n:=\{\|\eta\|_H\le n\}\in\mathcal F_r\).
By locality, the solution with bounded initial value \(I_{E_n}\eta\)
is \(I_{E_n}X^{r,\eta}\), up to indistinguishability. Applying the
estimate \eqref{eq:SEE-deterministic-start-conditional-estimate} gives
\begin{align*}
 &\mathbb E\left[
   I_{E_n}\left(
     \sup_{r\le s\le T}\|X^{r,\eta}(s)\|_H^{2m}
     +\left(\int_r^T\|X^{r,\eta}(s)\|_V^2\,ds\right)^m
   \right)\,\middle|\,\mathcal F_r\right]
 \le C_m I_{E_n}\|\eta\|_H^{2m}.
\end{align*}
Since \(E_n\uparrow\Omega\) up to a null set, conditional monotone
convergence proves \eqref{eq:SEE-deterministic-start-conditional-estimate}
for \(\eta\). Taking expectations shows that
$
 X^{r,\eta}\in\mathbb D^{2m}_{\mathbb F}(r,T;H)
                    \cap L^{2,2m}_{\mathbb F}(r,T;V).
$
Every solution in this unweighted class belongs to the weighted class,
so uniqueness follows from
Theorem~\ref{thm:FV-SEE-weighted-wellposed}.

For the jointly measurable version, fix \(u\in H\).
To compare solutions with different starting times on the same interval,
extend \(X^{r,u}\) to \([0,T]\), for each deterministic \(r\), by setting
\[
 U^r(s)=u\quad(0\le s<r),\quad
 U^r(s)=L(r,s)u\quad(r\le s\le T).
\]
For each \(r\), choose a c\`adl\`ag version. Since \(L(r,r)u=u\),
the extension is also c\`adl\`ag at \(r\).
For deterministic \(r'\ge r\), the flow identity and
\eqref{eq:SEE-deterministic-start-conditional-estimate} at \(m=2\)
give
\begin{align}
 \mathbb E\sup_{0\le s\le T}\|U^{r'}(s)-U^r(s)\|_H^4
 &\le (1+C_2)\mathbb E\sup_{r\le s\le r'}
                         \|L(r,s)u-u\|_H^4
 \longrightarrow0,\quad r'\downarrow r.
 \label{eq:SEE-start-time-right-continuity}
\end{align}
Indeed, on \([r',T]\) the difference equals
\(L(r',s)(u-L(r,r')u)\). On \([r,r']\), it equals
\(u-L(r,s)u\). The convergence in
\eqref{eq:SEE-start-time-right-continuity} follows from right continuity
at \(r\) and dominated convergence, using the integrable bound
\(8(\sup_{r\le s\le T}\|L(r,s)u\|_H^4+\|u\|_H^4)\).

We next approximate the starting time from above by dyadic times.
Let \(\pi_n(r)=T2^{-n}\lceil2^nr/T\rceil\), so that
\(\pi_n(r)\downarrow r\), and put
\(U_n(\omega,r,s)=U^{\pi_n(r)}(s,\omega)\).
For each \(n\), the starting time \(\pi_n(r)\) takes finitely many
deterministic values. At the countably many dyadic times, choose versions
with c\`adl\`ag paths everywhere by replacing them by the constant
path \(u\) on one null event. Thus \(U_n\) is jointly measurable in
\((\omega,r,s)\), and \(s\mapsto U_n(\omega,r,s)\) is c\`adl\`ag
for every \((\omega,r)\).

We now select a convergent subsequence measurably in \(r\).
For each \(N\ge1\), the deterministic function
\[
 b_N(r):=\sup_{n,k\ge N}\mathbb E
       \sup_{0\le s\le T}\|U_n(r,s)-U_k(r,s)\|_H^4
\]
is Borel measurable, because each pair of dyadic starting times takes
only finitely many values. Equation
\eqref{eq:SEE-start-time-right-continuity} implies
\(b_N(r)\downarrow0\) for every \(r\).
Choose increasing Borel integer-valued functions \(N_j(r)\) recursively
by taking the smallest \(N>N_{j-1}(r)\) such that
\(b_N(r)\le2^{-8j}\), with \(N_0(r)=0\).
For each fixed \(r\), Markov's inequality gives
\[
 \mathbb P\left(
   \sup_{0\le s\le T}
   \|U_{N_{j+1}(r)}(r,s)-U_{N_j(r)}(r,s)\|_H>2^{-j}
             \right)
 \le 2^{-4j}.
\]
By Borel--Cantelli, \(U_{N_j(r)}(r,\cdot)\) converges uniformly in
\(s\), almost surely for each fixed \(r\).

To define the joint version, consider the set of
\((\omega,r)\) on which \(U_{N_j(r)}(r,\cdot)\) is uniformly Cauchy.
This set is measurable because the c\`adl\`ag property allows all
suprema to be taken over
\((\mathbb Q\cap[0,T])\cup\{T\}\).
We take the uniform limit on this set and use the constant path \(u\)
on its complement. The resulting field is jointly measurable, with
a c\`adl\`ag path for every \((\omega,r)\).
For each deterministic \(r\),
\eqref{eq:SEE-start-time-right-continuity} identifies this limit with
\(U^r\), up to indistinguishability. Restricted to \(s\ge r\), the field
is therefore an adapted version of \(L(r,s)u\). Completeness of the
filtration preserves adaptedness after modification on the exceptional
event for that fixed \(r\).
This proves the last assertion of \((C1)\). Finally,
\eqref{eq:SEE-deterministic-start-conditional-estimate} at
\(m=2\alpha\) gives \eqref{eq:BSIE-high-order-maximal-estimate}.
\end{proof}

\subsection{Proof of Lemma~\ref{lem:FV-BSIE-weak-right-continuity}}
\label{subsec:proof-FV-BSIE-weak-right-continuity}

\begin{proof}
Proposition~\ref{cor:FV-BSIE-conditional-estimate} and
\eqref{eq:BSIE-Ito-high-moment-data} give
\begin{align}
 \sup_{0\le t\le T}\mathbb E\|P(t)\|_{\mathfrak L(H)}^{2\alpha}
 \le C_\alpha\mathbb E\Bigg[
     \|\Xi\|_{\mathfrak L(H)}^{2\alpha}
     +\left(\int_0^T\|F(s,0)\|_{\mathfrak L(H)}^2\,ds\right)^\alpha
     +\left(\int_{(0,T]}\|G(s)\|_{\mathfrak L(H)}\,d\xi(s)\right)^{2\alpha}
   \Bigg]<\infty.\label{eq:BSIE-P-high-moment-from-data}
\end{align}
Then
\begin{equation}
 F(\cdot,P(\cdot))
 \in L^{2,2\alpha}_{\mathbb F,w}(0,T;\mathfrak L(H)).
 \label{eq:BSIE-Ito-FP-high-moment}
\end{equation}

For \(t\in[r,T]\), write
\(\Phi_P^{u,v}(t):=\Phi_P(t)(u,v)\), using the extension to random initial
vectors established in \eqref{eq:BSIE-Ito-random-duality}.
Estimate \eqref{eq:BSIE-high-order-maximal-estimate} and
\eqref{eq:BSIE-P-high-moment-from-data}--\eqref{eq:BSIE-Ito-FP-high-moment}
	show that
	$
	\Phi_P^{u,v}(t)\in L^\alpha(\Omega).$
	
	Set
	\(\delta_hu:=u-L(r,r+h)u, \ \delta_hv:=v-L(r,r+h)v.\)
	Since \(s\mapsto L(r,s)u\) is c\`adl\`ag and \(L(r,r)u=u\),
	\(\delta_hu\to0\) \(\mathbb P\text{-a.s.}\) as \(h\downarrow0\).
	By \eqref{eq:BSIE-high-order-maximal-estimate} and the dominated convergence theorem, 
	\begin{equation}
		\|\delta_hu\|_{L^{4\alpha}(\mathcal F_{r+h};H)}
		+
		\|\delta_hv\|_{L^{4\alpha}(\mathcal F_{r+h};H)}
		\longrightarrow0.
		\label{eq:BSIE-delta-h-vanishing}
	\end{equation}

	For fixed \(h>0\), we apply the flow property of $L$ in $(C1)$ at times in
	\((\mathbb Q\cap(r+h,T))\cup\{r+h,T\}\). Extending these identities by
	right continuity,
	we obtain, outside one null set,
	\begin{align}
		L(r+h,s)u-L(r,s)u
		&=
		L(r+h,s)\delta_hu,
		\quad s\in[r+h,T],
		\label{eq:BSIE-flow-increment-u}
		\\
		L(r+h,s-)u-L(r,s-)u
		&=
		L(r+h,s-)\delta_hu,
		\quad s\in(r+h,T].
		\label{eq:BSIE-flow-left-increment-u}
	\end{align}
	Replacing \(u\) and \(\delta_hu\) by \(v\) and \(\delta_hv\),
	respectively, in \eqref{eq:BSIE-flow-increment-u} and
	\eqref{eq:BSIE-flow-left-increment-u} gives the same
	identities for \(v\).
	
	We now prove that
	\begin{equation}
		\Phi_P^{u,v}(r+h)
		\longrightarrow
		\Phi_P^{u,v}(r)
		\quad\text{in }L^\alpha(\Omega).
		\label{eq:BSIE-Phi-direct-right-continuity}
	\end{equation}
By \eqref{eq:BSIE-Phi-R}, the difference
\(\Phi_P^{u,v}(r+h)-\Phi_P^{u,v}(r)\) consists of the terminal
difference, the two integrals of the integrand differences over \([r+h,T]\)
and \((r+h,T]\), minus the original integrals over
\([r,r+h]\) and \((r,r+h]\), respectively. We estimate these five terms in order.

Set
\[
 d_h:=\|\delta_hu\|_{L^{4\alpha}(\Omega;H)}
       \|v\|_{L^{4\alpha}(\Omega;H)}
       +\|u\|_{L^{4\alpha}(\Omega;H)}
        \|\delta_hv\|_{L^{4\alpha}(\Omega;H)}.
\]
H\"older's inequality and
\eqref{eq:BSIE-high-order-maximal-estimate} give
\begin{align}
 &\Big\|
    \langle\Xi L(r+h,T)u,L(r+h,T)v\rangle_H
    -\langle\Xi L(r,T)u,L(r,T)v\rangle_H
  \Big\|_{L^\alpha}
 \le C \|\Xi\|_{L_w^{2\alpha}(\mathcal F_T;\mathfrak L(H))} d_h,
 \label{eq:BSIE-terminal-direct-increment}\\
 &\Bigg\|
   \int_{r+h}^T\Big[
     \langle F(s,P(s))L(r+h,s)u,L(r+h,s)v\rangle_H
     -\langle F(s,P(s))L(r,s)u,L(r,s)v\rangle_H
   \Big]\,ds
  \Bigg\|_{L^\alpha}
 \notag\\
 &\quad\le
 C\|F(\cdot,P(\cdot))\|_{L^{2,2\alpha}_{\mathbb F,w}(0,T;\mathfrak L(H))}
 d_h,
 \label{eq:BSIE-dt-long-direct-increment}\\
 &\Bigg\|
   \int_{(r+h,T]}\Big[
     \langle G(s)L(r+h,s-)u,L(r+h,s-)v\rangle_H
     -\langle G(s)L(r,s-)u,L(r,s-)v\rangle_H
   \Big]\,d\xi(s)
  \Bigg\|_{L^\alpha}
 \notag\\
 &\quad\le
 C\|G\|_{L^{1,2\alpha}_{\mathbb F,\xi,w}(0,T;\mathfrak L(H))}
 d_h.
 \label{eq:BSIE-dxi-long-direct-increment}
\end{align}
Here \(C\) is independent of \(h,u,v\).

Moreover, \eqref{eq:BSIE-high-order-maximal-estimate} gives
	\begin{align}
		&\left\|
		\int_r^{r+h}
		\big\langle
		F(s,P(s))L(r,s)u,
		L(r,s)v
		\big\rangle_H\,ds
		\right\|_{L^\alpha}
		\notag\\
		&\quad\le
		C
		\|u\|_{L^{4\alpha}(\Omega;H)}
		\|v\|_{L^{4\alpha}(\Omega;H)}
		\left\|
		\int_r^{r+h}
		\|F(s,P(s))\|_{\mathfrak L(H)}\,ds
		\right\|_{L^{2\alpha}},
		\label{eq:BSIE-dt-short-direct-increment}
	\end{align}
where \(C\) is independent of \(h,u,v\). By the Cauchy--Schwarz inequality 
and \eqref{eq:BSIE-Ito-FP-high-moment},
\[
 \left\|\int_r^{r+h}\|F(s,P(s))\|_{\mathfrak L(H)}\,ds\right\|_{L^{2\alpha}}
 \le h^{1/2}
 \|F(\cdot,P(\cdot))\|_{L^{2,2\alpha}_{\mathbb F,w}(0,T;\mathfrak L(H))}
 \longrightarrow0.
\]
For the finite-variation integral,
	\begin{align}
		&\left\|
		\int_{(r,r+h]}
		\big\langle
		G(s)L(r,s-)u,
		L(r,s-)v
		\big\rangle_H\,d\xi(s)
		\right\|_{L^\alpha}
		\notag\\
		&\quad\le
		C
		\|u\|_{L^{4\alpha}(\Omega;H)}
		\|v\|_{L^{4\alpha}(\Omega;H)}
		\left\|
		\int_{(r,r+h]}
		\|G(s)\|_{\mathfrak L(H)}\,d\xi(s)
		\right\|_{L^{2\alpha}}.
		\label{eq:BSIE-dxi-short-direct-increment}
	\end{align}
Continuity from above for the finite measure
\(\|G(s,\omega)\|_{\mathfrak L(H)}d\xi(s,\omega)\) and the dominated convergence theorem give
\[
 \int_{(r,r+h]}\|G(s)\|_{\mathfrak L(H)}\,d\xi(s)
 \longrightarrow0
 \quad\text{in }L^{2\alpha}(\Omega).
\]

Combining \eqref{eq:BSIE-delta-h-vanishing} with
    \eqref{eq:BSIE-terminal-direct-increment},
    \eqref{eq:BSIE-dt-long-direct-increment},
    \eqref{eq:BSIE-dxi-long-direct-increment},
    \eqref{eq:BSIE-dt-short-direct-increment}, and
    \eqref{eq:BSIE-dxi-short-direct-increment}, we derive
    \eqref{eq:BSIE-Phi-direct-right-continuity}.
	
Finally, \eqref{eq:BSIE-Ito-random-duality} gives
	\begin{align*}
		\left\|
		\langle P(r+h)u,v\rangle_H
		-
		\langle P(r)u,v\rangle_H
		\right\|_{L^\alpha}
		&\le
		\left\|
		\Phi_P^{u,v}(r+h)
		-
		\Phi_P^{u,v}(r)
		\right\|_{L^\alpha}
		\\
		&\quad+
		\left\|
		\mathbb E\left[
		\Phi_P^{u,v}(r)
		\,\middle|\,
		\mathcal F_{r+h}
		\right]
		-
		\mathbb E\left[
		\Phi_P^{u,v}(r)
		\,\middle|\,
		\mathcal F_r
		\right]
		\right\|_{L^\alpha}.
	\end{align*}
The first term tends to zero by
\eqref{eq:BSIE-Phi-direct-right-continuity}. For the second term,
right-continuity of \(\mathbb F\) and the reverse martingale
convergence theorem give
\[
 \mathbb E[\Phi_P^{u,v}(r)\mid\mathcal F_{r+h}]
 \longrightarrow
 \mathbb E[\Phi_P^{u,v}(r)\mid\mathcal F_r]
 \quad\text{in }L^\alpha(\Omega).
\]
	This proves \eqref{eq:FV-BSIE-weak-right-continuity}.
\end{proof}

\subsection{Moving-spike estimate and proof of the operator-valued It\^o-type formula}
\label{subsec:BSIE-Ito-proofs}

Let \(t_0\in[0,T)\), let \(0<\rho<T-t_0\), and put
\(E_\rho:=[t_0,t_0+\rho), \quad t_1:=t_0+\rho.\)

Let \(A,B\) satisfy \((A1)\), let \(D\) be weakly progressively
measurable with \(\|D(t)\|_{\mathfrak L(H)}\le C_D\), and let
\(\zeta_0\in L^\infty(\mathcal F_{t_0};V)\).
    Let \(x^\rho\) be the solution of
    \begin{equation}
        \left\{
        \begin{aligned}
            dx^\rho(t)
            &=
            A(t)x^\rho(t)\,dt
            +
            \big[
            B(t)x^\rho(t)
            +
            \zeta_0 I_{E_\rho}(t)
            \big]\,dW(t)
            \\
            &\quad+
            D(t)x^\rho(t-)\,d\xi(t),
            \\
            x^\rho(0)&=0.
        \end{aligned}
        \right.
        \label{eq:BSIE-Ito-xrho}
    \end{equation}
    To compare \(x^\rho\) with a Brownian increment, define
    \[
    \eta^\rho(t)
    :=
    \frac{1}{\sqrt\rho}\,
    \zeta_0
    \int_{t_0}^t I_{E_\rho}(s)\,dW(s),
    \quad
    t\in[t_0,t_1].
    \]
    Define \(z^\rho\) as follows:
    $
    z^\rho(t)=0$ for
    $ t<t_0,
    $ $
    z^\rho(t)=\eta^\rho(t)$ for $
     t_0\le t<t_1,
    $
    while, on \([t_1,T]\), \(z^\rho\) is the solution of
    \begin{equation}
        \left\{
        \begin{aligned}
            dz^\rho(t)
            &=
            A(t)z^\rho(t)\,dt
            +
            B(t)z^\rho(t)\,dW(t)
            +
            D(t)z^\rho(t-)\,d\xi(t),
            \\
            z^\rho(t_1)&=\eta^\rho(t_1).
        \end{aligned}
        \right.
        \label{eq:BSIE-Ito-zrho}
    \end{equation}

The following estimate compares the state driven by the spike with a
Brownian increment continued by the homogeneous flow. It will be used to prove
the quadratic BSIE formula.
\begin{proposition}
    \label{prop:moving-Brownian-spike-FV}
    Let \(\alpha\ge1\), choose \(\kappa>2\) and
    \(\beta\ge\max\{\beta_\alpha,\beta_{2\alpha}\}\) in the forward
    weight, and suppose that
    \(\mathfrak E_{2\alpha\kappa,C_D}(\xi)<\infty\).
    The processes \(x^\rho,z^\rho\) defined in
    \eqref{eq:BSIE-Ito-xrho} and \eqref{eq:BSIE-Ito-zrho} satisfy
    \begin{equation}
        \mathbb E
        \sup_{0\le t\le T}
        \left\|
        x^\rho(t)-\sqrt\rho\,z^\rho(t)
        \right\|_H^{2\alpha}
        =
        o(\rho^\alpha),
        \quad
        \rho\downarrow0.
        \label{eq:moving-spike-estimate}
    \end{equation}
    Moreover,
    \begin{align}
        &\mathbb E
        \sup_{0\le t\le T}
        \|x^\rho(t)\|_H^{2\alpha}
        +
        \mathbb E
        \sup_{0\le t\le T}
        \|\sqrt\rho\,z^\rho(t)\|_H^{2\alpha}
        \le
        C_\alpha\rho^\alpha
        \left(\mathbb E\|\zeta_0\|_H^{4\alpha}\right)^{1/2},
        \label{eq:xrho-zrho-size-estimate}
    \end{align}
    where \(C_\alpha\) is independent of \(\rho\).
\end{proposition}
\begin{proof}
	On \([0,t_0)\), both processes vanish. On \([t_0,t_1]\), set
	$
	e^\rho(t)
	:=
	x^\rho(t)-\sqrt\rho\,\eta^\rho(t).
	$
	Subtracting
    \(d(\sqrt\rho\,\eta^\rho(t))=\zeta_0\,dW(t)\)
    from the equation for \(x^\rho\) on \([t_0,t_1]\) gives
	\begin{equation}
		\left\{
		\begin{aligned}
			de^\rho(t)
			&=
			\big[
			A(t)e^\rho(t)
			+
			\sqrt\rho\,A(t)\eta^\rho(t)
			\big]\,dt
			\\
			&\quad+
			\big[
			B(t)e^\rho(t)
			+
			\sqrt\rho\,B(t)\eta^\rho(t)
			\big]\,dW(t)
			\\
			&\quad+
			\big[
			D(t)e^\rho(t-)
			+
			\sqrt\rho\,D(t)\eta^\rho(t-)
			\big]\,d\xi(t),
			\\
			e^\rho(t_0)&=0.
		\end{aligned}
		\right.
		\label{eq:moving-spike-error-equation}
	\end{equation}
	
	By Brownian scaling and the independence of
	\(W(t_0+\cdot)-W(t_0)\) from \(\mathcal F_{t_0}\),
	\begin{equation}
		\sup_{0<\rho<T-t_0}
		\mathbb E
		\sup_{t_0\le t\le t_1}
		\|\eta^\rho(t)\|_V^{2\alpha}
		\le
		C_\alpha
		\mathbb E\|\zeta_0\|_V^{2\alpha}.
		\label{eq:eta-rho-moment}
	\end{equation}

We apply \eqref{eq:FV-SEE-high-order-Vstar-estimate} to
\eqref{eq:moving-spike-error-equation} on \([t_0,t_1]\), with moment
parameter \(\alpha\), Lipschitz bound \(C_D\), and zero-point data
\(\sqrt\rho\,A\eta^\rho\), \(\sqrt\rho\,B\eta^\rho\), and
\(\sqrt\rho\,D\eta^\rho_-\), respectively.
The moment conditions on these data follow from the boundedness of
\(\zeta_0\) in \(V\), Brownian scaling for \(\eta^\rho\), and, when
\(C_D>0\), the condition \(\mathfrak E_{2\alpha\kappa,C_D}(\xi)<\infty\),
which gives every polynomial moment of \(C_D\xi(T)\).
For \(C_D=0\), the Stieltjes term vanishes.
Dividing the resulting bound by \(\rho^\alpha\) gives
    \begin{align}
\rho^{-\alpha}\mathbb E\sup_{t_0\le t\le t_1}
       \|e^\rho(t)\|_H^{2\alpha}
    &\le C_\alpha\left[\mathbb E
     \left(\int_{t_0}^{t_1}
     [\|A(s)\eta^\rho(s)\|_{V^*}^2+
       \|B(s)\eta^\rho(s)\|_H^2]ds\right)^{2\alpha}\right]^{1/2}
     \notag\\
    &\quad+C_\alpha\left[\mathbb E
     \left(\int_{(t_0,t_1]}
       \|D(s)\eta^\rho(s-)\|_Hd\xi(s)\right)^{4\alpha}\right]^{1/2}.
\label{eq:moving-spike-normalized-error-bound}
\end{align}
    By \eqref{eq:eta-rho-moment} at order \(2\alpha\) and the bounds
    on \(A\) and \(B\), the first term on the right-hand side of
    \eqref{eq:moving-spike-normalized-error-bound} is \(O(\rho^\alpha)\).
    If \(C_D>0\), the exponential condition at time zero implies
    \(C_D\xi(T)\in L^{8\alpha}\). Right-continuity and dominated
    convergence yield
    \(\mathbb E[C_D(\xi(t_1)-\xi(t_0))]^{8\alpha}\to0\), whence
    \[
 \mathbb E\left(\int_{(t_0,t_1]} \|D(s)\eta^\rho(s-)\|_Hd\xi(s)\right)^{4\alpha} \le \left(\mathbb E\sup_{t_0\le s\le t_1} \|\eta^\rho(s)\|_H^{8\alpha}\right)^{1/2} \left(\mathbb E[C_D(\xi(t_1)-\xi(t_0))]^{8\alpha}\right)^{1/2} \longrightarrow0.
\]
    When \(C_D=0\),
    \(\int_{(t_0,t_1]}\|D(s)\eta^\rho(s-)\|_Hd\xi(s)=0\), so no
    moment of the finite-variation driver is needed. In either case,
    \(\mathbb E\sup_{[t_0,t_1]}\|e^\rho\|_H^{2\alpha}
      =o(\rho^\alpha)\).
    On \([t_1,T]\), the difference \(x^\rho-\sqrt\rho\,z^\rho\)
    solves the homogeneous equation with initial value \(e^\rho(t_1)\).
    For the homogeneous equation
    \eqref{eq:BSIE-deterministic-start-homogeneous-SEE}, the conditional
    form of \eqref{eq:unbounded-xi-unweighted-SEE-estimate}, with moment
    parameter \(\alpha\), Lipschitz bound \(C_D\), and zero
    inhomogeneous data, gives
    \begin{equation}
        \mathbb E\left[
          \sup_{r\le s\le T}\|X^{r,\eta}(s)\|_H^{2\alpha}
          \,\middle|\,\mathcal F_r\right]
        \le C_\alpha\|\eta\|_H^{2\alpha},
        \quad\mathbb P\text{-a.s.},
        \label{eq:moving-spike-homogeneous-estimate}
    \end{equation}
    for \(\eta\in L^\infty(\mathcal F_r;H)\), with \(C_\alpha\)
    independent of \(r\).
    For \(\eta\in L^{2\alpha}(\mathcal F_r;H)\), let
    \(X^{r,\eta}\) be the weighted solution given by
    Theorem~\ref{thm:FV-SEE-weighted-wellposed}.
    By linearity and uniqueness, the solution with initial value
    \(I_{\{\|\eta\|_H\le n\}}\eta\) is
    \(I_{\{\|\eta\|_H\le n\}}X^{r,\eta}\).
    Applying \eqref{eq:moving-spike-homogeneous-estimate} to these
    truncations and using conditional monotone convergence gives
    the same bound for \(X^{r,\eta}\).
    Taking \(r=t_1\) and \(\eta=e^\rho(t_1)\), and then taking
    expectations, bounds the maximal \(2\alpha\)-moment on
    \([t_1,T]\) by
    \(C_\alpha\mathbb E\|e^\rho(t_1)\|_H^{2\alpha}=o(\rho^\alpha)\).
    Together with the estimate on \([t_0,t_1]\), this proves
    \eqref{eq:moving-spike-estimate}.

    Applying \eqref{eq:unbounded-xi-unweighted-SEE-estimate}
    to \(x^\rho\), with zero drift and Stieltjes data and
    diffusion datum \(\zeta_0I_{E_\rho}\), gives
    \[
    \mathbb E\sup_{0\le t\le T}\|x^\rho(t)\|_H^{2\alpha}
    \le C_\alpha
    \left[\mathbb E\left(\int_{E_\rho}\|\zeta_0\|_H^2ds
                              \right)^{2\alpha}\right]^{1/2}
    =C_\alpha\rho^\alpha
                 (\mathbb E\|\zeta_0\|_H^{4\alpha})^{1/2}.
    \]
    Brownian scaling bounds the maximal \(2\alpha\)-moment of
    \(\sqrt\rho\,z^\rho=\sqrt\rho\,\eta^\rho\) on \([t_0,t_1]\).
    Applying \eqref{eq:moving-spike-homogeneous-estimate}
    with \(r=t_1\) and initial value
    \(\sqrt\rho\,\eta^\rho(t_1)\) bounds the same maximal moment
    on \([t_1,T]\).
    Adding the bounds on the two intervals gives an upper bound of
    \(C_\alpha\rho^\alpha\mathbb E\|\zeta_0\|_H^{2\alpha}\),
    which proves \eqref{eq:xrho-zrho-size-estimate}.
\end{proof}

\begin{proof}[Proof of Theorem~\ref{thm:FV-BSIE-Ito}]
We first estimate the remainder \(\sigma^\rho\) for constant spike
amplitudes and extend the estimate to general amplitudes by approximation.
We then construct \(Z^\rho\) by martingale representation and choose a
process version of the remainder.

	Fix a progressively measurable representative of \(\zeta\). By the
	\(L^{8\alpha}(\Omega;H)\)-valued Lebesgue differentiation theorem, for
	Lebesgue-a.e. \(t_0\in[0,T)\),
	\begin{equation}
		\frac1\rho
		\int_{t_0}^{t_0+\rho}
		\mathbb E
		\|\zeta(s)-\zeta(t_0)\|_H^{8\alpha}\,ds
		\longrightarrow0.
		\label{eq:BSIE-Ito-zeta-Lebesgue-point}
	\end{equation}
	We fix such a \(t_0\). In particular,
	\[
	\limsup_{\rho\downarrow0}
	\frac1\rho
	\int_{t_0}^{t_0+\rho}
	\mathbb E\|\zeta(s)\|_H^{8\alpha}\,ds
	<\infty,
	\]
	and hence, by H\"older's inequality and
	\eqref{eq:BSIE-P-high-moment-from-data},
	\begin{equation}
		\limsup_{\rho\downarrow0}
		\frac1\rho
		\int_{t_0}^{t_0+\rho}
		\mathbb E\left[
		\|P(s)\|_{\mathfrak L(H)}^\alpha
		\|\zeta(s)\|_H^{2\alpha}
		\right]ds
		<\infty.
		\label{eq:BSIE-Ito-P-zeta-local-bound}
	\end{equation}
	
    Corollary~\ref{thm:FV-SEE-wellposed}, with moment parameter \(2\alpha\),
    gives a solution \(x^\rho\) in
    \(\mathbb D^{4\alpha}_{\mathbb F}(0,T;H)
    \cap L^{2,4\alpha}_{\mathbb F}(0,T;V)\).
    The exponential assumption of the corollary holds by
    \eqref{eq:BSIE-flow-exponential-condition}, and the diffusion
    datum has the required integrability because \(\zeta\in L^{8\alpha}_{\mathbb F}(0,T;H)\).
    Estimate \eqref{eq:unbounded-xi-unweighted-SEE-estimate}
    with moment parameter \(2\alpha\) gives
    \begin{equation}
    \left\|\sup_{0\le t\le T}\|x^\rho(t)\|_H\right\|_{L^{4\alpha}}
    \le C\left[\mathbb E\left(
      \int_{E_\rho}\|\zeta(s)\|_H^2ds\right)^{4\alpha}\right]^{1/(8\alpha)}
    \le C\sqrt\rho\left[
      \frac1\rho\int_{E_\rho}\mathbb E\|\zeta(s)\|_H^{8\alpha}ds
      \right]^{1/(8\alpha)}.
    \label{eq:BSIE-Ito-general-state-bound}
    \end{equation}
    By H\"older's inequality, \eqref{eq:BSIE-Ito-general-state-bound}, the data
    conditions, \eqref{eq:BSIE-Ito-FP-high-moment}, and
    \eqref{eq:BSIE-Ito-P-zeta-local-bound},
    the terminal, Lebesgue, Stieltjes, and current-time quadratic
    terms in \eqref{eq:sigma-rho-definition} all belong to
    \(L^\alpha(\Omega)\), including when \(\alpha=1\).

	We therefore define the remainder at each deterministic \(t\in[0,T]\) by
	\begin{align}
		\sigma^\rho(t)
		&:=
		\mathbb E\Bigg[
		\big\langle
		\Xi x^\rho(T),x^\rho(T)
		\big\rangle_H
		\notag\\
		&\quad+
		\int_t^T
		\Big[
		\big\langle
		F(s,P(s))x^\rho(s),x^\rho(s)
		\big\rangle_H
		-
		\big\langle
		P(s)\zeta(s),\zeta(s)
		\big\rangle_H I_{E_\rho}(s)
		\Big]\,ds
		\notag\\
		&\quad+
		\int_{(t,T]}
		\big\langle
		G(s)x^\rho(s-),x^\rho(s-)
		\big\rangle_H\,d\xi(s)
		\,\Bigg|\,\mathcal F_t
		\Bigg]
		-
		\big\langle
		P(t)x^\rho(t),x^\rho(t)
		\big\rangle_H.
		\label{eq:sigma-rho-definition}
	\end{align}

	We first prove \eqref{eq:FV-BSIE-Ito-sigma} when the spike amplitude is
	constant on \(E_\rho\):
	\(\zeta(s)=\zeta_0, \quad s\in E_\rho, \quad \zeta_0\in L^\infty(\mathcal F_{t_0};V).\)
	Let \(z^\rho\) be the process introduced in
	Proposition~\ref{prop:moving-Brownian-spike-FV}. Applying that
	proposition with \(2\alpha\) in place of \(\alpha\), we have
	\begin{align}
		\mathbb E\sup_{0\le t\le T}
		\left\|
		x^\rho(t)-\sqrt\rho\,z^\rho(t)
		\right\|_H^{4\alpha}
		&=
		o(\rho^{2\alpha}),
		\label{eq:moving-spike-estimate-high}
		\\
		\mathbb E\sup_{0\le t\le T}
		\|x^\rho(t)\|_H^{4\alpha}
		+
		\mathbb E\sup_{0\le t\le T}
		\|\sqrt\rho\,z^\rho(t)\|_H^{4\alpha}
		&=
		O(\rho^{2\alpha}).
		\label{eq:xrho-zrho-size-high}
	\end{align}
	
	Since \(z^\rho\) solves the homogeneous equation on \([t_1,T]\),
	\eqref{eq:BSIE-Ito-random-duality}, applied at \(t_1=t_0+\rho\), gives
	\begin{align}
&\big\langle
		P(t_1)\sqrt\rho\,z^\rho(t_1),
		\sqrt\rho\,z^\rho(t_1)
		\big\rangle_H
		\notag\\
		&=
		\mathbb E\Bigg[
		\big\langle
		\Xi\sqrt\rho\,z^\rho(T),
		\sqrt\rho\,z^\rho(T)
		\big\rangle_H
		\notag\\
		&\quad+
		\int_{t_1}^T
		\big\langle
		F(s,P(s))\sqrt\rho\,z^\rho(s),
		\sqrt\rho\,z^\rho(s)
		\big\rangle_H\,ds
		\notag\\
		&\quad+
		\int_{(t_1,T]}
		\big\langle
		G(s)\sqrt\rho\,z^\rho(s-),
		\sqrt\rho\,z^\rho(s-)
		\big\rangle_H\,d\xi(s)
		\,\Bigg|\,\mathcal F_{t_1}
		\Bigg].
\label{eq:constant-spike-flow-duality}
\end{align}
	On \([t_0,t_1]\),
	\[
	z^\rho(t)
	=
	\zeta_0
	\frac{W(t)-W(t_0)}{\sqrt\rho}.
	\]
The conditional second moments of Brownian increments give
\begin{align}
 &\mathbb E\left[
   (W(t_1)-W(t_0))^2
   \langle P(t_0)\zeta_0,\zeta_0\rangle_H
   \,\middle|\,\mathcal F_t\right]
 \notag\\
 &\quad=
 \rho\,\mathbb E\left[
   \langle P(t_0)\zeta_0,\zeta_0\rangle_H
   \,\middle|\,\mathcal F_t\right],
 \quad t\le t_0,
 \label{eq:constant-spike-Brownian-before}\\
 &\Big(
   \mathbb E[(W(t_1)-W(t_0))^2\mid\mathcal F_t]
   -(W(t)-W(t_0))^2
  \Big)\langle P(t)\zeta_0,\zeta_0\rangle_H
 \notag\\
 &\quad=(t_1-t)\langle P(t)\zeta_0,\zeta_0\rangle_H,
 \quad t_0\le t\le t_1.
 \label{eq:constant-spike-Brownian-during}
\end{align}

	We use \eqref{eq:constant-spike-Brownian-before} for \(t\le t_0\)
    and \eqref{eq:constant-spike-Brownian-during} for
    \(t\in[t_0,t_1]\) in
    \eqref{eq:constant-spike-flow-duality}. Comparing this identity with
    \eqref{eq:sigma-rho-definition} and using the contraction property of
	conditional expectation in \(L^\alpha(\Omega)\), we obtain, for any
	\(\alpha'\in(\alpha,2\alpha)\),
	\begin{align}
		\sup_{0\le t\le T}
		\|\sigma^\rho(t)\|_{L^\alpha(\Omega)}
		&\le
		C\rho
		\sup_{t_0\le r,s\le t_1}
		\left\|
		\big\langle
		[P(r)-P(s)]\zeta_0,\zeta_0
		\big\rangle_H
		\right\|_{L^{\alpha'}(\Omega)}
		\notag\\
		&\quad+
		C\rho
		\left\|
		\sup_{t_0\le s\le t_1}
		\|z^\rho(s)\|_H^2
		\right\|_{L^{2\alpha}(\Omega)}
		\notag\\
		&\quad\times
		\Bigg[
		\left\|
		\int_{t_0}^{t_1}
		\|F(s,P(s))\|_{\mathfrak L(H)}\,ds
		\right\|_{L^{2\alpha}(\Omega)}
		\notag\\
		&\hspace{36mm}+
		\left\|
		\int_{(t_0,t_1]}
		\|G(s)\|_{\mathfrak L(H)}\,d\xi(s)
		\right\|_{L^{2\alpha}(\Omega)}
		\Bigg]
		\notag\\
		&\quad+
		C
		\left(
		\mathbb E\sup_{0\le s\le T}
		\|x^\rho(s)-\sqrt\rho\,z^\rho(s)\|_H^{4\alpha}
		\right)^{1/(4\alpha)}
		\notag\\
		&\quad\times
		\left(
		\mathbb E\sup_{0\le s\le T}
		\big(
		\|x^\rho(s)\|_H
		+
		\|\sqrt\rho\,z^\rho(s)\|_H
		\big)^{4\alpha}
		\right)^{1/(4\alpha)} .
		\label{eq:constant-spike-sigma-direct-bound}
	\end{align}
The three terms on the right-hand side of
\eqref{eq:constant-spike-sigma-direct-bound} arise from the increments of
\(\langle P(\cdot)\zeta_0,\zeta_0\rangle_H\), the integrals involving
\(F\) and \(G\) over the spike interval, and the replacement of
\(x^\rho\) by \(\sqrt\rho\,z^\rho\), respectively.
	The constant \(C\) is finite by
	\eqref{eq:BSIE-Ito-high-moment-data},
	\eqref{eq:BSIE-P-high-moment-from-data}, and
	\eqref{eq:BSIE-Ito-FP-high-moment}.
	
Lemma~\ref{lem:FV-BSIE-weak-right-continuity} gives
\[
 \langle[P(s)-P(t_0)]\zeta_0,\zeta_0\rangle_H
 \longrightarrow0
 \quad\text{in }L^\alpha(\Omega),\quad s\downarrow t_0.
\]
Since \(\zeta_0\) is bounded and
\eqref{eq:BSIE-P-high-moment-from-data} bounds these pairings in
\(L^{2\alpha}(\Omega)\), norm interpolation yields convergence
in \(L^{\alpha'}(\Omega)\) for every
\(\alpha'\in(\alpha,2\alpha)\). In particular,
\[
 \sup_{t_0\le r,s\le t_1}
 \|\langle[P(r)-P(s)]\zeta_0,\zeta_0\rangle_H\|_{L^{\alpha'}}
 \longrightarrow0.
\]
Here \(\alpha'>1\), even when \(\alpha=1\), so H\"older's
inequality controls the Brownian factors in
\eqref{eq:constant-spike-sigma-direct-bound} using their finite
moments.
	Brownian scaling and
	\eqref{eq:xrho-zrho-size-high} show that
	\(\left\| \sup_{t_0\le s\le t_1}\|z^\rho(s)\|_H^2 \right\|_{L^{2\alpha}(\Omega)}\)
	is bounded uniformly in \(\rho\).
The short-interval convergence established after
\eqref{eq:BSIE-dt-short-direct-increment} and
\eqref{eq:BSIE-dxi-short-direct-increment}, applied with \(r=t_0\),
also gives
\[
 \int_{t_0}^{t_1}\|F(s,P(s))\|_{\mathfrak L(H)}\,ds
 \longrightarrow0,
 \quad
 \int_{(t_0,t_1]}\|G(s)\|_{\mathfrak L(H)}\,d\xi(s)
 \longrightarrow0
 \quad\text{in }L^{2\alpha}(\Omega).
\]
    The product in \eqref{eq:constant-spike-sigma-direct-bound}
    containing \(\|\sup_s\|x^\rho(s)-\sqrt\rho z^\rho(s)\|_H\|_{L^{4\alpha}}\)
    is \(o(\sqrt\rho)O(\sqrt\rho)=o(\rho)\) by
    \eqref{eq:moving-spike-estimate-high} and
    \eqref{eq:xrho-zrho-size-high}. Therefore,
\begingroup
\postdisplaypenalty=10000
	\begin{equation}
		\sup_{0\le t\le T}
		\|\sigma^\rho(t)\|_{L^\alpha(\Omega)}
		=
		o(\rho)
		\label{eq:constant-V-spike-sigma}
	\end{equation}
	for bounded \(\mathcal F_{t_0}\)-measurable \(V\)-valued constant
	spike amplitudes.
\endgroup
	
We next consider a general constant amplitude
\(\zeta_0\in L^{8\alpha}(\mathcal F_{t_0};H)\).
Since \(V\) is dense in \(H\), \(\zeta_0\) can be approximated in
\(L^{8\alpha}(\Omega;H)\) by bounded
\(\mathcal F_{t_0}\)-measurable \(V\)-valued simple random variables.
To estimate the difference between the corresponding states, we apply
\eqref{eq:unbounded-xi-unweighted-SEE-estimate} to
    \(x^\rho_\zeta-x^\rho_{\widetilde\zeta}\), with diffusion datum
    \((\zeta-\widetilde\zeta)I_{E_\rho}\). This gives
    \begin{align}
&\left\|\sup_{0\le t\le T}
       \|x^\rho_\zeta(t)-x^\rho_{\widetilde\zeta}(t)\|_H
       \right\|_{L^{4\alpha}}
    \notag\\
    &\quad\le C\left[\mathbb E
      \left(\int_{E_\rho}\|\zeta(s)-\widetilde\zeta(s)\|_H^2ds
      \right)^{4\alpha}\right]^{1/(8\alpha)}
    \notag\\
    &\quad\le C\rho^{1/2}
       \left[\frac1\rho\int_{E_\rho}
        \mathbb E\|\zeta(s)-\widetilde\zeta(s)\|_H^{8\alpha}ds
       \right]^{1/(8\alpha)}.
\label{eq:BSIE-spike-amplitude-stability}
\end{align}
    Estimate \eqref{eq:BSIE-spike-amplitude-stability} and the inequality
	\[
	|\langle Ru,u\rangle_H-\langle Rv,v\rangle_H|
	\le
	\|R\|_{\mathfrak L(H)}
	\|u-v\|_H(\|u\|_H+\|v\|_H)
	\]
	show that the difference between the remainders for the two constant
    amplitudes is bounded by
	\[
	C\rho
	\|\zeta_0-\zeta_0^n\|_{L^{8\alpha}(\Omega;H)}
	\left(
	\|\zeta_0\|_{L^{8\alpha}(\Omega;H)}
	+
	\|\zeta_0^n\|_{L^{8\alpha}(\Omega;H)}
	\right)
	\]
	in \(L^\alpha(\Omega)\), uniformly in \(t\) and \(\rho\). Hence
	\eqref{eq:constant-V-spike-sigma} holds for every
	\(\zeta_0\in L^{8\alpha}(\mathcal F_{t_0};H)\).
	
For the original progressively measurable process \(\zeta\), we freeze
its value at the chosen Lebesgue point \(t_0\). Let \(x^{\rho,t_0}\)
denote the solution with amplitude \(\zeta(t_0)\) on \(E_\rho\).
By \eqref{eq:BSIE-spike-amplitude-stability}
    with \(\widetilde\zeta(s)=\zeta(t_0)\) and
    \eqref{eq:BSIE-Ito-zeta-Lebesgue-point},
    \[
 \mathbb E\sup_{0\le t\le T} \|x^\rho(t)-x^{\rho,t_0}(t)\|_H^{4\alpha} \le C\rho^{2\alpha} \left[\frac1\rho\int_{t_0}^{t_0+\rho} \mathbb E\|\zeta(s)-\zeta(t_0)\|_H^{8\alpha}ds\right]^{1/2} =o(\rho^{2\alpha}).
\]
    The terminal, \(F\)-, \(G\)-, and current-time \(P\)-terms in
	\eqref{eq:sigma-rho-definition} for the two amplitudes therefore differ by
	\(o(\rho)\) in \(L^\alpha(\Omega)\), uniformly in \(t\).
    For the difference between the integrals
    \(\int_{E_\rho\cap[t,T]}\langle P(s)\zeta(s),\zeta(s)\rangle_Hds\)
    and
    \(\int_{E_\rho\cap[t,T]}\langle P(s)\zeta(t_0),\zeta(t_0)\rangle_Hds\),
    conditional Jensen's and H\"older's inequalities give
	\begin{align*}
		&\sup_{0\le t\le T}
		\mathbb E\Bigg|
		\mathbb E\Bigg[
		\int_{E_\rho\cap[t,T]}
		\Big[
		\langle P(s)\zeta(s),\zeta(s)\rangle_H
		-
		\langle P(s)\zeta(t_0),\zeta(t_0)\rangle_H
		\Big]\,ds
		\,\Bigg|\,\mathcal F_t
		\Bigg]
		\Bigg|^\alpha
		\\
		&\quad\le
		\rho^{\alpha-1}
		\int_{t_0}^{t_0+\rho}
		\mathbb E\Big[
		\|P(s)\|_{\mathfrak L(H)}^\alpha
		\|\zeta(s)-\zeta(t_0)\|_H^\alpha
		\big(
		\|\zeta(s)\|_H+\|\zeta(t_0)\|_H
		\big)^\alpha
		\Big]\,ds
		\\
		&=o(\rho^\alpha),
	\end{align*}
	by \eqref{eq:BSIE-Ito-zeta-Lebesgue-point} and
	\eqref{eq:BSIE-Ito-P-zeta-local-bound}. Combining this \(o(\rho^\alpha)\) bound for the difference
    between the two \(P\)-integrals with
    \eqref{eq:constant-V-spike-sigma} for \(\zeta_0=\zeta(t_0)\)
    and the other quadratic-form differences gives
	\[
	\sup_{0\le t\le T}
	\mathbb E|\sigma^\rho(t)|^\alpha
	=
	o(\rho^\alpha),
	\]
	which is \eqref{eq:FV-BSIE-Ito-sigma}.
	
    We now construct \(Z^\rho\) and prove \eqref{eq:FV-BSIE-Ito-Z}
    by applying martingale representation to the terminal random variable
    \begin{align}
    R^\rho:={}&\langle\Xi x^\rho(T),x^\rho(T)\rangle_H
      +\int_0^T\langle F(s,P(s))x^\rho(s),x^\rho(s)\rangle_Hds
      \notag\\
      &+\int_{(0,T]}\langle G(s)x^\rho(s-),x^\rho(s-)\rangle_Hd\xi(s)
      -\int_{E_\rho}\langle P(s)\zeta(s),\zeta(s)\rangle_Hds.
    \label{eq:BSIE-Ito-martingale-terminal}
    \end{align}
    When \(\alpha>1\), H\"older's inequality with moment orders
    \((2\alpha,4\alpha,4\alpha)\),
    \eqref{eq:BSIE-Ito-general-state-bound}, and
    \eqref{eq:BSIE-Ito-P-zeta-local-bound} give
    \begin{equation}
    \|R^\rho\|_{L^\alpha}
    \le C\left\|\sup_t\|x^\rho(t)\|_H\right\|_{L^{4\alpha}}^2
      +C\int_{E_\rho}\|\zeta(s)\|_{L^{4\alpha}(\Omega;H)}^2ds
    =O(\rho).
    \label{eq:BSIE-Ito-terminal-alpha-bound}
    \end{equation}
    For \(\alpha=1\), we first establish an \(L^r\) bound for
    \(R^\rho\), with \(r>1\), so that Doob's and the BDG inequalities
    can be applied in \(L^r\). Using \(\kappa>2\), choose
    \[
      2<\kappa'<\kappa,\quad
      q=\frac{8\kappa}{\kappa+\kappa'}\in(4,8),\quad
      r=\frac{2q}{q+4}>1.
    \]
    Take \(\beta'\) large enough for the weighted SEE estimate at
    exponent \(4\) with \(\kappa'\), and write
    \(\widehat\rho(t)=e^{-\beta't-\kappa'C_D\xi(t)}\).
    Apply \eqref{eq:unbounded-xi-weighted-SEE-estimate} with moment parameter \(4\)
    and weight \(\widehat\rho\) to the constructed solution
    \(x^\rho\in\mathbb D^4_{\mathbb F}(0,T;H)\).
    The localization and Fatou argument in
    Proposition~\ref{prop:unbounded-xi-SEE-estimate} justifies this
    estimate without assuming an eighth moment of \(x^\rho\).
    The data terms in the weighted estimate are finite because
    \(\zeta\in L^8_{\mathbb F}(0,T;H)\). Thus
    \begin{equation}
    \mathbb E\sup_{0\le t\le T}
       \widehat\rho(t)^4\|x^\rho(t)\|_H^8
    \le C\mathbb E\left(\int_{E_\rho}\|\zeta(s)\|_H^2ds\right)^4
    \le C\rho^4\frac1\rho\int_{E_\rho}\mathbb E\|\zeta(s)\|_H^8ds.
    \label{eq:BSIE-Ito-endpoint-weighted-eight}
    \end{equation}
    Applying H\"older's inequality with exponents \(8/q\) and \(8/(8-q)\)
    to this weighted bound yields
    \begin{align*}
    \left\|\sup_t\|x^\rho(t)\|_H\right\|_{L^q}
    &\le e^{\beta'T/2}
      \left[\mathbb E\sup_t\widehat\rho(t)^4
                    \|x^\rho(t)\|_H^8\right]^{1/8}
      \left[\mathbb E e^{\frac{4q\kappa'}{8-q}C_D\xi(T)}
                    \right]^{(8-q)/(8q)}\\
    &\le C\sqrt\rho\left[
       \frac1\rho\int_{E_\rho}\mathbb E\|\zeta(s)\|_H^8ds
                    \right]^{1/8}=O(\sqrt\rho),
    \end{align*}
    because
    \(\frac{4q\kappa'}{8-q}=4\kappa, \ \frac1r=\frac12+\frac2q.\)
    The exponential moment in this estimate follows from
    \(\mathfrak E_{4\kappa,C_D}(\xi)<\infty\) at time zero.
    The data conditions and
    \eqref{eq:BSIE-Ito-FP-high-moment} give
    \[
    \begin{aligned}
    \Xi&\in L_w^2(\mathcal F_T;\mathfrak L(H)),\\
    F(\cdot,P(\cdot))
    &\in L^{2,2}_{\mathbb F,w}(0,T;\mathfrak L(H))
      \subset L^{1,2}_{\mathbb F,w}(0,T;\mathfrak L(H)),\\
    G&\in L^{1,2}_{\mathbb F,\xi,w}(0,T;\mathfrak L(H)).
    \end{aligned}
    \]
    Thus \(\|\Xi\|\), \(\int_0^T\|F(s,P(s))\|ds\), and
    \(\int_{(0,T]}\|G(s)\|d\xi(s)\) belong to \(L^2(\mathcal F_T)\).
    The product of each with \(\sup_t\|x^\rho(t)\|_H^2\) therefore has \(L^r(\mathcal F_T)\)-norm \(O(\rho)\).
    For \(\int_{E_\rho}\langle P(s)\zeta(s),\zeta(s)\rangle_Hds\)
    in \eqref{eq:BSIE-Ito-martingale-terminal}, Minkowski's and
    H\"older's inequalities give
    \begin{align*}
    \left\|\int_{E_\rho}
       \langle P(s)\zeta(s),\zeta(s)\rangle_Hds\right\|_{L^r}
    &\le \sup_s\bigl(\mathbb E\|P(s)\|_{\mathfrak L(H)}^2\bigr)^{1/2}
            \int_{E_\rho}\|\zeta(s)\|_{L^q(\Omega;H)}^2ds\\
    &\le C\rho\left[
       \frac1\rho\int_{E_\rho}\mathbb E\|\zeta(s)\|_H^8ds
                 \right]^{1/4}=O(\rho).
    \end{align*}
    Consequently, \(R^\rho\in L^r\) and
    \(\|R^\rho\|_{L^r}=O(\rho)\) at the endpoint.

    The Brownian martingale representation theorem now gives a
    continuous version
    \[
      M^\rho(t):=\mathbb E[R^\rho\mid\mathcal F_t]
      =\mathbb E R^\rho+\int_0^t Z^\rho(s)dW(s).
    \]
    For \(\alpha>1\), Doob's inequality in \(L^\alpha\) and the BDG inequality,
    together with \eqref{eq:BSIE-Ito-terminal-alpha-bound}, imply
    \[
      \left\|\left(\int_0^T|Z^\rho(s)|^2ds\right)^{1/2}
          \right\|_{L^\alpha}=O(\rho).
    \]
    For \(\alpha=1\), these inequalities are applied in \(L^r\),
    where \(r>1\); hence
    \[
      \mathbb E\left(\int_0^T|Z^\rho(s)|^2ds\right)^{1/2}
      \le \left\|\left(\int_0^T|Z^\rho(s)|^2ds\right)^{1/2}
          \right\|_{L^r}
      \le C_r\|R^\rho\|_{L^r}=O(\rho).
    \]
    These estimates prove \eqref{eq:FV-BSIE-Ito-Z} for every
    \(\alpha\ge1\).

    We now choose a process version of the fixed-time remainder in
    \eqref{eq:sigma-rho-definition}:
    \begin{align*}
    \sigma^\rho(t)
    &=M^\rho(t)-\langle P(t)x^\rho(t),x^\rho(t)\rangle_H
      -\int_0^t\langle F(s,P(s))x^\rho(s),x^\rho(s)\rangle_H\,ds\\
    &\quad-\int_{(0,t]}\langle G(s)x^\rho(s-),x^\rho(s-)\rangle_H\,d\xi(s)
      +\int_0^t I_{E_\rho}(s)
        \langle P(s)\zeta(s),\zeta(s)\rangle_H\,ds.
    \end{align*}
    Each process on the right-hand side is progressively measurable.
    For each deterministic time, this version agrees \(\mathbb P\text{-a.s.}\)
    with the random variable in \eqref{eq:sigma-rho-definition}.
    With this choice,
    \eqref{eq:FV-BSIE-Ito} holds outside one evanescent set.
    Estimate \eqref{eq:FV-BSIE-Ito-sigma} gives
    \(\mathbb E\int_0^T|\sigma^\rho(t)|^\alpha dt
    \le T\sup_{t\le T}\mathbb E|\sigma^\rho(t)|^\alpha<\infty\),
    hence \(\sigma^\rho\in L^\alpha_{\mathbb F}(0,T)\).

	For uniqueness, suppose that two pairs satisfy \eqref{eq:FV-BSIE-Ito}. Subtracting the two identities gives
	\[
	\sigma_1^\rho(t)-\sigma_2^\rho(t)
	=
	-\int_t^T
	\big(
	Z_1^\rho(s)-Z_2^\rho(s)
	\big)\,dW(s).
	\]
The difference integrand lies in
\(L^{2,\alpha}_{\mathbb F}(0,T)\subset L^{2,1}_{\mathbb F}(0,T)\).
Its stochastic integral is therefore an \(H^1\) martingale by
the BDG inequality. Taking conditional expectation in the difference identity
gives \(\sigma_1^\rho=\sigma_2^\rho\); uniqueness of the Brownian
martingale representation gives
\(Z_1^\rho=Z_2^\rho\), \(dt\otimes d\mathbb P\)-almost everywhere.
	This completes the proof.
\end{proof}

\subsection{Proofs of the variational estimates}
\label{subsec:variational-adjoint-estimate-proofs}

\begin{proof}[Proof of Lemma~\ref{lem:regular-variation-estimates}]
We first estimate the state increment and the two variations,
and then the errors in the first- and second-order expansions.
Put \(\Delta x^\rho=x^\rho-\bar x\) and define
\begin{align}
    a_x^\rho(t)
    &:=
    \int_0^1
        a_x\big(t,\bar x(t)+\theta\Delta x^\rho(t),u^\rho(t)\big)
    \,d\theta,
    \label{eq:a-x-rho-average-proof}
    \\
    b_x^\rho(t)
    &:=
    \int_0^1
        b_x\big(t,\bar x(t)+\theta\Delta x^\rho(t),u^\rho(t)\big)
    \,d\theta,
    \notag
    \\
    \gamma_x^\rho(t)
    &:=
    \int_0^1
        \gamma_x\big(t,\bar x(t-)+\theta\Delta x^\rho(t-)\big)
    \,d\theta.
    \label{eq:gamma-x-rho-average-proof}
\end{align}
Subtracting the equation for \(\bar x\) from that for \(x^\rho\)
and applying the integral mean-value formula with the averaged
derivatives in \eqref{eq:a-x-rho-average-proof} and
\eqref{eq:gamma-x-rho-average-proof}, we obtain
\begin{equation}
\left\{
\begin{aligned}
    d\Delta x^\rho(t)
    &=
    \big[
        A(t)\Delta x^\rho(t)
        +
        a_x^\rho(t)\Delta x^\rho(t)
        +
        \delta a(t)I_{E_\rho}(t)
    \big]\,dt
    \\
    &\quad
    +
    \big[
        B(t)\Delta x^\rho(t)
        +
        b_x^\rho(t)\Delta x^\rho(t)
        +
        \delta b(t)I_{E_\rho}(t)
    \big]\,dW(t)
    \\
    &\quad
    +
    \gamma_x^\rho(t)\Delta x^\rho(t-)\,d\bar\xi(t),
    \quad t\in[0,T],
    \\
    \Delta x^\rho(0)&=0.
\end{aligned}
\right.
\label{eq:Delta-x-rho-equation-proof}
\end{equation}

For \(\alpha\ge1\), applying
\eqref{eq:unbounded-xi-unweighted-SEE-estimate} to
\eqref{eq:first-variation-regular}, with diffusion datum
\(\delta b I_{E_\rho}\) and zero drift and Stieltjes data, gives
\begin{align}
&\mathbb E\left[
    \sup_{0\le t\le T}\|x_1^\rho(t)\|_H^{2\alpha}
    +\left(\int_0^T\|x_1^\rho(t)\|_V^2dt\right)^\alpha
\right]
\notag\\
&\quad\le
C_\alpha\mathfrak E_{2\alpha\kappa,L_\gamma}(\bar\xi)^{1/2}
\left\{
\mathbb E\left(
    \int_{E_\rho}\|\delta b(t)\|_H^2dt
\right)^{2\alpha}
\right\}^{1/2}
=O(\rho^\alpha).
\label{eq:unbounded-xi-regular-first-variation-estimate}
\end{align}
For integer \(\alpha\), the bound for
\(\mathbb E(\int_{E_\rho}\|\delta b(t)\|_H^2\,dt)^{2\alpha}\)
follows from
\[
\mathbb E\left(
    \int_{E_\rho}\|\delta b(t)\|_H^2dt
\right)^{2\alpha}
\le
\rho^{2\alpha-1}\int_{E_\rho}\mathbb E\|\delta b(t)\|_H^{4\alpha}dt
=O(\rho^{2\alpha}).
\]

Applying \eqref{eq:unbounded-xi-unweighted-SEE-estimate} to
\eqref{eq:Delta-x-rho-equation-proof} also gives
\begin{align}
 &\mathbb E\left[\sup_{t\le T}\|\Delta x^\rho(t)\|_H^{2\alpha}
       +\left(\int_0^T\|\Delta x^\rho(t)\|_V^2\,dt\right)^\alpha\right]
 \notag\\
 &\quad\le C_\alpha\left\{\mathbb E\left[
       \left(\int_{E_\rho}\|\delta a(t)\|_H\,dt\right)^{4\alpha}
       +\left(\int_{E_\rho}\|\delta b(t)\|_H^2\,dt\right)^{2\alpha}
                    \right]\right\}^{1/2}
 =O(\rho^\alpha).
 \label{eq:regular-state-increment-moments}
\end{align}
Here the conditional exponential-moment factor is absorbed into
\(C_\alpha\). At the chosen \(t_0\), the drift datum satisfies
\(\mathbb E(\int_{E_\rho}\|\delta a(t)\|_Hdt)^{4\alpha}
=O(\rho^{4\alpha})\). The diffusion datum satisfies
\(\mathbb E(\int_{E_\rho}\|\delta b(t)\|_H^2dt)^{2\alpha}
=O(\rho^{2\alpha})\).

For the second variation, we first estimate the quadratic
Stieltjes datum in \eqref{eq:second-variation-regular}:
\begin{equation}
 \mathbb E\left( \int_{(0,T]}\|x_1^\rho(t-)\|_H^2d\bar\xi(t) \right)^{4\alpha} \le \left(\mathbb E\bar\xi(T)^{8\alpha}\right)^{1/2} \left(\mathbb E\sup_{t\le T}\|x_1^\rho(t)\|_H^{16\alpha}\right)^{1/2} =O(\rho^{4\alpha}).
\label{eq:second-variation-Stieltjes-data-bound}
\end{equation}
Applying \eqref{eq:unbounded-xi-unweighted-SEE-estimate} to
\eqref{eq:second-variation-regular}, together with
\eqref{eq:second-variation-Stieltjes-data-bound} and the expressions
for the drift and diffusion coefficients in
\eqref{eq:second-variation-regular}, gives
\[
\mathbb E\sup_{t\le T}\|x_2^\rho(t)\|_H^{2\alpha}
+\mathbb E\left(\int_0^T\|x_2^\rho(t)\|_V^2dt\right)^\alpha
=O(\rho^{2\alpha}).
\]

We next estimate the first-order approximation error. Define
$
    r_1^\rho(t):=\Delta x^\rho(t)-x_1^\rho(t).
$
Subtracting \eqref{eq:first-variation-regular} from
\eqref{eq:Delta-x-rho-equation-proof} gives
\begin{equation}
\left\{
\begin{aligned}
    dr_1^\rho(t)
    &=
    \big[
        \bar A(t)r_1^\rho(t)+R_{a,1}^\rho(t)
    \big]\,dt
    +
    \big[
        \bar B(t)r_1^\rho(t)+R_{b,1}^\rho(t)
    \big]\,dW(t)
    \\
    &\quad
    +
    \big[
        \bar\gamma_x(t)r_1^\rho(t-)+R_{\gamma,1}^\rho(t)
    \big]\,d\bar\xi(t),
    \\
    r_1^\rho(0)&=0,
\end{aligned}
\right.
\label{eq:r1-rho-equation-proof}
\end{equation}
where
\begin{align*}
    R_{a,1}^\rho(t)
    &:=
    a(t,x^\rho(t),u^\rho(t))
    -
    a(t,\bar x(t),\bar u(t))
    -
    \bar a_x(t)\Delta x^\rho(t),
    \\
    R_{b,1}^\rho(t)
    &:=
    b(t,x^\rho(t),u^\rho(t))
    -
    b(t,\bar x(t),\bar u(t))
    -
    \bar b_x(t)\Delta x^\rho(t)
    -
    \delta b(t)I_{E_\rho}(t),
    \\
    R_{\gamma,1}^\rho(t)
    &:=
    \gamma(t,x^\rho(t-))
    -
    \gamma(t,\bar x(t-))
    -
    \bar\gamma_x(t)\Delta x^\rho(t-).
\end{align*}
Applying the integral mean-value formula to the state differences
in the definitions of
\(R_{a,1}^\rho,R_{b,1}^\rho,R_{\gamma,1}^\rho\) gives
\begin{align*}
    R_{a,1}^\rho(t)
    &=
    \delta a(t)I_{E_\rho}(t)
    +
    \delta a_x(t)\Delta x^\rho(t)I_{E_\rho}(t)
    \\
    &\quad
    +
    \int_0^1
    \Big[
        a_x\big(t,\bar x(t)+\theta\Delta x^\rho(t),u^\rho(t)\big)
        -
        a_x\big(t,\bar x(t),u^\rho(t)\big)
    \Big]
    \Delta x^\rho(t)\,d\theta ,
\end{align*}
\begin{align*}
    R_{b,1}^\rho(t)
    &=
    \delta b_x(t)\Delta x^\rho(t)I_{E_\rho}(t)
    \\
    &\quad
    +
    \int_0^1
    \Big[
        b_x\big(t,\bar x(t)+\theta\Delta x^\rho(t),u^\rho(t)\big)
        -
        b_x\big(t,\bar x(t),u^\rho(t)\big)
    \Big]
    \Delta x^\rho(t)\,d\theta ,
\end{align*}
and
\begin{align*}
    R_{\gamma,1}^\rho(t)
    =
    \int_0^1
    \Big[
        \gamma_x\big(t,\bar x(t-)+\theta\Delta x^\rho(t-)\big)
        -
        \bar\gamma_x(t)
    \Big]
    \Delta x^\rho(t-)\,d\theta.
\end{align*}
By boundedness of the second derivatives,
\begin{equation}
\begin{aligned}
    \|R_{a,1}^\rho(t)\|_H
    &\le
    \|\delta a(t)\|_H I_{E_\rho}(t)
    +
    C I_{E_\rho}(t)\|\Delta x^\rho(t)\|_H
    +
    C\|\Delta x^\rho(t)\|_H^2,
    \\
    \|R_{b,1}^\rho(t)\|_H
    &\le
    C I_{E_\rho}(t)\|\Delta x^\rho(t)\|_H
    +
    C\|\Delta x^\rho(t)\|_H^2,
    \\
    \|R_{\gamma,1}^\rho(t)\|_H
    &\le
    C\|\Delta x^\rho(t-)\|_H^2.
\end{aligned}
\label{eq:first-variation-remainders-pointwise}
\end{equation}
We integrate the bounds for
\(\|R_{a,1}^\rho\|_H\), \(\|R_{b,1}^\rho\|_H^2\), and
\(\|R_{\gamma,1}^\rho\|_H\) from
\eqref{eq:first-variation-remainders-pointwise} against
\(dt,dt,d\bar\xi\), respectively. Then
\eqref{eq:regular-state-increment-moments} at the required moment
orders and H\"older's inequality for the factor \(\bar\xi(T)\)
give
\[
 \|R_{a,1}^\rho\|_{L^{1,4\alpha}_{\mathbb F}(0,T;H)}
 +\|R_{b,1}^\rho\|_{L^{2,4\alpha}_{\mathbb F}(0,T;H)}
 +\|R_{\gamma,1}^\rho\|_{L^{1,4\alpha}_{\mathbb F,\bar\xi}(0,T;H)}
 =O(\rho).
\]
Applying \eqref{eq:unbounded-xi-unweighted-SEE-estimate} to
\eqref{eq:r1-rho-equation-proof} with these data bounds gives
\[
\mathbb E\sup_{t\le T}\|r_1^\rho(t)\|_H^{2\alpha}
+\mathbb E\left(\int_0^T\|r_1^\rho(t)\|_V^2dt\right)^\alpha
=O(\rho^{2\alpha}).
\]

It remains to estimate the second-order approximation error. Define
$
    r_2^\rho(t):=
    \Delta x^\rho(t)-x_1^\rho(t)-x_2^\rho(t).
$
Subtracting \eqref{eq:first-variation-regular} and
\eqref{eq:second-variation-regular} from
\eqref{eq:Delta-x-rho-equation-proof} gives
\begin{equation}
\left\{
\begin{aligned}
    dr_2^\rho(t)
    &=
    \big[
        \bar A(t)r_2^\rho(t)+R_{a,2}^\rho(t)
    \big]\,dt
    +
    \big[
        \bar B(t)r_2^\rho(t)+R_{b,2}^\rho(t)
    \big]\,dW(t)
    \\
    &\quad
    +
    \big[
        \bar\gamma_x(t)r_2^\rho(t-)+R_{\gamma,2}^\rho(t)
    \big]\,d\bar\xi(t),
    \\
    r_2^\rho(0)&=0,
\end{aligned}
\right.
\label{eq:r2-rho-equation-proof}
\end{equation}
where
\begin{align*}
    R_{a,2}^\rho(t)
    &:=
    a(t,x^\rho(t),u^\rho(t))
    -
    a(t,\bar x(t),\bar u(t))
    -
    \bar a_x(t)\Delta x^\rho(t)
    -
    \delta a(t)I_{E_\rho}(t)
    \\
    &\quad
    -
    \frac12
    \bar a_{xx}(t)\big(x_1^\rho(t),x_1^\rho(t)\big),
    \\
    R_{b,2}^\rho(t)
    &:=
    b(t,x^\rho(t),u^\rho(t))
    -
    b(t,\bar x(t),\bar u(t))
    -
    \bar b_x(t)\Delta x^\rho(t)
    -
    \delta b(t)I_{E_\rho}(t)
    \\
    &\quad
    -
    \delta b_x(t)x_1^\rho(t)I_{E_\rho}(t)
    -
    \frac12
    \bar b_{xx}(t)\big(x_1^\rho(t),x_1^\rho(t)\big),
    \\
    R_{\gamma,2}^\rho(t)
    &:=
    \gamma(t,x^\rho(t-))
    -
    \gamma(t,\bar x(t-))
    -
    \bar\gamma_x(t)\Delta x^\rho(t-)
    -
    \frac12
    \bar\gamma_{xx}(t)
    \big(x_1^\rho(t-),x_1^\rho(t-)\big).
\end{align*}
Applying Taylor's formula in the state variable to
\(a,b,\gamma\) in the definitions of
\(R_{a,2}^\rho,R_{b,2}^\rho,R_{\gamma,2}^\rho\), and using
\(r_1^\rho=\Delta x^\rho-x_1^\rho\), we obtain
\begin{equation}
\begin{aligned}
	R_{a,2}^\rho(t)
	&=
	\int_0^1(1-\theta)
	\Big[
	a_{xx}\bigl(t,\bar x(t)+\theta\Delta x^\rho(t),\bar u(t)\bigr)
	-\bar a_{xx}(t)
	\Big]
	\bigl(\Delta x^\rho(t),\Delta x^\rho(t)\bigr)\,d\theta
	\\
	&\quad
	+\frac12\bar a_{xx}(t)
	\bigl(r_1^\rho(t),\Delta x^\rho(t)\bigr)
	+\frac12\bar a_{xx}(t)
	\bigl(x_1^\rho(t),r_1^\rho(t)\bigr)
	\\
	&\quad
	+I_{E_\rho}(t)
	\int_0^1
	\Big[
	a_x\bigl(t,\bar x(t)+\theta\Delta x^\rho(t),v(t)\bigr)
	\\
	&\hspace{38mm}
	-a_x\bigl(t,\bar x(t)+\theta\Delta x^\rho(t),\bar u(t)\bigr)
	\Big]
	\Delta x^\rho(t)\,d\theta ,
\end{aligned}
\label{eq:second-remainder-drift-Taylor}
\end{equation}
\begin{equation}
\begin{aligned}
	R_{b,2}^\rho(t)
	&=
	\int_0^1(1-\theta)
	\Big[
	b_{xx}\bigl(t,\bar x(t)+\theta\Delta x^\rho(t),\bar u(t)\bigr)
	-\bar b_{xx}(t)
	\Big]
	\bigl(\Delta x^\rho(t),\Delta x^\rho(t)\bigr)\,d\theta
	\\
	&\quad
	+\frac12\bar b_{xx}(t)
	\bigl(r_1^\rho(t),\Delta x^\rho(t)\bigr)
	+\frac12\bar b_{xx}(t)
	\bigl(x_1^\rho(t),r_1^\rho(t)\bigr)
	\\
	&\quad
	+\delta b_x(t)r_1^\rho(t)I_{E_\rho}(t)
	\\
	&\quad
	+I_{E_\rho}(t)
	\int_0^1
	\Big[
	b_x\bigl(t,\bar x(t)+\theta\Delta x^\rho(t),v(t)\bigr)
	\\
	&\hspace{29mm}
	-b_x\bigl(t,\bar x(t)+\theta\Delta x^\rho(t),\bar u(t)\bigr)
	-\delta b_x(t)
	\Big]
	\Delta x^\rho(t)\,d\theta ,
\end{aligned}
\label{eq:second-remainder-diffusion-Taylor}
\end{equation}
\begin{equation}
\begin{aligned}
	R_{\gamma,2}^\rho(t)
	&=
	\int_0^1(1-\theta)
	\Big[
	\gamma_{xx}\bigl(
	t,\bar x(t-)+\theta\Delta x^\rho(t-)
	\bigr)
	-\bar\gamma_{xx}(t)
	\Big]
	\\
	&\hspace{25mm}\times
	\bigl(
	\Delta x^\rho(t-),\Delta x^\rho(t-)
	\bigr)\,d\theta
	\\
	&\quad
	+\frac12\bar\gamma_{xx}(t)
	\bigl(
	r_1^\rho(t-),\Delta x^\rho(t-)
	\bigr)
	+\frac12\bar\gamma_{xx}(t)
	\bigl(
	x_1^\rho(t-),r_1^\rho(t-)
	\bigr).
\end{aligned}
\label{eq:second-remainder-Stieltjes-Taylor}
\end{equation}

We first estimate the terms containing differences of second
derivatives in \eqref{eq:second-remainder-drift-Taylor}--
\eqref{eq:second-remainder-Stieltjes-Taylor}.
By \eqref{eq:regular-state-increment-moments},
\(\Delta x^\rho\to0\) uniformly in time in probability, and the
same convergence holds for the left limits. By continuity of the
Hessians in the state variable and their uniform bounds, the
dominated convergence theorem for convergence
in probability gives
\begin{align*}
 &\mathbb E\left[
 \int_{(0,T]}\int_0^1(1-\theta)
 \left\|\gamma_{xx}(t,\bar x(t-)+\theta\Delta x^\rho(t-))
                   -\bar\gamma_{xx}(t)\right\|
 \,d\theta\,d\bar\xi(t)\right]^{8\alpha}\longrightarrow0,\\
 &\mathbb E\left[
 \int_0^T\int_0^1(1-\theta)
 \left\|\psi_{xx}(t,\bar x(t)+\theta\Delta x^\rho(t),\bar u(t))
                   -\bar\psi_{xx}(t)\right\|^2
 \,d\theta\,dt\right]^{4\alpha}\longrightarrow0,
 \quad\psi=a,b.
\end{align*}
The integrals inside the expectations are bounded by \(C\bar\xi(T)\) and \(CT\),
respectively; the required domination follows from
\(\mathbb E\bar\xi(T)^{8\alpha}<\infty\).
The Hessian norms in these estimates are taken in the space of
bounded bilinear maps from \(H\times H\) to \(H\).

For the first term in
\eqref{eq:second-remainder-Stieltjes-Taylor}, the
Cauchy--Schwarz inequality yields
\begin{equation}
\begin{aligned}
 &\mathbb E\left(\int_{(0,T]}\int_0^1(1-\theta)
 \left\|\gamma_{xx}\bigl(t,\bar x(t-)+\theta\Delta x^\rho(t-)\bigr)
                    -\bar\gamma_{xx}(t)\right\|
 \|\Delta x^\rho(t-)\|_H^2\,d\theta\,d\bar\xi(t)\right)^{4\alpha}\\
 &\quad\le
 \left(\mathbb E\sup_{t\le T}\|\Delta x^\rho(t)\|_H^{16\alpha}\right)^{1/2}\\
 &\quad\times\left\{\mathbb E\left[
 \int_{(0,T]}\int_0^1(1-\theta)
 \left\|\gamma_{xx}\bigl(t,\bar x(t-)+\theta\Delta x^\rho(t-)\bigr)
                    -\bar\gamma_{xx}(t)\right\|
 \,d\theta\,d\bar\xi(t)\right]^{8\alpha}\right\}^{1/2}\\
 &\quad=o(\rho^{4\alpha}).
\end{aligned}
\label{eq:second-remainder-Stieltjes-Hessian-bound}
\end{equation}

For the first terms in
\eqref{eq:second-remainder-drift-Taylor} and
\eqref{eq:second-remainder-diffusion-Taylor}, the
Cauchy--Schwarz inequality yields, for \(\psi=a,b\),
\begin{equation}
\begin{aligned}
 &\mathbb E\left(\int_0^T
 \left\|\int_0^1(1-\theta)
 \Big[\psi_{xx}\bigl(t,\bar x(t)+\theta\Delta x^\rho(t),\bar u(t)\bigr)
                     -\bar\psi_{xx}(t)\Big]
       (\Delta x^\rho(t),\Delta x^\rho(t))\,d\theta\right\|_H^2
 \,dt\right)^{2\alpha}\\
 &\quad\le C_\alpha
 \left(\mathbb E\sup_{t\le T}\|\Delta x^\rho(t)\|_H^{16\alpha}\right)^{1/2}\\
 &\quad\times\left\{\mathbb E\left[\int_0^T\int_0^1(1-\theta)
 \left\|\psi_{xx}\bigl(t,\bar x(t)+\theta\Delta x^\rho(t),\bar u(t)\bigr)
                     -\bar\psi_{xx}(t)\right\|^2
 \,d\theta\,dt\right]^{4\alpha}\right\}^{1/2}\\
 &\quad=o(\rho^{4\alpha}).
\end{aligned}
\label{eq:second-remainder-dt-Hessian-bound}
\end{equation}

We now estimate the bilinear terms containing \(r_1^\rho\)
and the terms supported on \(E_\rho\).
The two bilinear terms and the term multiplied by \(I_{E_\rho}\)
in \eqref{eq:second-remainder-drift-Taylor} have total norm at most
\[
C I_{E_\rho}\bigl(\|\Delta x^\rho\|_H+\|\Delta x^\rho\|_H^2\bigr)
+C\|r_1^\rho\|_H(\|\Delta x^\rho\|_H+\|x_1^\rho\|_H).
\]
The two bilinear terms and both terms multiplied by \(I_{E_\rho}\)
in \eqref{eq:second-remainder-diffusion-Taylor} have total norm at most
\[
C I_{E_\rho}\bigl(\|r_1^\rho\|_H+\|\Delta x^\rho\|_H^2\bigr)
+C\|r_1^\rho\|_H(\|\Delta x^\rho\|_H+\|x_1^\rho\|_H).
\]
The two bilinear terms in
\eqref{eq:second-remainder-Stieltjes-Taylor} have total norm at most
\[
 C\|r_1^\rho(t-)\|_H
       (\|\Delta x^\rho(t-)\|_H+\|x_1^\rho(t-)\|_H).
\]
For the \(d\bar\xi\)-integral of this bound, H\"older's inequality gives
\begin{align*}
 &\left\|\int_{(0,T]}\|r_1^\rho(t-)\|_H
       \bigl(\|\Delta x^\rho(t-)\|_H+\|x_1^\rho(t-)\|_H\bigr)
       \,d\bar\xi(t)\right\|_{L^{4\alpha}}\\
 &\quad\le \|\bar\xi(T)\|_{L^{12\alpha}}
       \left\|\sup_{t\le T}\|r_1^\rho(t)\|_H\right\|_{L^{12\alpha}}
       \left\|\sup_{t\le T}
          (\|\Delta x^\rho(t)\|_H+\|x_1^\rho(t)\|_H)
       \right\|_{L^{12\alpha}}
 =O(\rho^{3/2}).
\end{align*}
Combining the bounds for the bilinear terms and the terms supported
on the spike interval with
\eqref{eq:second-remainder-Stieltjes-Hessian-bound} and
\eqref{eq:second-remainder-dt-Hessian-bound} gives
\begin{align}
&\mathbb E\left(\int_0^T\|R_{a,2}^\rho(t)\|_H\,dt\right)^{4\alpha}
 +\mathbb E\left(\int_0^T\|R_{b,2}^\rho(t)\|_H^2\,dt\right)^{2\alpha}+
 \mathbb E\left(\int_{(0,T]}\|R_{\gamma,2}^\rho(t)\|_H
                                    \,d\bar\xi(t)\right)^{4\alpha}
 =o(\rho^{4\alpha}).
\label{eq:regular-second-error-data-moments}
\end{align}
Applying \eqref{eq:unbounded-xi-unweighted-SEE-estimate} to
\eqref{eq:r2-rho-equation-proof} with the three data bounds in
\eqref{eq:regular-second-error-data-moments}, we obtain
\[
\mathbb E\sup_{t\le T}
\|x^\rho(t)-\bar x(t)-x_1^\rho(t)-x_2^\rho(t)\|_H^{2\alpha}
=o(\rho^{2\alpha}).
\]
\end{proof}

\begin{proof}[Proof of Lemma~\ref{lem:singular-variation-estimate}]
We use \(K=\bar\xi+\eta\) as defined in
Section~\ref{subsec:variational-equations}.
We apply Remark~\ref{rem:multiple-FV-driver-reduction} with the two
nondecreasing drivers \(\bar\xi,\eta\) to
\eqref{eq:first-variation-singular} and the exact state-difference
equation, using
\[
 d\nu=d\eta-d\bar\xi,\quad
 d\xi^\varepsilon=(1-\varepsilon)d\bar\xi+\varepsilon d\eta.
\]
The coefficients multiplying the unknown state in these two
Stieltjes integrals have Lipschitz bounds \(L_\gamma,0\) for the first
variation, and \((1-\varepsilon)L_\gamma,\varepsilon L_\gamma\)
for the exact difference. Hence their combined state Lipschitz
constant is at most \(L_\gamma\), uniformly in \(\varepsilon\).

The homogeneous driver in the exact difference equation is
\(\xi^\varepsilon\). The identity
\begin{equation}
\gamma(t,x^\varepsilon(t-))d\xi^\varepsilon(t)
-\bar\gamma(t)d\bar\xi(t)
=\bigl[\gamma(t,x^\varepsilon(t-))-\bar\gamma(t)\bigr]
 d\xi^\varepsilon(t)
+\varepsilon\bar\gamma(t)d\nu(t)
\label{eq:singular-state-Stieltjes-decomposition}
\end{equation}
expresses the Stieltjes term as a homogeneous increment integrated
against \(d\xi^\varepsilon\) and an inhomogeneous term independent
of the perturbed state.
Moreover,
\[
 \mathbb E\left( \int_{(0,T]}\|\bar\gamma(t)\|_Hd|\nu|_v(t) \right)^{4\alpha} \le C_\alpha \left(\mathbb E K(T)^{8\alpha}\right)^{1/2} \left(\mathbb E[1+\sup_{t\le T}\|\bar x(t)\|_H]^{8\alpha}\right)^{1/2} <\infty,
\]
since \(d|\nu|_v\le dK\). We apply the multiple-driver form of
\eqref{eq:unbounded-xi-unweighted-SEE-estimate} (see
Remark~\ref{rem:multiple-FV-driver-reduction}) to the equation
for \(x^\varepsilon-\bar x\), with the decomposition
\eqref{eq:singular-state-Stieltjes-decomposition}, and to
\eqref{eq:first-variation-singular}.
The respective additive data are
\(\varepsilon\bar\gamma\,d\nu\) and \(\bar\gamma\,d\nu\), and we obtain
\[
\mathbb E\sup_{t\le T}\|x^\varepsilon(t)-\bar x(t)\|_H^{2\alpha}
\le C_\alpha\varepsilon^{2\alpha},
\]
and
\begin{align}
&\mathbb E\left[
    \sup_{t\le T}\|y(t)\|_H^{2\alpha}
    +\left(\int_0^T\|y(t)\|_V^2dt\right)^\alpha
\right]
\notag\\
&\quad\le
C_\alpha\mathfrak E_{2\alpha\kappa,L_\gamma}(K)^{1/2}
\left\{\mathbb E\left(
    \int_{(0,T]}\|\bar\gamma(t)\|_Hd|\nu|_v(t)
\right)^{4\alpha}\right\}^{1/2}<\infty.
\label{eq:unbounded-xi-singular-first-variation-estimate}
\end{align}
Consequently, writing
\(y^\varepsilon=(x^\varepsilon-\bar x)/\varepsilon\), we have
\begin{equation}
 \sup_{0<\varepsilon\le1}\mathbb E\sup_{t\le T}
 \|y^\varepsilon(t)\|_H^{2\alpha}<\infty,\quad \alpha\ge1.
 \label{eq:singular-scaled-increment-high-moment}
\end{equation}
Divide the equation for \(x^\varepsilon-\bar x\) by
\(\varepsilon\), and subtract \eqref{eq:first-variation-singular}.
In the resulting equation for \(y^\varepsilon-y\), we write the
homogeneous coefficients as \(\bar A,\bar B,\bar\gamma_xd\bar\xi\).
The inhomogeneous terms are
\[
\begin{aligned}
 R_\psi^\varepsilon(t)
  &=\varepsilon^{-1}\big[\psi(t,x^\varepsilon(t),\bar u(t))
    -\psi(t,\bar x(t),\bar u(t))
    -\bar\psi_x(t)(x^\varepsilon(t)-\bar x(t))\big],
    \quad\psi=a,b,\\
 R_\gamma^\varepsilon(t)
  &=\varepsilon^{-1}\big[\gamma(t,x^\varepsilon(t-))
    -\bar\gamma(t)-\bar\gamma_x(t)(x^\varepsilon(t-)-\bar x(t-))\big],\\
 \widetilde R_\gamma^\varepsilon(t)
  &=\varepsilon^{-1}\big[\gamma(t,x^\varepsilon(t-))-\bar\gamma(t)\big].
\end{aligned}
\]
Taylor's formula and the
bounded second derivatives give bounds for
\(R_a^\varepsilon,R_b^\varepsilon,R_\gamma^\varepsilon\);
the Lipschitz bound on \(\gamma\) gives the bound for
\(\widetilde R_\gamma^\varepsilon\):
\begin{align*}
\|R_a^\varepsilon(t)\|_H+\|R_b^\varepsilon(t)\|_H
&\le C\varepsilon\|y^\varepsilon(t)\|_H^2,
\\
\|R_\gamma^\varepsilon(t)\|_H
&\le C\varepsilon\|y^\varepsilon(t-)\|_H^2,
\\
\|\widetilde R_\gamma^\varepsilon(t)\|_H
&\le L_\gamma\|y^\varepsilon(t-)\|_H.
\end{align*}
The term \(\widetilde R_\gamma^\varepsilon\) is integrated
against \(\varepsilon d\nu\), whereas
\(R_\gamma^\varepsilon\) is integrated against \(d\bar\xi\).
The \(L^{4\alpha}\) moment of the sum of
\(\int_{(0,T]}\|R_\gamma^\varepsilon\|_Hd\bar\xi\)
and \(\varepsilon\int_{(0,T]}
\|\widetilde R_\gamma^\varepsilon\|_Hd|\nu|_v\) satisfies
\begin{equation}
\begin{aligned}
&\mathbb E\left(
 \int_{(0,T]}\|R_\gamma^\varepsilon(t)\|_Hd\bar\xi(t)
 +\varepsilon\int_{(0,T]}
       \|\widetilde R_\gamma^\varepsilon(t)\|_Hd|\nu|_v(t)
\right)^{4\alpha}
\\
&\quad\le C_\alpha\varepsilon^{4\alpha}
\left(\mathbb E K(T)^{8\alpha}\right)^{1/2}
\left(\mathbb E[1+\sup_{t\le T}\|y^\varepsilon(t)\|_H]^{16\alpha}\right)^{1/2}
\le C_\alpha\varepsilon^{4\alpha}.
\end{aligned}
\label{eq:singular-remainder-Stieltjes-data}
\end{equation}
The bounds on \(R_a^\varepsilon\) and
\(R_b^\varepsilon\), together with
\eqref{eq:singular-scaled-increment-high-moment} at orders
\(4\alpha\) and \(8\alpha\), give
\begin{equation}
 \|R_a^\varepsilon\|_{L^{1,4\alpha}_{\mathbb F}(0,T;H)}
 +\|R_b^\varepsilon\|_{L^{2,4\alpha}_{\mathbb F}(0,T;H)}
 =O(\varepsilon).
\label{eq:singular-remainder-Lebesgue-data}
\end{equation}
Applying the multiple-driver estimate from
Remark~\ref{rem:multiple-FV-driver-reduction} and
\eqref{eq:unbounded-xi-unweighted-SEE-estimate} to
\(y^\varepsilon-y\), using
\eqref{eq:singular-remainder-Stieltjes-data} and
\eqref{eq:singular-remainder-Lebesgue-data}, gives
\[
\mathbb E\sup_{t\le T}\|y^\varepsilon(t)-y(t)\|_H^{2\alpha}
\le C_\alpha\varepsilon^{2\alpha}.
\]
\end{proof}

\subsection[Derivation of the pointwise singular-control conditions]{Derivation of the pointwise conditions \eqref{eq:SMP-singular-nonnegative} and
	\eqref{eq:SMP-singular-complementarity} from the singular variational inequality}
\label{subsec:singular-VI-pointwise-conditions}

The integrability and optional-projection properties verified in
Section~\ref{subsec:SMP} apply to \(\chi=\eta,\bar\xi\) and to
bounded increasing test processes. Thus
\eqref{eq:singular-variational-inequality} holds for the chosen
optional version of \(\mathcal G(t,\bar x(t-))\).

Let \(\zeta\) be any bounded adapted nondecreasing c\`adl\`ag process with
\(\zeta(0)=0\). For \(\lambda>0\), the process
\(\eta^\lambda:=\bar\xi+\lambda\zeta\)
is admissible, because
\[
 \mathfrak E_{m,L_\gamma}(\eta^\lambda)
 \le e^{mL_\gamma\lambda\|\zeta(T)\|_\infty}
      \mathfrak E_{m,L_\gamma}(\bar\xi)<\infty,
\]
and its polynomial moments are finite. Substituting \(\eta^\lambda\) into
\eqref{eq:singular-variational-inequality} and dividing by \(\lambda\) gives
\begin{equation}
    \mathbb E\int_{(0,T]}
    \mathcal G(t,\bar x(t-))\,d\zeta(t)\ge0.
    \label{eq:singular-positive-against-increasing-process}
\end{equation}

Let \(\tau\) be an \((0,T]\)-valued stopping time and
\(A\in\mathcal F_\tau\). Taking
\(\zeta(t)=I_A I_{\{\tau\le t\}}\)
in \eqref{eq:singular-positive-against-increasing-process} yields
$
    \mathbb E\!
    \left[I_A\mathcal G(\tau,\bar x(\tau-))\right]\ge0.
$
The value of the optional version at \(\tau\) is
\(\mathcal F_\tau\)-measurable.
Taking \(A=\{\mathcal G(\tau,\bar x(\tau-))<0\}\) therefore gives
\(\mathcal G(\tau,\bar x(\tau-))\ge0\) almost surely for every
\((0,T]\)-valued stopping time \(\tau\). 
By the optional section theorem
(see \cite[Theorem~IV.84]{DellacherieMeyer1978}),
\[
 \mathbb P\left\{
   \mathcal G(t,\bar x(t-))\ge0
   \text{ for all }t\in(0,T]
 \right\}=1.
\]
This is \eqref{eq:SMP-singular-nonnegative}.

Finally, taking \(\eta\equiv0\) in
\eqref{eq:singular-variational-inequality} gives
\begin{equation}
    \mathbb E\int_{(0,T]}
    \mathcal G(t,\bar x(t-))\,d\bar\xi(t)\le0.
    \label{eq:singular-contact-expectation}
\end{equation}
The integral in \eqref{eq:singular-contact-expectation} is
nonnegative by \eqref{eq:SMP-singular-nonnegative}; its expectation
is therefore zero, so it vanishes almost surely. Since the integrand
is nonnegative, this is equivalent to
\(\mathcal G(t,\bar x(t-))\,d\bar\xi(t)=0\) as a random measure,
proving \eqref{eq:SMP-singular-complementarity}.

\begingroup
\small
\setlength{\parskip}{0pt}

\endgroup

\end{document}